\documentclass[11pt]{article}
\usepackage[margin=1in]{geometry}

\usepackage{authblk} 
\usepackage{graphicx}
\usepackage{amssymb}

\usepackage{amsmath}
\usepackage{enumerate}
\usepackage{amsfonts}
\usepackage[mathcal]{euscript}
\usepackage{amsthm}
\usepackage[all,2cell,import]{xy} \UseAllTwocells
\usepackage{xcolor}
\usepackage{tkz-euclide}
\usepackage{tikz}

\usepackage{url}
\makeatletter
\def\url@leostyle{%
	\@ifundefined{selectfont}{\def\UrlFont{\sf}}{\def\UrlFont{\small\ttfamily}}}
\makeatother
\usepackage[refpage]{nomencl}

\makenomenclature
\usepackage{makeidx}
\makeindex
\numberwithin{equation}{section}

\usepackage[colorlinks=true, pdfstartview=FitV, linkcolor=black,
			   citecolor=black, urlcolor=black]{hyperref}
\usepackage{bookmark}
\usepackage[titletoc]{appendix} 
\hypersetup{hypertexnames=false} 

          \def\dudtau{\frac{\del u}{d\tau}}
          \def\dudt{\frac{\del u}{dt}}

          \newcommand{\nc}{\newcommand}
          \nc{\DMO}{\DeclareMathOperator}	
          \nc{\commentout}[1]{}

          \nc{\newnotation}{\nomenclature}
          \nc{\wrap}{\cW}
          \nc{\Tw}{\mathsf{Tw}}
          \nc{\loc}{\mathsf{Loc}}
          \nc{\Top}{Top}
          \nc{\emb}{\mathsf{emb}}
          \nc{\ind}{\mathsf{Ind}}
          \nc{\Ind}{\mathsf{Ind}}
          \nc{\Loc}{\mathsf{Loc}}
          \nc{\Cob}{\mathsf{Cob}}
          \nc{\mul}{\mathsf{Mul}}
          \nc{\fat}{\mathsf{fat}}
          \nc{\cob}{\mathsf{Cob}}
          \nc{\coh}{\mathsf{Coh}}
          \nc{\Liouaut}{\Aut}
          \nc{\Liouauto}{{\Aut^o}}
          \nc{\Liouautb}{\Aut^{b}}
          \nc{\Liouautgr}{\Aut^{gr}}
          \nc{\Liouautgrb}{\Aut^{gr,b}}
          \nc{\idem}{\mathsf{Idem}}
          \nc{\sets}{\mathsf{Sets}}
          \nc{\near}{\mathsf{near}}
          \nc{\sing}{\mathsf{Sing}}
          \nc{\Sing}{\mathsf{Sing}}
          \nc{\perf}{\mathsf{Perf}}
          \nc{\block}{\mathsf{block}}
          \nc{\ssets}{\mathsf{sSets}}
          \nc{\cmpct}{\mathsf{cmpct}}
          \nc{\compact}{\mathsf{cmpct}}
          \nc{\pwrap}{\mathsf{PWrap}}
          \nc{\coder}{\mathsf{Coder}}
          \nc{\bimod}{\mathsf{Bimod}}
          \nc{\grmod}{\mathsf{GrMod}}
          \nc{\Morita}{\mathsf{Morita}}
          \nc{\morita}{\mathsf{Morita}}
          \nc{\spaces}{\mathsf{Spaces}}
          \nc{\pwrms}{\mathsf{PWrFuk}_{M,S}}
          \nc{\pwrmf}{\mathsf{PWrFuk}_{M,F}}
          \nc{\pwrapmf}{\mathsf{PWrFuk}_{M,F}}
          \nc{\fuk}{\mathsf{Fukaya}}
          \nc{\infwr}{\mathsf{InfWr}}
          \nc{\fukaya}{\mathsf{Fukaya}}
          \nc{\autml}{\mathsf{Aut}_{M,\Lambda}}
          \nc{\fukml}{\mathsf{Fukaya}_{M,\Lambda}}
          \nc{\fukmle}{\mathsf{Fukaya}_{M,\Lambda,\epsilon}}
          \nc{\fukmod}{\wrfukcompact(M)\modules}
          \nc{\lag}{\mathsf{Lag}}
          \nc{\lagm}{\lag_M}
          \nc{\lago}{\lag^o}
          \nc{\lagml}{\lag_{M,\Lambda}} 
          \nc{\lagmle}{\lag_{M,\Lambda,\epsilon}}
          \nc{\Fun}{\mathsf{Fun}}
          \nc{\fun}{\mathsf{Fun}}
          \nc{\vect}{\mathsf{Vect}}
          \nc{\chain}{\mathsf{Chain}}
          \nc{\chainn}{Chain}
          \nc{\wrfuk}{\mathsf{WrFukaya}}
          \nc{\wrfukcompact}{\mathsf{WrFukaya}_{\mathsf{cmpct}}}
          \nc{\pwrfuk}{\mathsf{PWrFukaya}}
          \nc{\inffuk}{\mathsf{InfFuk}}
          \nc{\pwrfukml}{\mathsf{PWrFukaya}_{M,\Lambda}}
          \nc{\inffukml}{\mathsf{InfFuk}_{M,\Lambda}}
          \nc{\nattrans}{\mathsf{NatTrans}}
          \nc{\corres}{\mathsf{Corres}}
          \nc{\fukep}{\fukaya_\Lambda(M,\epsilon)}
          \nc{\fukepop}{\fukaya_\Lambda(M,\epsilon)^{\op}}
          \nc{\lagep}{\lag_\Lambda(M,\epsilon)}
          \DMO{\cyl}{cyl} 
          \nc{\dbcoh}{D^b\mathsf{Coh}}
          \nc{\corr}{\mathsf{Corr}}

          \nc{\cat}{\mathsf{Cat}}
          \nc{\Cat}{\mathsf{Cat}}
          \nc{\ainfty}{\mathsf{A}_\infty}
          \nc{\inftycat}{\mathcal{C}\!\operatorname{at}_\infty}
          \nc{\inftyCat}{\mathcal{C}\!\operatorname{at}_\infty}
          \nc{\inftyGpd}{\mathcal{G}\!\operatorname{pd}_\infty}
          \nc{\Ainftycat}{\mathcal{C}\!\operatorname{at}_{A_\infty}}
          \nc{\dgcat}{\mathcal{C}\!\operatorname{at}_{dg}}
          \nc{\ainftycat}{\mathcal{C}\!\operatorname{at}_{A_\infty}}
          \nc{\stablecat}{\mathcal{C}\!\operatorname{at}_\infty^{\Ex}}

          \DMO{\im}{im}
          \DMO{\ev}{ev}
          \DMO{\stable}{Ex}
          \DMO{\inj}{inj}
          \DMO{\fib}{fib}
          \DMO{\conf}{Conf}
          \DMO{\chains}{Chains}
          \DMO{\cochains}{Cochains}
          \DMO{\cone}{Cone}
          \DMO{\Map}{Map}
          \DMO{\ran}{Ran}
          \DMO{\rot}{Rot}
          \DMO{\leg}{Leg}
          \DMO{\imm}{imm}
          \DMO{\adj}{adj}
          \DMO{\symp}{Symp}
          \DMO{\tree}{Tree}
          \DMO{\cube}{Cube}
          \DMO{\deep}{deep}
          \DMO{\back}{back}
          \DMO{\Hoch}{Hoch}
          \DMO{\front}{front}
          \DMO{\flow}{Flow}
          \DMO{\floer}{Floer}
          \DMO{\Maps}{Maps}
          \DMO{\exact}{exact}
          \DMO{\excess}{Excess}
          \DMO{\Decomp}{Decomp}
          \DMO{\decomp}{Decomp}
          \DMO{\collar}{collar}
          \DMO{\yoneda}{Yoneda}
          \DMO{\hamspace}{Ham}
          \DMO{\sympspace}{Symp}
          \DMO{\holomaps}{Holomaps}
          \DMO{\comp}{Comp}
          \DMO{\crit}{Crit}
          \DMO{\test}{{test}}
          \DMO{\sign}{sign}
          \DMO{\topp}{top}
          \DMO{\indx}{Index}
          \DMO{\Break}{Break} 
          \DMO{\zero}{zero} 
          \DMO{\ob}{Ob}
          \DMO{\gr}{Gr} 
          \DMO{\Gr}{Gr} 
          \DMO{\cl}{Cl} 
          \DMO{\grlag}{GrLag}
          \DMO{\Pin}{Pin}
          \DMO{\Graph}{Graph}
          \DMO{\pin}{Pin}
          \DMO{\gap}{Gap}
          \DMO{\Ex}{Ex}
          \DMO{\id}{id}
          \DMO{\End}{End}
          \DMO{\sym}{Sym}
          \DMO{\aut}{Aut}
          \DMO{\Aut}{Aut}
          \DMO{\haut}{hAut}
          \DMO{\hAut}{hAut}
          \DMO{\DK}{DK} 
          \DMO{\poly}{poly} 
          \DMO{\diff}{Diff}
          \DMO{\coll}{coll}
          \DMO{\dist}{dist} 
          \DMO{\coker}{coker} 
          \nc{\kernel}{\ker} 
          \DMO{\sspan}{span}
          \DMO{\hocolim}{hocolim}	
          \DMO{\holim}{holim}
          \DMO{\sk}{sk}

          \DMO{\ho}{ho}
          \DMO{\fin}{fin}
          \DMO{\tor}{Tor}
          \DMO{\ext}{Ext}
          \DMO{\ret}{Ret}
          \DMO{\ham}{Ham}
          \DMO{\con}{con}
          \DMO{\leaf}{leaf}
          \DMO{\supp}{supp}
          \DMO{\edge}{edge}
          \DMO{\colim}{colim}
          \DMO{\edges}{edges}
          \DMO{\Image}{image}
          \DMO{\roots}{roots}
          \DMO{\height}{height}
          \DMO{\finmod}{FinMod}
          \DMO{\leaves}{leaves}
          \DMO{\planar}{planar}
          \DMO{\vertices}{vertices}

          \nc{\lagg}{\lag^{\cG}}
          \nc{\iso}{\mathsf{Iso}}
          \nc{\Set}{\mathsf{Set}}
          \nc{\Ass}{\mathsf{ \bf Ass}}
          \nc{\Mod}{\mathsf{Mod}}
          \nc{\modules}{\mathsf{Mod}}
          \nc{\sset}{\mathsf{sSet}}
          \nc{\liou}{\mathsf{Liou}}
          \nc{\poset}{\mathsf{Poset}}
          \nc{\trno}{T^*\RR^n_{\geq 0}}
          \nc{\spectra}{\mathsf{Spectra}}
          \nc{\tensorfin}{\tensor^{\fin}}
          \nc{\lagptg}{\lag_{pt,pt}^{\cG}}
          \nc{\Fin}{\mathcal{F}\mathsf{in}}
          \nc{\lagnl}{\lag_{N,\Lambda}}
          \nc{\lagmlg}{\lag_{M,\Lambda}^{\cG}}
          \nc{\lagsplit}{\lag^{\mathsf{split}}}
          \nc{\lagktimes}{(\lag^{\dd k})^\times}
          \nc{\lagplanar}{\lag^{\times,\planar}}
          \nc{\Cont}{\text{\rm Cont}}
          \nc{\Ham}{\text{\rm Ham}}
          \nc{\Dev}{\text{\rm Dev}}
          \nc{\Lin}{\text{\rm Lin}}
          \nc{\Int}{\text{\rm Int}}
          \nc{\Hom}{\text{\rm Hom}}
          \nc{\Chord}{\text{\rm Chord}}
          \nc{\nbhd}{\mathcal{N}\text{\rm{bhd}}}

          \nc{\onef}{1_{\fukaya}}
          \nc{\smsh}{\wedge}
          \nc{\un}{\underline}
          \nc{\xto}{\xrightarrow}
          \nc{\xra}{\xto}
          \nc{\tensor}{\otimes}
          \nc{\del}{\partial}
          \nc{\dd}{\diamond}
          \nc{\tri}{\triangle}
          \nc{\bb}{\Box}
          \nc{\into}{\hookrightarrow}
          \nc{\onto}{\twoheadrightarrow}
          \nc{\contains}{\supset}
          \nc{\transverse}{\pitchfork}
          \nc{\uncirc}{\underline{\circ}}

          \nc{\Jbar}{\overline{J}}
          \nc{\Fbar}{\overline{F}}
          \nc{\delbar}{\overline{\del}}
          \nc{\thetabar}{\overline{\theta}}
          \nc{\omegabar}{\overline{\omega}}
          \nc{\Liou}{\text{\rm Liou}}

          \nc{\Yhat}{\widehat{Y}}

          \nc{\trbar}{\overline{T^*\RR}}
          \nc{\tr}{T^*\RR}
          \nc{\tsa}{Ts\cA}
          \nc{\tsb}{Ts\cB}
          \nc{\cmbar}{\overline{\cM}}
          \nc{\crbar}{\overline{\cR}}

          \nc{\vece}{ {\vec \epsilon}}	
          \nc{\vecd}{ {\vec \delta}}
          \nc{\ov}{\overline}
          \DMO{\op}{op}
          \nc{\opp}{ ^{\op}}

          \nc{\hiro}{\textcolor{blue}}
          \nc{\YG}{\textcolor{orange}}

          \nc{\eqn}{\begin{equation}}
          \nc{\eqnn}{\begin{equation}\nonumber}
          \nc{\eqnd}{\end{equation}}
          \nc{\enum}{\begin{enumerate}}
          \nc{\enumd}{\end{enumerate}}
          \nc{\beastar}{\begin{eqnarray*}}
          \nc{\eeastar}{\end{eqnarray*}}

          \def\cA{\mathcal A}\def\cB{\mathcal B}\def\cC{\mathcal C}\def\cD{\mathcal D}
          \def\cF{\mathcal F}\def\cG{\mathcal G}\def\cH{\mathcal H}
          \def\cJ{\mathcal J}\def\cL{\mathcal L}
          \def\cM{\mathcal M}\def\cO{\mathcal O}\def\cP{\mathcal P}
          \def\cR{\mathcal R}\def\cS{\mathcal S}
          \def\cW{\mathcal W}\def\cX{\mathcal X}

          \def\BB{\mathbb B}\def\CC{\mathbb C}\def\DD{\mathbb D}

          \def\NN{\mathbb N}
          \def\QQ{\mathbb Q}\def\RR{\mathbb R}
          
          \def\ZZ{\mathbb Z}

          \nc{\Euc}{\mathsf{Euc}}
          \nc{\mfld}{\mathsf{Mfld}}
          \nc{\DTop}{\mathsf{DTop}}
          \nc{\simp}{\mathsf{Simp}}
          \nc{\Ainftycatt}{A_\infty Cat}
          \nc{\dgcatt}{dg Cat}
          \nc{\StableCat}{StableCat}
          \nc{\subdivision}{\mathsf{subdiv}}
          \nc{\Kan}{\mathcal{K}\mathsf{an}}

          \def\dudtau{\frac{\del u}{\del\tau}}
          \def\dudt{\frac{\del u}{\del t}}

          \theoremstyle{definition}
          \newtheorem{theorem}{Theorem}[subsection]
          \newtheorem{proposition}[theorem]{Proposition}
          \newtheorem{prop}[theorem]{Proposition}
          \newtheorem{lemma}[theorem]{Lemma}
          \newtheorem{warning}[theorem]{Warning}
          \newtheorem{cor}[theorem]{Corollary}
          \newtheorem{corollary}[theorem]{Corollary}
          \newtheorem{construction}[theorem]{Construction}
          \newtheorem{conjecture}[theorem]{Conjecture}
          \newtheorem{definition}[theorem]{Definition}
          \newtheorem{defn}[theorem]{Definition}
          \newtheorem{notation}[theorem]{Notation}
          \newtheorem{example}[theorem]{Example}
          \newtheorem{choice}[theorem]{Choice}
          
          \newtheorem{recollection}[theorem]{Recollection}
          \newtheorem{remark}[theorem]{Remark}
          \newtheorem{figurelabel}[theorem]{Figure} 
          \newtheorem{convention}[theorem]{Convention}
          \newtheorem{sublem}[theorem]{Sublemma}

\title{Continuous and coherent actions on wrapped Fukaya categories}

\author[$\dagger$]{Yong-Geun Oh}
\author[$\star$]{Hiro Lee Tanaka}

\affil[$\dagger$]{Center for Geometry and Physics (IBS), Pohang, Korea \&
Department of Mathematics, POSTECH, Pohang Korea.}
\affil[$\star$]{Department of Mathematics, Texas State University}

\begin{document}

\maketitle

\begin{abstract}
We establish the continuous functoriality of wrapped Fukaya categories with respect to Liouville automorphisms, yielding a way to probe the homotopy type of the automorphism group of a Liouville sector. These methods prove Liouville and monotone cases of a conjecture of Teleman from the 2014 ICM. In the case of a cotangent bundle, we show that the Abouzaid equivalence between the wrapped category and the $\infty$-category of local systems intertwines our action with the action of diffeomorphisms of the zero section. In particular, our methods yield a typically non-trivial map from the rational homotopy groups of Liouville automorphisms to the rational string topology algebra of the zero section.
\end{abstract}

\tableofcontents

\clearpage

\section{Introduction}
Liouville manifolds, and Liouville sectors more generally, are fruitful objects of study in symplectic geometry. They give rise to interesting  geometric questions whose solutions are amenable to flexible, and often {\em topological}, techniques. However, there have been fundamental questions about Liouville sectors whose answers are wanting in the literature.

One question concerns the study of automorphisms of Liouville sectors. For compact symplectic manifolds the automorphism groups $\symp$ and $\ham$ enjoy enticing formal properties---$\ham$ is simple, for example, and its group isomorphism class is a complete invariant of a compact symplectic manifold. Our knowledge of these groups is quite sparse in most examples. The same is true concerning automorphisms of Liouville sectors, though it should be mentioned that the community has identified non-trivial collections of automorphisms in many examples (often by manipulating Dehn twists about Lagrangian spheres), and progress has been made in studying spaces of {\em compactly supported} symplectic automorphisms of various Liouville manifolds. See for example~\cite{keating-stabilizations, keating-free-subgroups, dimitroglou-evans-exotic-spheres, evans-mapping-class-stein}.

A second question concerns what continuous dependence is enjoyed by wrapped Fukaya categories of Liouville sectors.
While it is known that stop removal and open inclusions of Liouville sectors induce functors on their wrapped Fukaya categories~\cite{gps,gps-2}, it has remained unclear to what extent wrapped categories display a dependence on families of embeddings---that is, on isotopies. When Liouville sectors are stabilized, results of~\cite{last-stable} prove that stabilized embedding spaces act coherently on wrapped Fukaya categories.

We prove here that such dependencies exist with as much coherence as one could hope for. This of course broadens the use of Floer-theoretic invariants in studying automorphism groups.

\begin{theorem}\label{theorem. main theorem informal}
Let $M$ be a Liouville sector, and let $\Liouauto(M)$ denote the topological group of Liouville automorphisms (Definition~\ref{defn:Liouaut}). Then $\Liouauto(M)$ acts coherently on the wrapped Fukaya category of $M$. That is, there exists an $A_\infty$ homomorphism
	\eqnn
	\Liouauto(M) \to \Aut(\cW(M))
	\eqnd
from the space of Liouville automorphisms of $M$ to the space of automorphisms of the wrapped Fukaya category of $M$.
\end{theorem}

The word ``continuous'' in the title of this paper refers to the fact that the map in Theorem~\ref{theorem. main theorem informal} can be realized as a map of topological spaces---i.e., is continuous. So for example, it induces a map on homotopy groups (Corollary~\ref{cor. induced map on HH*}). The word ``coherent'' refers to the fact that the map in Theorem~\ref{theorem. main theorem informal} is a map of $A_\infty$-algebras in spaces---the coherence is exhibited in the compatibility of the higher $A_\infty$-homotopies.

Our methods apply equally well to the compact monotone setting (see Section~\ref{subsection. hamiltonian setting}). Thus, the present work establishes (in the Liouville and monotone settings) a widespread expectation in the community: that well-behaved actions on $M$ induce continuous actions on the Fukaya category of $M$. For example, this is stated as a conjecture in Teleman's 2014 ICM talk~\cite[Conjecture 2.9]{teleman-icm}. One key departure from Teleman's proof proposal is that we do not employ Lagrangian correspondences. Teleman also sought to utilize a theorem equating group actions on dg-categories with certain maps of $E_2$-algebras; in our construction, the map of $E_2$-algebras is exhibited {\em after} exhibiting the group action. (See Corollary~\ref{cor. E-2 map} and Section~\ref{subsection. hamiltonian setting}.) Moreover, our proof avoids a great deal of Floer-theoretic set-up by utilizing the categorical technique of localizations, twice. (See Section~\ref{section. proof outline}.)

Before we go on, we should be explicit about what decorations we put on $M$---this affects what kind of category $\cW(M)$ is. For example, Theorem~\ref{theorem. main theorem informal} as stated is only true for $\cW(M)$ being two-periodically graded and linear over $\ZZ/2\ZZ$. This is because a Liouville automorphism $\phi \in \Liouauto(M)$ does not ``know'' how to respect any choice of grading $gr$, nor of background class $b \in H^2(M;\ZZ/2\ZZ)$---and such data would be required to form a $\ZZ$-graded and $\ZZ$-linear Fukaya category. One should thus also consider the natural automorphism group consisting of Liouville automorphisms equipped with data respecting these decorations (Section~\ref{section.automorphisms with decorations}).

\begin{notation}\label{notation. liouautgrb}
We will write
	\eqnn
	\Liouauto(M), \qquad \Liouautgrb(M),\qquad \Liouaut(M)
	\eqnd
to respectively denote the group of Liouville automorphisms, of Liouville automorphisms with data respecting a chosen $gr$ and $b$, and of Liouville automorphisms with data respecting some unspecified decorative choices.

Any statement about the non-superscripted $\Liouaut$ is true for all possible decorations, while a statement about $\Liouautgrb$ is only true for the wrapped category of a Liouville sector $M$ equipped with a grading and background class, and a statement about $\Liouauto(M)$ is true for the undecorated case.

We will write $\cW(M)$ for the wrapped category, suppressing dependence on decorations.
\end{notation}

Then our methods also yield:

\begin{theorem}\label{theorem. main theorem informal decorated}
For any Liouville sector $M$ equipped with decorations, there exists an $A_\infty$ homomorphism
	\eqnn
	\Liouaut(M) \to \Aut(\cW(M))
	\eqnd
from the space of Liouville automorphisms of $M$ (equipped with data respecting chosen decorations) to the space of automorphisms of the wrapped Fukaya category of $M$.
\end{theorem}

In Theorem~\ref{theorem. main theorem informal decorated}, the wrapped category on the right depends on the decorations, but we suppress this dependence from the notation (following Notation~\ref{notation. liouautgrb}).

\begin{remark}
Often, this decorated automorphism group is not too homotopically different from $\Liouauto(M)$---for example, when the decorations are a grading $gr$ and a background class $b \in H^2(M;\ZZ/2\ZZ)$, the map $\Liouautgrb(M) \to \Liouauto(M)$ induces an isomorphism on $\pi_{\geq 3}$ (Proposition~\ref{prop. Liouautgrb homotopy groups}).
\end{remark}

Now take based loop spaces and apply Dunn additivity to obtain:

\begin{cor}\label{cor. E-2 map}
There exists a map of $E_2$-algebras
	\eqnn
	\Omega \Liouaut(M) \to \Omega \aut(\cW(M)).
	\eqnd
\end{cor}

\begin{remark}
By generalities concerning $A_\infty$-categories, one can identify the endomorphisms of the identity functor of any $A_\infty$-category $\cC$ with the (topological space associated to the) non-positive truncation of the Hochschild cochains of $\cC$.\footnote{Given any cochain complex $C$, the non-positive truncation $\tau^{\leq 0} C$ is the cochain complex whose $i$th group is given by $C^i$ if $i < 0$, by $0$ if $i >0$, and by $\ker d^0$ if $i = 0$. The space associated to a non-positive cochain complex is constructed via the Dold-Kan correspondence.} See for example Section~1.3.3 of~\cite{tanaka-Aoo-units}. The based loop space at the identity functor $\id_{\cC}$ is precisely the space of those invertible endomorphisms---i.e., natural {\em equivalences} of the identity functor. Put another way, the based loop space $\Omega \aut(\cW(M))$ is the space of units of the Hochschild cochain algebra. The $E_2$-algebra structure of $\Omega \Aut(\cW(M))$ coincides with the $E_2$-algebra structure inherited from the Hochschild cochains of $\cW(M)$.
\end{remark}

By taking homotopy groups of domain and target, we have the following:

\begin{cor}\label{cor. induced map on HH*}
The above maps induce group homomorphisms
	\eqnn
	\pi_{k+1} \Liouaut(M)
	\to
	\begin{cases}
	HH^0(\cW(M))^\times & k = 0 \\
	HH^{-k}(\cW(M)) & k \geq 1.
	\end{cases}
	\eqnd
That is, the homotopy groups of $\Liouaut(M)$ map to the Hochschild cohomology groups of the wrapped Fukaya category.
\end{cor}

\begin{remark}
Here, $HH^0(\cW(M))^\times$ indicates the units of the multiplication of degree 0 Hochschild cohomology elements, and the homomorphism is with respect to this multiplication in degree 0. For $k \geq 1$, the group structure on the target is additive. That the homotopy groups of $\Aut(\cW(M))$ may be identified with certain Hochschild cohomology groups is an $A_\infty$-categorical version of  a result of T\"oen~\cite{toen-homotopy-theory-of-dg-cats}; see Corollary~1.12 of~\cite{tanaka-Aoo-units}.
\end{remark}

\begin{remark}\label{remark. seidel rep generalization}
To put the above corollary in context, recall that the Hochschild cohomology of the Fukaya category of $M$ is isomorphic to the quantum cohomology ring of $M$ in various contexts:
when $M$ is monotone~\cite{sheridan-fano}, and when $M$ is toric~\cite{afooo}, for example.
So one can view the above results as an analogue of the Seidel homomorphism $\pi_1 \ham(M) \to QH(M)^\times$~\cite{seidel-representation}, but generalized to higher homotopy groups of the appropriate analogue of $\ham$ in the Liouville setting. Indeed, one of our main geometric constructions is inspired by work of Savelyev who generalized the Seidel homomorphism to higher homotopy groups of $\ham$ in the monotone setting~\cite{savelyev}. We also conjecture a connection with the work of Lekili and Evans~\cite{lekili-evans}; see Conjecture~\ref{conjecture. LE equivalent to our map} below.
\end{remark}

Moreover, if a Lie group $G$ acts smoothly on $M$ by Liouville automorphisms respecting relevant decorations, we have an induced homomorphism $G \to \Liouaut(M)$, hence an $E_2$-algebra map $\Omega G \to \Omega \Aut (\cW(M))$. By noting that this map factors through the Hurewicz map for $\Omega G$ by adjunction (Notation~\ref{notation. dg nerve and adjoint}), we conclude that the map $C_* \Omega G \to \Hoch^*(\cW(M))$ to the Hochschild cochain complex is a map of $E_2$-algebras in chain complexes. Thus

\begin{cor}
Let $G$ act smoothly on $M$ by Liouville automorphisms respecting relevant structures. Then the action induces a map of Gerstenhaber algebras
	\eqnn
	H_*(\Omega G)
	\to
	HH^*(\cW(M)).
	\eqnd
\end{cor}

We suspect that this $E_2$-map is precisely the sought-after $E_2$-map in Teleman's ICM address; see for example Theorem 2.5 and the surrounding discussion in~\cite{teleman-icm}, and see also Conjecture~\ref{conjecture. LE equivalent to our map} below.

The latter half of this work applies the above results to the special case when $M = T^*Q$ is a cotangent bundle. We fix an orientation on $Q$ and we choose $b \in H^2(M;\ZZ/2\ZZ)$ to be pulled back from the second Stiefel-Whitney class of $Q$. We also fix the canonical grading on $T^*Q$.

Then a result of Abouzaid~\cite{abouzaid-loops} states that the resulting wrapped Fukaya category $\cW(T^*Q)$ is equivalent to the $A_\infty$-category $\loc(Q)$ of local systems on the zero section $Q$. By abstract non-sense, $\loc(Q)$ is equivalent to $\Mod(C_*(\Omega Q))$, modules over the $A_\infty$-algebra of chains on the based loop space. We thus have a quasi-isomorphism of chain complexes from the Hochschild chain complex to chains on the free loop space:
	\eqnn
    C^{\Hoch}_*(\cW(T^*Q)) \simeq
    C^{\Hoch}_*(\Mod(C_*\Omega Q)) \simeq  C_*(\cL Q).
    \eqnd
Here, the second equivalence is due to Goodwillie~\cite{goodwillie-cyclic-homology} and Burghelea-Fiedorowicz~\cite{burghelea-fiedorowicz-ii}. Finally, by a generalization of Poincar\'e duality, a choice of fundamental class on $Q$ results in an isomorphism of graded commutative algebras\footnote{See for example Malm~\cite{malm-thesis-arxiv}; this is in fact an isomorphism of BV algebras with the string topology operations on the right-hand side, and the natural BV structure on Hochschild cohomology on the left-hand side.}
	\eqnn
    HH^*(C_*\Omega Q) \simeq H_{\ast + \dim Q}(C_*\cL Q)
    \eqnd
from the Hochschild cohomology groups of the $A_\infty$-algebra $C_* \Omega Q$ to the shifted homology of the free loop space.
We thus have the following:

\begin{cor}
Let $Q$ be compact and orientable, and choose an orientation of $Q$. The above maps induce group homomorphisms
	\eqn\label{eqn. ham to LQ}
	\pi_{k+1} \Liouautgrb(T^*Q)
    \to
    \begin{cases}
    H_{\dim Q}(\cL Q)^\times & k = 0 \\
    H_{k + \dim Q}(\cL Q) & k \geq 1
    \end{cases}
	\eqnd
That is, the homotopy groups of $\Liouautgrb(T^*Q)$ map to the homology of the free loop space of the zero section.
\end{cor}

(As before, when $k = 0$, the group structure on the target is multiplicative with respect to the string multiplication on $C_*\cL Q$, and when $k \geq 1$, the group structure is the additive one on homology.)

Note that because $\Liouautgrb \to \Liouauto$ induces an isomorphism on homotopy groups in degrees $\geq 3$ (Proposition~\ref{prop. Liouautgrb homotopy groups}), the corollary also allows us to map these higher homotopy groups of $\Liouauto$ to the homology of the free loop space of $Q$.

To state our final results, we note that there is a natural action of the diffeomorphism group $\diff(Q)$ on $\loc(Q)$. On the other hand, any diffeomorphism induces a Liouville automorphism of $T^*Q$. This induced map can be lifted to naturally respect the choices of $b$ and $gr$, so we have a map $\diff(Q) \to \Liouautgrb(T^*Q)$.

    \begin{theorem}\label{theorem. diff and ham compatible}
    The diagram
    	\eqnn
        \xymatrix{
        \diff(Q) \ar[r] \ar[d] & \Liouautgrb(T^*Q) \ar[d]^{\text{Thm}~\ref{theorem. main theorem informal}} \\
        \aut( \loc(Q)) \ar[r]^{\sim} & \aut( \cW(T^*Q) )
        }
        \eqnd
    commutes up to homotopy. Here, the bottom arrow is induced by the Abouzaid equivalence~\cite{abouzaid-loops} and the left vertical arrow is induced by the action of $\diff(Q)$ on $\loc(Q)$.
    \end{theorem}

Informally, Theorem~\ref{theorem. diff and ham compatible} states that the Abouzaid equivalence $\cW(T^*Q) \simeq \loc(Q)$ can be made equivariant with respect to the natural $\diff(Q)$ action on both categories, and that this action factors through our construction from Theorem~\ref{theorem. main theorem informal decorated}.

\begin{corollary}\label{corollary. survival}
Let $[\alpha] \in \pi_{k+1}\diff(Q)$ and suppose the image of $[\alpha]$ is non-zero under the map $\pi_{k+1}\diff(Q) \to HH^{-k}(C_*\Omega Q )$ from~\eqref{eqn. ham to LQ}. Then the image of $[\alpha]$ is non-zero in $\pi_{k+1}\Liouautgrb(T^*Q).$ In particular, such $[\alpha]$ detect non-trivial elements in the homotopy groups of $\Liouautgrb(T^*Q)$.
\end{corollary}

Now let
	\eqnn
	\haut(Q)
	\eqnd
denote the space of those continuous maps $f: Q \to Q$ that happen to be homotopy equivalences.
By taking based loops, and noting that the action of $\diff$ on $\loc$ factors through $\haut$, we have the following:

\begin{cor}\label{corollary. diff haut}
The following diagram commutes:
	\eqnn
	\xymatrix{
	\pi_{k+1}\diff(Q) \ar[r] \ar[d] & \pi_{k+1}\Liouautgrb(T^*Q) \ar[d] \\
	\pi_{k+1}\haut(Q) \ar[r] & H_{\dim Q + k} (\cL Q).
	}
	\eqnd
\end{cor}

\begin{example}
Let $Q = S^1$ and take $R=\ZZ$. We must contemplate the composite
	\eqnn
    \ZZ \cong \pi_0 \Omega \diff(S^1) \simeq \pi_0 \Omega \haut(S^1) \to (HH^0(C_* \Omega S^1))^{\times} \cong \ZZ \times \ZZ/2\ZZ.
    \eqnd
To understand the last map, we note the equivalence $C_* \Omega S^1 \simeq \ZZ[x^{\pm 1}]$ as an $A_\infty$ algebra. The 0th Hochschild cohomology is (the center of) this ring, and its units are the monomials of the form $\pm x^i$ for $i \in \ZZ$; the homomorphism $\ZZ \cong \pi_1 \haut(S^1) \to (HH^0)^\times$ is given by $i \mapsto x^i$.

In particular, Theorem~\ref{theorem. diff and ham compatible} shows that $\pi_1 \diff(S^1) \to \pi_1 \Liouautgrb(T^* S^1)$ is an injection. While less is known about the space of diffeomorphisms of tori $T^n$, we regardless produce non-trivial elements in $\pi_1\Liouautgrb(T^* T^n)$ straightforwardly using the same method.
\end{example}

As pointed out to us by Sylvan, because both the inclusion $Q \to T^*Q$ and the projection $T^*Q \to Q$ are homotopy equivalences, we have a factorization
	\eqn\label{eqn. obvious factorization}
	\diff(Q) \to \Liouauto(T^*Q) \to \haut(Q).
	\eqnd
(It does, however, take a little bit of effort to render the second map in~\eqref{eqn. obvious factorization} a map of of $A_\infty$ spaces.) It is thus natural to try to detect homotopy groups of $\Liouauto(T^*Q)$ by identifying elements that survive the map $\diff(Q) \to \haut(Q)$. For example, a result of Felix-Thomas~\cite{felix-thomas} states that when $Q$ is simply-connected, the map
	\eqnn
	\pi_{k+1}\haut(Q) \tensor \QQ \to H_{\dim Q + k}(\cL Q; \QQ)
	\eqnd
is an injection. So we can immediately deduce non-trivial elements in the rational homotopy groups of $\Liouauto(T^*Q)$ if we can detect non-trivial elements of $\pi_\bullet(\diff(Q)) \tensor \QQ$ that survive in the string algebra. Indeed, this gives a shorter proof of Corollary~\ref{corollary. survival}.

\subsection{Future directions}

\subsubsection{Spectral versions} \label{section. hurewicz factorization}
When $Q$ is compact and oriented, the map $\Omega \haut(Q) \to C_*(\cL Q)[-n]$ (with homological shifting conventions) has another description. At the level of homotopy/homology groups, it is equal to the composition
	\eqnn
	\pi_\ast \Omega \haut(Q)
	\to
	H_\ast \Omega \haut(Q)
	\xra{[Q]}
	H_\ast \Omega\haut(Q) \tensor H_n(Q)
	\xra{\ev}
	H_{\ast + n}(\cL Q).
	\eqnd
The first arrow is the Hurewicz map, the next is tensoring with a choice of fundamental class, and the last is induced by the evaluation map $\Omega \haut(Q) \times Q \to \cL Q$. The verification of this presentation is independent of any Floer theory. (See also~\cite{felix-thomas}.)

In other words, outside of $\pi_0$, the information our actions obtains about $\Liouaut$ factors through the Hurewicz map of $\haut$, and in fact of $\Liouaut$ as well. This is not satisfying, and at the very least, one would hope for a construction that factors through the stable homotopy type (which sees more about the topology of a space than its homology). This is another good reason to seek situations when we can define wrapped Fukaya categories over the sphere spectrum.

\subsubsection{Relation to Teleman and Lekili-Evans}\label{subsection. hamiltonian setting}
Let us change settings for a bit and consider a compact monotone $X$. The methods of the present work carry over straightforwardly to this setting---in fact, the cumbersome need to verify some of our $C^0$ estimates vanishes. We would like to emphasize that the essential geometry of such a construction was established in the work of Savelyev~\cite{savelyev} and is not due to us. The methods there result in a functor from $\simp(B\Ham(X))$ to $\Ainftycat$ (see Section~\ref{section. proof outline}), and by realizing the homotopy type of $B\Ham(X)$ as a localization of $\simp(B\Ham(X))$, we have the following analogue of Theorem~\ref{theorem. main theorem informal}:

\begin{theorem}\label{theorem. hamiltonian version}
One has a map of $A_\infty$ algebras
	\eqnn
	\Ham(X) \to \Aut(\fukaya(X)).
	\eqnd
\end{theorem}

Then any Lie group $G$ with a Hamiltonian action on $X$ enjoys a continuous homomorphism $G \to \Ham(X)$. Composing with the theorem above, we have:

\begin{cor}
A Hamiltonian action of a Lie group $G$ on a monotone $X$ induces a map of $A_\infty$-algebras $G \to \Aut(\fukaya(X))$.
\end{cor}

The above proves~\cite[Conjecture 2.9]{teleman-icm} in the monotone setting. The next result establishes (in the monotone setting) a conjectural $E_2$-algebra map sought in Theorem~2.5 of loc. cit., though our methods do not utilize Lagrangian correspondences.

\begin{cor}\label{cor. Ham E2 map}
Any Hamiltonian action of $G$ on a monotone $X$ induces a map
	\eqn\label{eqn. LE our version}
	C_*(\Omega G) \to \Hoch^{*}(\fukaya(X)).
	\eqnd
of $E_2$-algebras.
\end{cor}

\begin{proof}
By taking based loop spaces, we have a map of $E_2$-algebras
	\eqnn
	\Omega G \to \Omega \Aut(\fukaya(X)).
	\eqnd
By passing to chain complexes, we obtain an $E_2$-algebra map. Let us explain this a little further, as it is formal based on $\infty$-categorical arguments: The passage from spaces to spectra (otherwise known as $\Sigma^\infty$) is lax symmetric monoidal, and so is the passage from spectra to chain complexes (otherwise known as tensoring/smashing with $H\ZZ$).
\end{proof}

Of course, it is unsatisfactory not to compare the methods of the present work to the methods proposed in other works. To this end, recall that Lekili and Evans~\cite{lekili-evans} construct a map
	\eqn\label{eqn. LE map}
	\hom_{\cW(T^*G)}(T^*_e G, T^*_e G)
	\to
	\hom_{\fuk(X^- \times X)}(\Delta_X,\Delta_X).
	\eqnd
The construction is as follows: One notes that the graph of the Hamiltonian action $\Gamma \subset G \times X \to X$ quantizes to a Lagrangian $C \subset T^*G \times X^- \times X$, essentially by lifting $(g, x, g(x))$ to the value of the moment map at $g(x)$. A famous paradigm expects that Lagrangian correspondences give rise to functors and bimodules; Lekili and Evans realize this paradigm in this example by constructing a functor from (the full subcategory consisting of the cotangent fiber of) the wrapped Fukaya category of $T^*G$ to the Fukaya category of $X^- \times X$. It is trivial to verify that $C$ geometrically composes with the cotangent fiber $T^*_e G$ at the identity to yield the diagonal $\Delta_X \subset X^- \times X$; the map \eqref{eqn. LE map} is the induced map of $A_\infty$-algebras encoding this functor.

Importantly, taking cohomology on both sides of~\eqref{eqn. LE map}, we obtain a map
	\eqnn
	H_*(\Omega G) \to HH^*(\fukaya(X))
	\eqnd
from the homology of the based loop space of $G$ to the Hochschild cohomology of the Fukaya category of $X$. \footnote{Here, we are identifying the endomorphisms of the diagonal with Hochschild cohomology of the Fukaya category for monotone symplectic manifolds---this equivalence usually passes through quantum cohomology of $X$. See also Remark~\ref{remark. seidel rep generalization}.} We have seen a map like this before! We conjecture the following:

\begin{conjecture}\label{conjecture. LE equivalent to our map}
As $A_\infty$ algebra maps, the map~\eqref{eqn. LE our version} is homotopic to the map~\eqref{eqn. LE map}.
\end{conjecture}

That is, a map constructed from the geometry of correspondences and quilts in~\cite{lekili-evans} is equivalent to a map constructed using the geometry of bundles and the algebra of localizations here.

We warn the reader that the codomains of the two maps must be identified; part of the conjecture is that the equivalence is intertwined by a natural chain-level, $A_\infty$-algebra equivalence between Hochschild cochains and endomorphisms of the diagonal.

In particular, a proof of Conjecture~\ref{conjecture. LE equivalent to our map} would prove (by Corollary~\ref{cor. Ham E2 map}) that the correspondences-style construction~\eqref{eqn. LE map} can be lifted to an $E_2$-algebra map.

\subsubsection{Other future directions}

We enumerate natural avenues of pursuit through a sequence of remarks:

\begin{remark}[Other notions of automorphisms]
While we have taken $\Liouaut(M)$ and $\Liouauto(M)$ to be natural notions of automorphism group of a Liouville sector $M$, it is not clear that these are the only such natural choice (even up to homotopy equivalence). In another direction, one could contemplate the automorphism group of the skeleton of a Liouville manifold; the correct notion of automorphism  presumably depends both on the stratified homotopy type of the skeleton, and on some tubular or infinitesimal differential-geometric data attached to the skeleton (to be able to recover the equivalence type of its wrapped Fukaya category, for example).
\end{remark}

\begin{remark}[Filtered enhancements]
When $M=T^*Q$, all our computations ``factor'' through the homotopy type of $M \simeq Q$; this is unsurprising given that the wrapped Fukaya category of $T^*Q$ depends only on the homotopy type of $Q$ and our techniques incurably pass through wrapped Floer constructions. However, we suspect that by filtering automorphism groups in a way compatible with the action filtration on Floer complexes, one can glean far richer symplectic data.
\end{remark}

\begin{remark}[Is symplectic geometry helping differential topology, or vice versa?]
Even for the case of $M= T^*Q$ with $Q$ compact and oriented, we note that the factorization $\pi_\ast \diff(Q) \to \pi_\ast \haut(Q) \to H_{\ast + \dim Q}(\cL Q)$ from Corollary~\ref{corollary. diff haut} is difficult to study using tools of homotopy theory. It is not yet clear in which direction information will naturally flow in the future: Whether the higher homotopy groups of $\Liouaut$ are amenable to Floer-theoretic techniques (and hence yield information about the relation between $\diff$ and $\haut$), or whether the homotopy theory will yield information about Liouville automorphisms. Of course, in the near future---because symplectic techniques are newer---we would expect to see information from Floer theory provide new information about diffeomorphism spaces.
\end{remark}

\begin{remark}\label{remark. genus g computations}
At present, the relation between $\diff$ and $\haut$ is best understood rationally, and the best-understood examples are the even-dimensional ``genus $g$'' manifolds $(S^n \times S^n)^{\# g}$ obtained by taking the connect sum of $S^n \times S^n$ $g$ times. Even better understood are the manifolds obtained by removing a small open disk from $(S^n \times S^n)^{\# g}$. The obvious low-hanging fruit is to detect rational homotopy groups of $\Liouaut$ of the cotangent bundles of these manifolds (with boundary).
\end{remark}

\begin{remark}
Let us state two further reasons we take an interest in the continuous functoriality of Fukaya categories. First, and independently of any symplectic considerations, many useful algebraic invariants now have concrete geometric interpretations. The study of symplectic geometry has benefitted from the congruence of the geometry of symplectic phenomena with the geometry of algebraic phenomena. For example, the circle action on framed $E_2$-algebras, known as the BV operator in characteristic 0, is visible in the Fukaya categories having Calabi-Yau like properties simply by geometric rotation of Reeb orbits. Another example is the ability to encode the paracyclic structure of the s-dot construction of Fukaya categories using configurations of points on the boundary of a disk~\cite{tanaka-paracyclic}. Our Corollary~\ref{cor. E-2 map} follows this storyline.

Second, the moduli space of embeddings $M \to M'$ also encodes a part of the moduli space of Lagrangian correspondences, or more generally of bimodules between Fukaya categories. Via mirror symmetry, understanding this moduli space on the A model is expected to yield insights about the moduli space of bimodules between derived categories of sheaves on the B model.
\end{remark}

\subsection{Outline of proofs}\label{section. proof outline}
We make heavy use of $\infty$-categorical machinery. We refer the reader to \cite{oh-tanaka-localizations} and~\cite{tanaka-Aoo-units} for background and Section~\ref{section. A oo localizations} for the results we utilize.

\begin{choice}\label{choice. base ring R}
We fix a base ring $R$. We let
	\eqnn
	\Ainftycat
	\eqnd
denote the $\infty$-category of $A_\infty$-categories over $R$. (See Notation~1.4.4 of~\cite{tanaka-Aoo-units}.) Informally, its objects are $R$-linear $A_\infty$-categories, and two objects $\cA$ and $\cB$ enjoy a {\em space} of morphisms between them. When $\cA$ is sufficiently cofibrant (which is the case for when $\cA$ is a wrapped Fukaya category) this space can be described as having vertices given by functors, edges given by homotopies of functors, and higher simplices given by homotopies between homotopies (Theorem~1.5 of~\cite{tanaka-Aoo-units}). In $\Ainftycat$, all $A_\infty$ equivalences are also invertible up to homotopy.

By choosing appropriate structures on $M$ and appropriate brane structures on our Lagrangians, we assume $\cW(M)$ is $R$-linear. (For example, with the usual relative Pin structures on our branes, one can take $R$ to be $\ZZ$. When gradings are chosen, we can take $\Ainftycat$ to consist of usual $A_\infty$-categories, while when we have no gradings, we must demand that our $A_\infty$-categories and functors are 2-periodic.)

We let $\aut(\cW(M))$ denote the space of $R$-linear automorphisms of $\cW(M)$.
\end{choice}

Given a group $G$, one can construct a category $\BB G$ as follows: $\BB G$ has a unique object, and the endomorphisms of this object are defined to be $G$. If $G$ has a topology, one can demand that the category ``remembers'' the topology on this space of endomorphisms---or, at least, remember the homotopy type of $G$. We call the resulting construction the {\em classifying $\infty$-category} of $G$. One can perform this construction even more generally when $G$ is not a group, but is an $A_\infty$-algebra in spaces.

We are interested in the case $G=\Liouaut(M)$, the space of Liouville automorphisms of $M$ respecting the choices made on $M$ in Choice~\ref{choice. base ring R}. (See Section~\ref{section.automorphisms with decorations} for examples.)
We let $\BB\Liouaut(M)$ denote the classifying $\infty$-category of $\Liouaut(M)$.

The following is a precise formulation of Theorem~\ref{theorem. main theorem informal} and Theorem~\ref{theorem. main theorem informal decorated}.

\begin{theorem}\label{theorem. main theorem}
There exists a functor
	\eqn\label{eqn. main theorem functor}
	\BB\Liouaut(M) \to \Ainftycat
	\eqnd
sending a distinguished base point of the domain to the wrapped Fukaya category $\cW(M)$.
\end{theorem}

A functor induces a map of morphisms spaces, and in particular, of endomorphism spaces. Because every morphism in the domain is invertible, the functor sends morphisms in the domain to equivalences in the codomain. Thus, we conclude:

\begin{cor}[Precise form of Theorem~\ref{theorem. main theorem informal} and Theorem~\ref{theorem. main theorem informal decorated}.]
One has a map
	\eqnn
	\Liouaut(M) \to \aut(\cW(M))
	\eqnd
of group-like $A_\infty$-spaces.
\end{cor}

\begin{remark}
``Group-like $A_\infty$-spaces'' may seem a mouthful.
The intuition is that the map $\Liouaut(M) \to \aut(\cW(M))$ is morally a continuous group homomorphism.
\end{remark}

To construct the functor~\eqref{eqn. main theorem functor} we use an algebraic version of the following tautology: The barycentric subdivision of a space is homotopy equivalent to the original space. This saves us an enormous amount of analytic legwork.

To explain this, let us denote by $B\Liouaut(M)$ the classifying space.\footnote{Note that font distinguishes the $\infty$-category $\BB\Liouaut(M)$ from the space $B\Liouaut(M)$.} Let $\simp(B\Liouaut(M))$ denote the following category (in the classical sense): An object is a continuous map
	\eqnn
	j: |\Delta^n| \to B\Liouaut(M)
	\eqnd
from an $n$-simplex to the classifying space. A morphism from $j$ to $j'$ is the data of an injective poset map $[n] \to [n']$ such that the induced composition $|\Delta^n| \to |\Delta^{n'}| \xra{j'} B\Liouaut(M)$ is equal to $j$. To see why we call this a model for a barycentric subdivision, we encourage the reader to fix some $j$ and enumerate all the objects admitting a map to $j$. We note that the only invertible morphisms are the identity morphisms.

Then the algebraic version of the above tautology is as follows: If we localize $\simp(B\Liouaut(M))$---i.e., if we invert all morphisms of $\simp(B\Liouaut(M))$---one obtains $\BB \Liouaut(M)$. (See Section~\ref{section. smooth approximation}.)
Thus, by the universal property of localizations, to construct a functor  $\BB\Liouaut(M) \to \Ainftycat$ as in \eqref{eqn. main theorem functor}, we need only construct a functor
	\eqn\label{eqn. cW}
	\cW: \simp(B\Liouaut(M)) \to \Ainftycat
	\eqnd
for which every morphism of the domain is sent to an equivalence of $A_\infty$-categories.

\begin{warning}\label{warning. smooth simplices}
Our geometric constructions require the composition $|\Delta^n| \to B\Liouaut(M) \to B\Liouauto(M)$ to be {\em smooth}, not just continuous; we ignore this for now. See Section~\ref{section. smooth approximation} for more.
\end{warning}

Let us outline the construction of $\cW_j$. There is a tautological $M$-bundle over $B\Liouaut(M)$.
Given a smooth map $j: |\Delta^n| \to B\Liouaut(M)$, we obtain a smooth $M$-bundle over $|\Delta^n|$ by pulling back along $j$.
We define an $A_\infty$-category denoted by
	\eqnn
	\cO_j
	\eqnd
whose objects are branes living in the fibers above the vertices of the $n$-simplex $|\Delta^n|$, and morphisms are defined by choosing parallel transports over the edges of the $n$-simplex and intersecting branes. The $A_\infty$-operations are defined by counting certain holomorphic maps into the $M$-bundle compatible with connections. We refer readers to \cite{oh-tanaka-liouville-bundles}, Section~\ref{section. continuation maps}, and Section~\ref{subsection. unwrapped} for details.

Then, given a smooth map $j: |\Delta^n| \to B\Liouaut(M)$,
	\eqnn
	\cW_j
	\eqnd
is defined by localizing $\cO_j$  along non-negative continuation maps. This localization idea adapts Abouzaid-Seidel's unpublished work (as utilized in~\cite{gps}) to the bundle setting.
There are some subtleties to point out:
\begin{itemize}
	\item We are forced to utilize continuation maps defined by counting holomorphic disks with {\em one} boundary puncture (not strips).
	\item That one can define a ``family'' wrapped Fukaya category over $|\Delta^n|$ for $n \neq 0$ relies on Savelyev's observation~\cite{savelyev} that one can operadically map the tautological family of holomorphic disks to $|\Delta^n|$. This also specifies which holomorphic maps we are counting to define the $A_\infty$ operations.
	\item As usual, one needs to choose auxiliary data (connections and almost-complex structures) to count holomorphic curves. Our definition of $\cO_j$ is inductive on the dimension of the domain $|\Delta^n|$ of $j$, and the inductive step requires us to extend data defined on $\del |\Delta^n|$ to the interior of $|\Delta^n|$. We use that $j$ maps to $B\Liouaut(M)$, and not $B\Liouauto(M)$, so that $j$ encodes already the homotopical information needed to extend auxiliary data to the interior.
	 \item To obtain the functor \eqref{eqn. cW}, one must then verify that inclusions of simplices define $A_\infty$-functors $\cO_j \to \cO_{j'}$, and that the induced maps $\cW_j \to \cW_{j'}$ are equivalences. This requires a way to compute morphisms in the wrapped categories, and we do so using methods analogous to~\cite{gps}: A ``sequential colimit'' construction of wrapped Floer cochains recovers the morphisms in $\cW_j$ (Lemma~\ref{lemma. hom is wrapped cohomology}).
\end{itemize}

This concludes the summary of the construction of the functor $\cW$ in~\eqref{eqn. cW}, and hence of the functor in Theorem~\ref{theorem. main theorem}.

\begin{remark}
Note that the construction involves two distinct localizations---one to pass from a family of non-wrapped ``directed'' Fukaya categories $\cO$ to a family of wrapped Fukaya categories $\cW$, and the other to pass from a category $\simp(B\Liouaut(M))$ to the classifying $\infty$-category $\BB(\Liouaut(M))$.

Both localizations are used to avoid difficult or tedious analytical constructions.
\end{remark}

Now we explain the proof of Theorem~\ref{theorem. diff and ham compatible} which, in the example of $M = T^*Q$, verifies that the above construction yields non-trivial actions.

Fix a smooth manifold $Q$.
$B\diff(Q)$ has a tautological $Q$-bundle over it, and we can pull the $Q$-bundle back along any simplex $j: |\Delta^n| \to B\diff(Q)$. By assigning to each $j$ the $A_\infty$-category of local systems on the pulled back bundle, we obtain a functor\footnote{The notation $\Tw C_* \cP$ is an artifact of a particular presentation of the $A_\infty$-category of local systems we utilize in Section~\ref{section. chains on CP}. This presentation turns out to play well with  the construction from~\cite{abouzaid-loops}, so we will lug around the notation for this pay-off.}
	\eqnn
	\Tw C_* \cP: \simp(B\diff(Q)) \to \Ainftycat.
	\eqnd
On the other hand, there is a natural inclusion $\diff(Q) \to \Liouautgrb(T^*Q)$, and this induces a map of their classifying spaces (and hence of their categories of simplices). Call this induced map $\DD$, and consider the composite
	\eqnn
	\simp(B\diff(Q)) \xra{\DD} \simp(B\Liouautgrb(T^*Q)) \xra{\cO} \Ainftycat.
	\eqnd
We thus have two functors $\cO \circ \DD$ and $\Tw C_* \cP$ from $\simp(B\diff(Q))$ to $\Ainftycat$. In Section~\ref{subsec:familyAbouzaid}, we construct a natural transformation from the latter to the former.
This natural transformation is a non-wrapped,
family-friendly version of the construction Abouzaid utilized in~\cite{abouzaid-loops}.

The next step is to show that the natural transformation $\cO \circ \DD \to \Tw C_* \cP$ factors through the localization $\cW \circ \DD$. This requires us to prove the following geometric theorem,
which we state in greater generality (beyond the case of the zero section of the cotangent bundle):

\begin{theorem}[Theorem~\ref{thm. continuation maps are equivalences of twisted complexes}]\label{thm. continuation is equivalence in Tw CP general}
Let $M$ be a Liouville sector, and $c: L_0 \to L_1$ a continuation element associated to a non-negative isotopy. We also fix a compact test brane $X \subset M$. Then the map of twisted complexes
	\eqnn
	c_* : (L_0 \cap X, D_0) \to (L_1 \cap X, D_1)
	\eqnd
---induced by the Abouzaid functor from $\cO(M)$ to $\Tw C_*\cP(X)$---is an equivalence.
\end{theorem}

By the universal property of localization, we conclude that the natural transformation $\cO \circ \DD \to \Tw C_* \cP$ induces a natural transformation
	\eqn\label{eqn. W to CP}
	\cW \circ \DD \to \Tw C_* \cP.
	\eqnd
The remaining key step is to prove that this natural transformation is in fact a natural equivalence. This is a family version of the Abouzaid equivalence, and we prove it in Theorem~\ref{theorem. cW to cP}. We utilize Abouzaid's original result in our proof, and in particular, we must relate the quadratic definition of the wrapped complex to the localization definition. This is accomplished through a combination of Proposition~\ref{prop:quad=cofinal} and a standard result in computing morphism complexes of localizations (Recollection~\ref{recollection. A oo facts}\eqref{item. filtered localization hom}).

As a consequence we have a diagram of $\infty$-categories
	\eqnn
	\xymatrix{
	\BB\diff(Q) \ar[rr] \ar[dr]_-{\Tw C_* \cP} && \BB\Liouautgrb(T^*Q) \ar[dl]^{\cW}\\
	& \Ainftycat.
	}
	\eqnd
The homotopy commutativity of the diagram is exhibited by the natural equivalence \eqref{eqn. W to CP}.   Note that the left-hand map sends a distinguished object of $\BB\diff(Q)$ to $\Tw C_* \cP$.

The final step in proving Theorem~\ref{theorem. diff and ham compatible} is taken in Proposition~\ref{prop. P to Loc}. There, we show that this left-hand functor is naturally equivalent to the functor exhibiting the natural $\diff(Q)$ action on $\loc(Q) \simeq \Tw C_* \cP$. Then, because the right-hand functor sends a vertex to $\cW(Q)$, Theorem~\ref{theorem. diff and ham compatible} follows by taking based loops of each $\infty$-category in the above commutative triangle. (Equivalently, by recording the effect that the functors have on endomorphism spaces.)

\begin{remark}
A {\em consequence} of Theorem~\ref{theorem. cW to cP} is that the localization-style definition of the wrapped Fukaya category of a cotangent bundle is equivalent to Abouzaid's quadratic definition in~\cite{abouzaid-loops}, but this equivalence is not proven by writing an explicit, analytically defined functor from one to the other. In the present work, it is rather obtained by comparing both $A_\infty$-categories to the $A_\infty$-category of local systems, and even relies on the (independent) generation results of Abouzaid and of Ganatra-Pardon-Shende~\cite{gps-2}. Our proof that $\cW \simeq \Tw C_* \cP$ differs from that of~\cite{gps-microlocal} in that we compute the equivalence of endomorphism algebras in a way that passes through and relies on Abouzaid's construction, while~\cite{gps-microlocal} does not.

Of course, in principle one need not pass through the local system category; for more general Liouville sectors, it is possible to write down a comparison functor from the cofinally wrapped category to the quadratically wrapped category directly, but we do not do this here.
\end{remark}

\subsection{Conventions}
Here we make explicit our Fukaya-categorical conventions.

\begin{remark}
We follow the conventions of almost all the literature concerning wrapped categories, with one notable exception. The work~\cite{gps} reads boundary branes {\em clockwise} with respect to the boundary of a disk, while we use the more standard counterclockwise reading (Convention~\ref{convention. labels for mu^k}\eqref{item. ordering boundary arcs}).

This does result in some (very minor) differences in the algebra, which we enumerate here for the reader's convenience:
\enum
	\item In our work, a key filtered colimit will be of the form $CF^*(X,Y_0) \to CF^*(X,Y_{1}) \to \ldots$, induced by a sequence of morphisms $Y_0 \to Y_{1} \to \ldots$ (see Lemma~\ref{lemma. hom is wrapped cohomology}). In~\cite{gps}, however, the analogous colimit is of the form $CF^*(X_{0},Y) \to CF^*(X_{1},Y) \to \ldots$, induced by maps $\ldots \to X_{1} \to X_0$.
	\item When we define our directed $A_\infty$-category $\cO$, we will define $\hom_{\cO_j}( (i,L,w) , (i', L', w'))$ to be zero if $w > w'$; this is opposite the convention utilized in~\cite{gps}.
\enumd
\end{remark}

\begin{convention}[Strip coordinates and positive/negative punctures]\label{convention. strip coordinates}\label{convention. positive negative punctures}
We will often denote an element of the infinite strip $\RR \times [0,1]$ by $(\tau,t)$.
Every boundary-punctured holomorphic disk $S$ will be equipped with strip-like ends---i.e., holomorphic embeddings
	\eqnn
	\epsilon: [0,\infty) \times [0,1] \to S,
	\qquad
	\epsilon: (-\infty,0] \times [0,1] \to S,
	\eqnd
(called positive and negative, respectively)
that converge as $\tau \to \pm \infty$ to the boundary punctures of $S$. We will call these boundary punctures positive and negative accordingly.

In all our applications, there will exactly one negative boundary puncture of $S$, while all other boundary punctures will be positive.
\end{convention}

\begin{figure}[ht]
    \eqnn
			\xy
			\xyimport(8,8)(0,0){\includegraphics[width=2in]{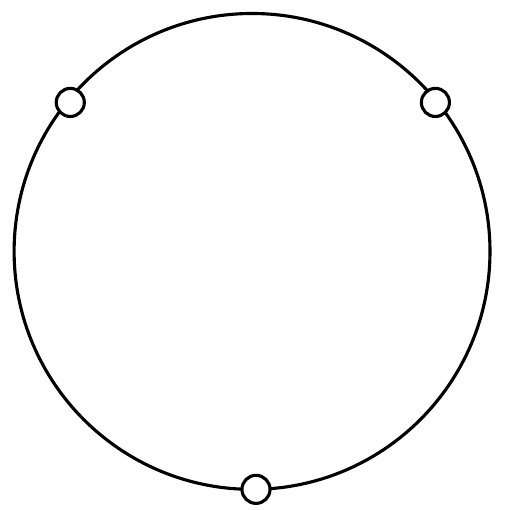}}
				,(4.1,1)*{-\infty}	
				,(1.65,6)*{+\infty}	
				,(6.1,6)*{+\infty}
				,(7.5,2)*{L_0}	
				,(5.8,7.9)*{L_1}	
				,(2.3,7.9)*{L_{k-1}}	
				,(0.5,2)*{L_k}
				,(4,8)*{\ldots}
			\endxy
    \eqnd
    \begin{figurelabel}\label{image.labeled disk}
    The labeling of a holomorphic disk contributing to a $\mu^k$ operation as in~\eqref{eqn. mu^k order}.
    \end{figurelabel}
\end{figure}

\begin{convention}[Morphisms and $\mu^k$]\label{convention. labels for mu^k}

The operations $\mu^1,\mu^2, \ldots$ of a Fukaya category will be defined as usual by counting (pseudo)holomorphic maps
	\eqnn
	u: S \to M
	\eqnd
satisfying certain boundary conditions.

\enum
\item (Generating chords.) The generators of the Floer cochain complex $CF^*(L_0,L_1)$ are in bijection with certain chords $[0,1] \to M$ from $L_0$ to $L_1$. In particular, in strip-like end coordinates, the limiting chords
	\eqnn
	\lim_{\tau \to \pm \infty} u\circ \epsilon(\tau,-)
	\eqnd
will be read as a morphism from the brane labeled at time $t=0 \in [0,1]$ to the brane labeled at time $t=1 \in [0,1]$.

Our chords will almost always be {\em constant} in our applications, but it will be healthy to conjure this convention as a rule of thumb.
\item (Inputs and outputs.) The {\em negative} puncture defines an output, while the positive punctures are inputs. More precisely, a strip-like end of a negative puncture will be decorated by boundary conditions giving rise to an output of the $\mu^k$ operations, and the positive strip-like ends will be decorated by boundary conditions encoding the inputs of the $\mu^k$ operations.
\item (Ordering the boundary arcs.) \label{item. ordering boundary arcs} The branes decorating the boundary arcs of $S$ will be read in an order given by the boundary orientation of $\del S$ induced from the standard holomorphic orientation of $S$---in particular, when we draw a picture of $S$, we read the brane labels counterclockwise.
\item (Orienting the boundary arcs.) Likewise, when we equip a boundary arc of $S$ with a moving boundary condition, the ``positivity'' of the moving boundary isotopy of a brane will be with respect to the boundary orientation of the arc. (See for example Definition~\ref{defn. non negative wrapping}.)
\item (Composition order.) Counting holomorphic disks with boundary conditions $L_0,\ldots,L_k$ (read counter-clockwise from the negative puncture) gives rise to the operation
	\eqn\label{eqn. mu^k order}
	\mu^k:
	CF(L_{k-1},L_k) \tensor \ldots \tensor CF(L_0,L_1)
	\to
	CF(L_0,L_k).
	\eqnd
\enumd
For a summary of these conventions, see Figure~\ref{image.labeled disk}.
\end{convention}

\begin{example}
Fix a brane $L \subset M$ and fix an isotopy from $L$ to $L'$. If the isotopy is non-negative (Definition~\ref{defn. non negative wrapping}), one  obtains an element of $CF(L,L')$ by counting holomorphic disks with a single boundary puncture, and with moving boundary condition dictated by the isotopy; this is detailed in~\cite{oh-tanaka-liouville-bundles}. See also Figure~\ref{figure. continuation disk}.
\end{example}

\begin{figure}[ht]
	\eqnn
    			\xy
    			\xyimport(8,8)(0,0){\includegraphics[width=2in]{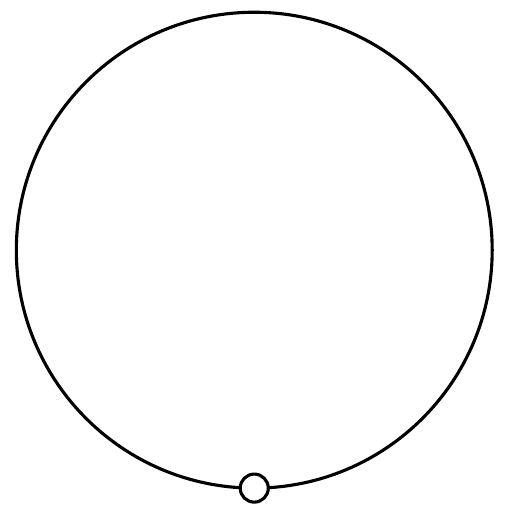}}
				,(4.1,1)*{-\infty}
				,(7.5,2)*{L}		
				,(0.4,2)*{L'}
				,(4,8.2)*{\cL}
    			\endxy
    \eqnd
    \begin{figurelabel}\label{figure. continuation disk}
    A holomorphic disk with one boundary puncture and with a moving boundary condition given by a non-negative isotopy $\cL$. The count of such disks gives rise to an element of $CF(L,L')$.
    \end{figurelabel}
\end{figure}

\begin{example}
Fix a brane $L \subset M$ and fix an isotopy $\cL$ from $L$ to $L'$. Fix also a brane $K \subset M$. If the isotopy is non-negative (Definition~\ref{defn. non negative wrapping}), one  obtains a chain map $CF(K,L) \to CF(K,L')$ by counting holomorphic strips with moving boundary condition at $t=1$. See Figure~\ref{figure. continuation strip}; we also refer to~\cite{oh-tanaka-liouville-bundles} for more details.
\end{example}

\begin{figure}[ht]
	\eqnn
    			\xy
    			\xyimport(8,8)(0,0){\includegraphics[width=3in]{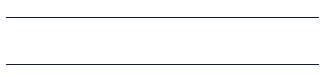}}
					,(-0.5,4.1)*{-\infty}
					,(8.4,4.1)*{+\infty}
					,(4,0.5)*{K}
					,(4,7.5)*{\cL}
					,(7.5,7.2)*{L}
					,(0.6,7.2)*{L'}
    			\endxy
    \eqnd
    \begin{figurelabel}\label{figure. continuation strip}
    A holomorphic strip with moving boundary condition $\cL$ at $t=1$ and fixed boundary condition $K$ at $t=0$, counts of which define a continuation map $CF(K,L) \to CF(K,L')$. Note that the moving boundary condition places $L$ near $\tau = \infty$, and places $L'$ near $\tau = -\infty$. (In particular, the isotopy evolves in the $-{\del/\del \tau}$ direction.)
    \end{figurelabel}
\end{figure}

\subsection{Acknowledgments}
We would like to thank
Gabriel Drummond-Cole,
Rune Haugseng,
Sheel Ganatra,
Sander Kupers,
John Pardon,
Yasha Savelyev, and
Zack Sylvan
for helpful conversations.

The first author is supported by the IBS project IBS-R003-D1.
The second author was supported by
IBS-CGP in Pohang, Korea and
the Isaac Newton Institute in Cambridge, England,
during the preparation of this work. This material is also based upon work supported by the National Science Foundation under Grant No. DMS-1440140 while the second author was in residence at the Mathematical Sciences Research Institute in Berkeley, California, during the Spring 2019 semester.

\clearpage
\section{Geometric background}
We assume the reader is familiar with the definition of Liouville sector. (Background is given in~\cite{gps} and in~\cite{oh-tanaka-liouville-bundles}.) Informally, a Liouville sector is an exact symplectic manifold $M$ with boundary, with two important features: (i) Outside a compact subset of $M$, the Liouville flow induces an isomorphism of $M$ with the symplectization of a compact, contact manifold with boundary, and (ii) The boundary $\del M$ satisfies a ``barrier condition'' guaranteeing that holomorphic curves with boundary bounded away from $\del M$ will have interiors bounded away from $\del M$.

\begin{notation}\label{notation. liouville notation}
\begin{enumerate}
\item \label{item. Liouville flow Z} We let $Z$ denote the Liouville vector field on $M$.
\item \label{item. Liouville form theta} We let $\theta$ denote the Liouville form on $M$.
\item \label{item. del infty M} We let $\del_\infty M$ denote the contact boundary of $M$. This is a contact manifold with boundary, and is well-defined up to co-oriented contact diffeomorphism.
\item\label{item. r, s, iota} We will often choose a proper exact embedding $\iota: \del_\infty M \times \RR_{\geq 0} \to M$ from the symplectization of the contact boundary. We will let  $r$ denote the $\RR_{\geq 0}$ coordinate, and set $r=e^s$,  i.e., $s = \log r$.
\end{enumerate}
\end{notation}

\begin{defn}[Liouville automorphisms]
\label{defn:liouville-isomorphism-M}
\label{defn:liouville-automorphism}
Let $M_i $,
$i = 0,\, 1$, be Liouville sectors. A {\em Liouville isomorphism} from $M_0$ to $M_1$ is a diffeomorphism $\phi: M_0 \to M_1$ satisfying
	\eqnn
	\phi^*\theta_1 = \theta_0 + df
	\eqnd
for some compactly supported smooth function $f: M_0 \to \RR$.
If $M_0 = M_1$, we call $\phi$ a Liouville {\em automorphism}.
\end{defn}

\begin{definition}\label{defn:Liouaut}
Let $M$ be a Liouville sector.
We let
	\eqnn
	\Liouauto(M)
	\eqnd
denote the topological group of Liouville automorphisms of $M$. 

We endow $\Liouauto(M)$ with the smallest topology satisfying the following properties:
\enum
\item The topology contains the weak Whitney topology inherited from $C^\infty(M,M)$, and
\item The map $\Liouauto(M) \to \aut(\del_\infty M)$ to the space of smooth contact diffeomorphisms of $\del_\infty M$ is continuous.
\enumd
\end{definition}

We also endow $\Liouauto(M)$ with a diffeology, so that we may speak of {\em smooth} maps into $\Liouauto(M)$. We impose the smallest diffeology satisfying the following:
\enum
\item It contains the subspace diffeology inherited from the diffeology of $C^\infty(M,M)$,
\item The map $\Liouauto(M) \to \aut(\del_\infty M)$ to the space of smooth contact diffeomorphisms of $\del_\infty M$ is smooth.
\enumd
We note that because $\del_\infty M$ is compact, we can unambiguously endow the contact automorphism space $\aut(\del_\infty M)$ with the usual group diffeology inherited from the diffeomorphism group of $\del_\infty M$.

\begin{remark}
Having endowed $\Liouauto(M)$ with a diffeological group structure, the Milnor classifying space $B\Liouauto(M)$ is naturally endowed with a diffeological space structure, by a construction of Christensen-Wu~\cite{christensen-wu-smooth-classifying-spaces}. 
\end{remark}

\begin{remark}
The details of the above topology and diffeology will not play an explicit role in this paper; all that will matter is that the results of Section~\ref{section. smooth approximation} are true thanks to the above topology and diffeology.
\end{remark}

\subsection{Liouville bundles}

\begin{defn}[Liouville bundle]\label{defn. liouville bundle}
Fix a Liouville sector $M$. A {\em Liouville bundle} with fiber $M$ is the choice of a smooth $M$-bundle $p: E \to B$, together with a smooth reduction of the structure group from $\diff(M)$ to $\Liouauto(M)$.
\end{defn}

Our main examples will be the universal Liouville bundle $E\Liouauto(M) \to B\Liouauto(M)$ (where the smoothness of reduction of structure group is understood in the diffeological sense) and those Liouville bundles pulled back along smooth maps $j: |\Delta^n_e| \to B\Liouauto(M)$.

\subsection{Branes and cofinal sequences of wrappings}\label{section. cofinal wrappings}

Let $M$ be a Liouville sector.

\begin{defn}
A subset $A \subset M$ is called {\em conical near infinity} if for some (and hence all) $\theta \in [\theta]_{\liou}$, and for some compact subset $K$, the complement $A \setminus K$ is closed under the positive Liouville flow.
\end{defn}

There are standard decorations one should put on Liouville sectors and their Lagrangians to obtain a $\ZZ$-graded, $\ZZ$-linear Fukaya category---for example, gradings and Pin structures. We assume these structures to be chosen throughout. To that end:

\begin{defn}\label{defn. branes}
Let $M$ be a Liouville sector.
A {\em brane} is a conical-near-infinity Lagrangian $L \subset M$ equipped with the relevant decorations.
\end{defn}

\begin{example}\label{example. compact brane}
So for example, if $L$ is a compact Lagrangian, then $L$ (when equipped with the appropriate decorations) is a compact brane. Note also that our branes have no boundary---by definition of Lagrangian (submanifold), $L$ is locally diffeomorphic to Euclidean space, and hence boundaryless.
\end{example}

\begin{defn}[Non-negative isotopy]\label{defn. non negative wrapping}
Now fix an exact Lagrangian isotopy $j: L \times [0,1]_t \to M$ through conical-near-infinity Lagrangians. (In particular, this induces an isotopy of Legendrians inside $\del_\infty M$.) We say this is a {\em non-negative wrapping}\footnote{In~\cite{gps}, this notion is called a positive wrapping (see Definition 3.20 of loc. cit.).}, or a {\em non-negative isotopy} (of the Lagrangians) if for some (and hence any) choice of Liouville form $\theta$ on $M$, we have the following outside a compact subset of $L$:
	\eqnn
	\theta (Dj(\del_t)) \geq 0.
	\eqnd
Put another way, the flow of $L$ in $\del_\infty M$ is non-negative with respect to the contact form induced by $\theta$.
\end{defn}

\begin{defn}[Cofinal sequence of non-negative wrappings]\label{defn. cofinal wrapping}
Now suppose one has chosen a sequence of conical-near-infinity branes
	\eqnn
	L^{(0)}, L^{(1)}, \ldots
	\eqnd
together with a non-negative wrapping from $L^{(i)}$ to $L^{(i+1)}$ for every $i$. We say this is a {\em cofinal sequence of non-negative wrappings} if the following holds: For any non-negative wrapping of $L^{(0)}$ to another conical-near-infinity brane $L'$, there exists
	\enum
	\item $w \in \ZZ$ and
	\item  a non-negative wrapping from $L'$ to $L^{(w)}$
	\enumd
such that the composite isotopy
	\eqnn
	L^{(0)} \to L' \to L^{(w)}
	\eqnd
is homotopic to the composite isotopy
	\eqnn
	L^{(0)} \to L^{(1)} \to \ldots \to L^{(w)}
	\eqnd
through non-negative isotopies.
\end{defn}

\begin{remark}
The non-negativity of the wrapping allows us to define so-called continuation elements
(see Section~\ref{section. continuation maps}); these yield in particular cohomology classes
	\eqnn
	c \in HF^*(L^{(i-1)},L^{(i)}).
	\eqnd
For any brane $K$ transversal to the $L^{(i)}$, we will hence be able to define a sequence of cohomology groups
	\eqnn
	\ldots \to HF^*(K,L^{(i-1)}) \xrightarrow{c_*} HF^*(K,L^{(i)}) \to \ldots
	\eqnd
by using the (cohomology-level) $\mu^2$ operation.
The directed limit (i.e., colimit) of this sequence will be isomorphic to the cohomology of the morphism complexes in our family wrapped categories (Lemma~\ref{lemma. hom is wrapped cohomology}).
\end{remark}

\subsection{Liouville automorphisms for decorations}\label{section.automorphisms with decorations}

We now describe automorphisms groups of Liouville manifolds that ``respect'' particular decorations that are extrinsic to the data of $[\theta]$, focusing on the examples of gradings and background  class $b \in H^2(M;\ZZ/2\ZZ)$. The methods here carry over to other decorations we anticipate will be of use in Floer theory, especially when one must trivialize more than $\det^2(TM)$.

\subsubsection{For gradings}
A grading on $M$ is the data of a trivialization $\det^2(TM) \cong \CC \times M$ as a complex line bundle. One may equivalently encode this data in the following homotopy-coherent diagram:
	\eqnn
	\xymatrix{
	&& EU(1) \simeq \ast \ar[d] \\
	M \ar[rr]_-{\det^2(TM)} \ar[urr] && BU(1) \simeq K(\ZZ,2).
	}
	\eqnd
Explicitly, the line bundle $\det^2(TM)$ is classified by a map to $BU(1) \simeq K(\ZZ,2)$, and a trivialization is given by the data of a null-homotopy of this map. We emphasize that when we say the above diagram is homotopy coherent, we are not merely asserting the existence of a null-homotopy; the diagram represents a choice of null-homotopy.

A Liouville automorphism equipped with data respecting this null-homotopy is not merely a Liouville automorphism $\phi: M \to M$; it is also the data of a higher homotopy coherent diagram:
	\eqnn
	\xymatrix{
		&&&\ast \ar[dd] \\
	M \ar[dr]_-{\phi} \ar[drrr] \ar[urrr]  \\
		& M \ar[rr] \ar[uurr] && BU(1)
	}
	\eqnd
The space of such data is encoded as a homotopy fiber product:

\begin{defn}[$\Liouautgr$]
We let $\Liouautgr(M)$ denote the space of Liouville automorphisms of $M$ respecting gradings. It is defined to be the homotopy pullback:
	\eqn\label{eqn. decorated Liouaut gr}
	\xymatrix{
	\Liouautgr(M) \ar[r] \ar[d]
		& \hom_{\Top_{/(\ast \to BU(1))}}(M,M) \ar[d] \\
	\Liouauto(M) \ar[r]
		& \hom_{\Top_{/BU(1)}}(M,M)
	}
	\eqnd
\end{defn}

\begin{remark}
Let us explain the maps in~\eqref{eqn. decorated Liouaut gr}.

$\Top_{/BU(1)}$ is the slice $\infty$-category of topological spaces equipped with a map to $BU(1)$. $\hom_{\Top_{/BU(1)}}$ is the morphism space in this $\infty$-category.
Likewise, $\hom_{\Top_{/(\ast \to BU(1))}}$ is the morphism space in the slice $\infty$-category of spaces equipped with a map to $BU(1)$ along with a null-homotopy of said map. aLet us also explain the bottom horizontal arrow. This is most efficiently encoded by first observing a homotopy equivalence
	\eqnn
	\Liouauto(M) \xleftarrow{\sim} \Liouaut^{\text{comp}} (M)
	\eqnd
from the space of those Liouville automorphisms of $M$ equipped with a homotopy between choices of $\omega$-compatible almost complex structures. The forgetful map is a homotopy equivalence because the space of almost-complex structures compatible with $\omega$ is contractible. On the other hand, $\Liouaut^{\text{comp}}(M)$ has a natural map to $\hom_{\Top_{/BU(1)}}(M,M)$ where the choice of almost-complex structure $J$ on $TM$ defines a map from $M$ to $BU(1)$ (classifying $\det^2(TM)$), and the homotopy between $J$ and $\phi^*J$ determines the homotopy between $M \to BU(1)$ and the composite $M \xra{\phi} M \to BU(1)$. In summary, the bottom horizontal arrow is determined by considering the composite
	\eqnn
	\Liouauto(M) \xleftarrow{\sim} \Liouaut^{\text{comp}} \to \hom_{\Top_{/BU(1)}}(M,M)	
	\eqnd
and choosing a homotopy inverse (together with a homotopy exhibiting the homotopy inverseness---this is a choice in a contractible space of choices) to the left-hand arrow.
\end{remark}

\begin{remark}\label{remark. Liouautgr is a group}
All the spaces in~\eqref{eqn. decorated Liouaut gr} are endomorphism spaces of particular $\infty$-categories. For example, $\Liouauto(M)$ is the endomorphism space of $M$, considered as an object of the topologically enriched category of Liouville sectors with morphisms being Liouville isomorphisms. $\Liouautgr(M)$ is the endomorphism space of a corresponding fiber product category; thus it is an $A_\infty$-space (and in fact, group-like). Because the forgetful map $\Liouautgr(M) \to \Liouauto(M)$ is now seen to arise from a functor, it gives rise to a map of $A_\infty$-algebras (in the $\infty$-category of spaces).
\end{remark}

We are interested in studying $\Liouauto(M)$, so it will be useful to know how far away the homotopy type of $\Liouautgr(M)$ is from that of $\Liouauto(M)$.

Because the forgetful map is a homomorphism (see Remark~\ref{remark. Liouautgr is a group}), it suffices to compute the fiber over the identity map $\phi = \id_M: M \to M$. So it suffices to compute the homotopy fiber of the map

$$
\hom_{\Top_{/(\ast \to BU(1))}}(M,M) \to \hom_{\Top_{/BU(1)}}(M,M)
$$

over the identity morphism. This homotopy fiber is straightforwardly seen to be homotopy equivalent to
	\eqnn
	\hom(M, \Omega BU(1)) \simeq \hom(M, U(1)).
	\eqnd
That is, we have a fiber sequence
	\eqnn
	\hom(M,U(1)) \to \Liouautgr(M) \to \Liouauto(M).
	\eqnd
The fiber is clearly 1-connected (i.e., has no homotopy groups in degrees $\geq 2$), so from the Serre long exact sequence of homotopy groups, we conclude:

\begin{prop}\label{prop. Liouaut gr homotopy groups}
The forgetful map $\Liouautgr(M) \to \Liouauto(M)$ induces an isomorphism on homotopy groups $\pi_k$ for $k \geq 3$. The induced map on $\pi_2$ is an injection.
\end{prop}

\subsubsection{For a background class \texorpdfstring{$b$}{b}}
Likewise, fix an element $b \in H^2(M;\ZZ/2\ZZ)$. This is classified by a map $\tilde b: M \to K(\ZZ/2\ZZ, 2)$, well-defined up to homotopy. We fix $\tilde b$ once and for all.

\begin{defn}[$\Liouautb(M)$]
We let $\Liouautb(M)$ denote the space of Liouville automorphisms of $M$ equipped with data respecting $b$. It is defined via the following homotopy pullback square:
	\eqnn
	\xymatrix{
	\Liouautb(M) \ar[r] \ar[d]
		& \hom_{\Top_{/K(\ZZ/2\ZZ,2)}}( (M,\tilde b) , (M,\tilde b))\ar[d] \\
	\Liouauto(M) \ar[r]
		& \hom_{\Top}(M,M)
	}
	\eqnd
\end{defn}

Informally, an element of $\Liouautb(M)$ is the data of a Liouville automorphism $\phi: M \to M$, together with a homotopy from $\tilde b \circ \phi$ to $\tilde b$. The same observation as in Remark~\ref{remark. Liouautgr is a group} shows $\Liouautb(M)$ is a group-like $A_\infty$-algebra in spaces, and $\Liouautb(M) \to \Liouauto(M)$ lifts to an $A_\infty$ map.

Moreover, the fiber of this map can be computed as the fiber of the map
$$
\hom_{\Top_{/K(\ZZ/2\ZZ,2)}}((M,\tilde b),(M,\tilde b)) \to \hom_{\Top}(M,M)
$$
which is given by
	\eqnn
	\hom_{\Top}(M, \Omega K(\ZZ/2\ZZ,2)) \simeq \hom_{\Top}(M, \RR P^\infty).
	\eqnd
As before, we conclude:

\begin{prop}\label{prop. Liouaut b homotopy groups}
The forgetful map $\Liouautb(M) \to \Liouauto(M)$ induces an isomorphism on homotopy groups $\pi_k$ for $k \geq 3$. The induced map on $\pi_2$ is an injection.
\end{prop}

\subsubsection{For both}

\begin{defn}[$\Liouautgrb(M)$]
We define $\Liouautgrb(M)$ as the homotopy pullback
	\eqnn
	\xymatrix{
	\Liouautgrb(M) \ar[r] \ar[d] & \Liouautgr(M) \ar[d] \\
	\Liouautb(M) \ar[r] & \Liouauto(M).
	}
	\eqnd
\end{defn}

\begin{prop}\label{prop. Liouautgrb homotopy groups}
The forgetful map $\Liouautgrb(M) \to \Liouauto(M)$ induces an isomorphism on homotopy groups $\pi_k$ for $k \geq 3$. The induced map on $\pi_2$ is an injection.
\end{prop}

\begin{proof}
Combine Proposition~\ref{prop. Liouaut gr homotopy groups} and~\ref{prop. Liouaut b homotopy groups}.
\end{proof}

\clearpage
\section{Imported results}
Our paper depends on three results from other papers, so we recall them here for the reader's benefit.

\subsection{Localizations of \texorpdfstring{$A_\infty$}{A-infinity}-categories}\label{section. A oo localizations}
Let $\cA$ be an $A_\infty$-category, and let $H^0 \cA$ denote its 0th cohomology category. We fix a collection of morphisms $C$ in $H^0 \cA$. The {\em localization of $\cA$ along $C$}, denoted $\cA[C^{-1}]$, is a new $A_\infty$-category obtained by freely adjoining inverses to elements of $C$. So a functor out of $\cA[C^{-1}]$ is the same thing as a functor out of $\cA$ sending every element of $C$ to an equivalence (see Recollection~\ref{recollection. A oo facts}\eqref{item. localization universal property} below).

Localizations arise in this paper because they have been shown to be a useful and computable technique for defining wrapped Fukaya categories---this follows idea of Abouzaid-Seidel, as utilized in~\cite{gps}.

\begin{recollection}\label{recollection. A oo facts}
We recall some facts about $A_\infty$-categories and their localizations.
We refer the reader to~\cite{oh-tanaka-localizations} for more precise details.
\enum
\item\label{item. Aoo cat as an oocat}
The $\infty$-category $\Ainftycat$ (see Choice~\ref{choice. base ring R}) can be obtained from the usual category $\Ainftycatt$ of $A_\infty$-categories by formally inverting the functors that are quasi-equivalences. Moreover, functor spaces in $\Ainftycat$ may be computed using known model category techniques. See~\cite{tanaka-Aoo-units}.
\item\label{item. filtered localization hom}
When $\cA$ enjoys certain algebraic properties---properties that $\cO_j$ will satisfy---one can compute $\hom$ complexes in ${\cA[C^{-1}]}$ as a filtered colimit of a diagram built from morphism complexes of $\hom_{\cA}$. (See Section~5.5 of~\cite{oh-tanaka-localizations}.)
\item\label{item. localization universal property}
For any $A_\infty$-category $\cB$, the natural map of functor spaces
	\eqnn
	\hom_{\Ainftycat}(\cA[C^{-1}], \cB) \to \hom_{\Ainftycat}(\cA, \cB)
	\eqnd
is an inclusion of connected components---the essential image may be identified with those functors out of $\cA$ that send elements of $C$ to equivalences in $\cB$. (See Definition~5.2 and Theorem 5.9 of~\cite{oh-tanaka-localizations}.)
\item\label{item. Tw} There is a formal way to enlarge an $A_\infty$-category $\cA$ so that the enlargement contains all mapping cones. This is called the category of twisted complexes of $\cA$, and is denoted $\Tw \cA$.
\item\label{item. Aoo nerve} We will also use that any $A_\infty$-category defines an $\infty$-category by taking the {\em $A_\infty$-nerve}; this is a construction due to the second author~\cite{tanaka-pairing} and Faonte~\cite{faonte} independently, generalizing a dg-construction due to Lurie~\cite{higher-algebra}.
\enumd
\end{recollection}

\subsection{Homotopy types from categories of simplices}\label{section. smooth approximation}

\begin{notation}[Simplices]\label{notation. standard simplices}\label{notation. extended simplices}
Fix an integer $d \geq 0$. We let $|\Delta^d|$ denote the standard topological $d$-dimensional simplex, given by the subset of those $(t_0,\ldots,t_d) \in \RR^{d+1}$ satisfying $t_i \geq 0$ and $\sum t_i = 1$. More generally, given any linear order $A$, we let $|\Delta^A|$ denote the subset of $\RR^A$ given by those $(t_a)_{a \in A}$ satisfying $t_a \geq 0$ and $\sum_{a \in A} t_a =1$.
{}
We will sometimes refer to $|\Delta^A|$ as the geometric realization of $A$.
{}
The {\em extended} $d$-simplex is the space $|\Delta^d_e| \subset \RR^{d+1}$ of those $(t_0,\ldots,t_d) \in \RR^{d+1}$ satisfying $\sum t_i = 1$. It is abstractly homeomorphic to $\RR^d$>
\end{notation}

When dealing with infinite-dimensional entities such as $B\Liouaut$ or $\Liouaut$, one must specify what we mean by a smooth map $j$.  We utilize the framework of diffeological spaces, and in particular, we will study smooth maps from $|\Delta^n_e| \cong \RR^n$ to $\widehat{B\Liouaut(M)}$, where the ``hat'' notation denotes that we have put a diffeological space structure on $B\Liouauto(M)$. We refer the reader to~\cite{oh-tanaka-smooth-approximation} for details.

\begin{notation}\label{notation. simp C oo}
We define a category $\simp(\widehat{B\Liouauto(M)})$.
Objects are smooth simplices $j: |\Delta^n_e| \to \widehat{B\Liouauto(M)}$ (for $n \geq 0$) and morphisms are simplicial maps $|\Delta^n_e| \to |\Delta^{n'}_e|$ that are compatible with $j$ and $j'$.
\end{notation}

\begin{recollection}\label{recollection. smooth approximation}
Here are the facts we will need. All of these are detailed in~\cite{oh-tanaka-smooth-approximation}.
\enum
\item \label{item. localization of BLiou}
The localization of $\simp$ along all morphisms is an $\infty$-category homotopy equivalent to the classifying space $B\Liouauto(M)$. We remind the reader that there is a natural way to convert a topological space into an $\infty$-category, called the singular complex functor; the resulting $\infty$-category is denoted $\sing(B\Liouauto(M))$. The more precise result is that the map $\simp(\widehat{B\Liouauto(M)}) \to \sing(B\Liouauto(M))$ is a localization. 

The same result holds for the group $\diff(Q)$ as well, so that $\sing(B\diff(Q)) \simeq \BB \diff(Q)$ is a localization of $\simp(\widehat{B\diff(Q)})$.

\item \label{item. pullback}
$\widehat{B\Liouaut(M)}$ has a tautological smooth $M$-bundle over it, and one may pull back this bundle along smooth maps $j$. If the domain of $j$ is a smooth manifold in the usual sense, the pullback is a smooth bundle in the usual sense. (See for example~\cite{christensen-wu-smooth-classifying-spaces}.)

\item \label{item. localization of subdiv}
Finally, for any simplicial set $S$, on has the barycentric subdivision $\subdivision(S)$. One has a natural map $\subdivision(S) \to S$, and this map induces an equivalence of localizations---that is, of the $\infty$-groupoids obtained by inverting all morphisms in $\subdivision(S)$ and in $S$. In particular, if $S$ is already a Kan complex, the map $\subdivision(S) \to S$ exhibits $S$ as the Kan completion of $\subdivision(S)$.  
\enumd
\end{recollection}

\subsection{Continuation maps and Floer theory in Liouville bundles}\label{section. continuation maps}
Our main constructions rely on counting certain holomorphic curves in Liouville {\em bundles}. Let us summarize the foundations we laid for such counts.

\begin{recollection}\label{recollection. continuation maps}
All the results below can be found in~\cite{oh-tanaka-liouville-bundles}:
\enum
\item\label{item. gromov compactness for bundles}
\label{item. definition of non-wrapped Fukaya category} The usual Gromov compactness results hold when analyzing holomorphic curves in Liouville bundles. In fact, one can perform straightforward generalizations of the usual $C^0$ estimates in the Liouville setting (so that holomorphic curves remain in an a compact subset determined by the boundary conditions) and of the usual energy estimates (so that the usual finite energy assumptions may be applied to encode how nodal curves may develop). This was utilized in \cite{oh-tanaka-liouville-bundles} to construct a \emph{non-wrapped}
Fukaya category associated to a Liouville bundle $E \to B$; the main theorem is that this non-wrapped Fukaya category is indeed an $A_\infty$-category. We apply this in the present work by associating an $A_\infty$-category $\cO_j$ to each simplex $j: |\Delta^n| \to B(\Aut(M))$; we will recall more in Section~\ref{subsection. unwrapped}.
\item\label{item. continuation maps respect composition} As mentioned after Warning~\ref{warning. smooth simplices}, we utilize two types of continuation maps in the present work: those defined by counting holomorphic strips, and those defined by counting holomorphic disks with one boundary puncture.
The count of continuation {\em strips} is homotopic to the count of continuation {\em once-punctured disks} when the latter is post-composed by the $\mu_2$ operation.
\enumd
\end{recollection}

\clearpage
\section{The wrapped Fukaya categories}

We fix a Liouville bundle $E \to B$.
The first agenda of this section is to define the wrapped Fukaya category $\cW_j$ associated to a smooth simplex $j: |\Delta_e^n| \to B$ (Definition~\ref{defn. cW_j}).
We do this by localizing the non-wrapped Fukaya categories $\cO_j$ along non-negative continuation maps. (The intuition is that in any wrapped Fukaya category, a non-negative finite wrapping should induce an equivalence.) We denote the resulting $A_\infty$-category by $\cW_j$. This is a bundle version of an idea originally due to Abouzaid and Seidel; see also~\cite{gps}, where we learned of the idea, and whose notation we largely follow.

We then prove that the assignment $j \mapsto \cW_j$ is locally constant (i.e., forms a local system of wrapped Fukaya categories), meaning any inclusion of simplices $j \subset j'$ induces an equivalence of $A_\infty$-categories (Proposition~\ref{prop. local triviality}). This allows us to prove Theorem~\ref{theorem. main theorem}.

\subsection{Non-wrapped Fukaya categories in Liouville bundles}
\label{subsection. unwrapped}

In this subsection, we briefly recall the definition of $\cO_j$ given in \cite{oh-tanaka-liouville-bundles}.
While we will apply our constructions to the universal example of the universal bundle over $B = \widehat{B \Liouauto(M)}$, we state our constructions with the generality of an arbitrary base $B$.

\begin{choice}[$\cL_b$ and cofinal wrapping sequences.]\label{choice. cofinal wrappings}
For every point $b \in B$, we choose a countable collection $\cL_b$ of eventually conical branes in the fiber $E_b$ such that the following holds: For every eventually conical brane $L' \subset E_b$, there exists an element $L \in \cL_b$ such that $L$ admits a non-negative wrapping\footnote{
Definition~\ref{defn. non negative wrapping}.
}
 to $L'$, or $L'$ admits a non-negative wrapping to $L$.

Then, for every $L \in \cL_b$, we choose a cofinal wrapping sequence\footnote{
Definition~\ref{defn. cofinal wrapping}
}
	\eqnn
    L = L^{(0)} \to L^{(1)} \to \ldots.
    \eqnd
Finally, because this totality of choices is a countable collection, we may assume that if $L^{(w)}$ and ${L'}^{(w')}$ are in the same fiber, then they are either transverse, or $L=L'$ and $w=w'$.
\end{choice}

\begin{notation}[The wrapping index $w$]
In the cofinal wrapping sequence of Choice~\ref{choice. cofinal wrappings}, we will often denote the superscript index by $(w)$. The $w$ stands for ``wrapping index.''
\end{notation}

Fix a smooth map $j: |\Delta_e^n| \to B$. (Note $|\Delta_e^n|$ is an extended simplex as in Definition~\ref{notation. extended simplices}.)

Then we consider the non-wrapped Fukaya category $\cO_j$ associated to $j$ whose definition given
inductively on $n$---we first define $\cO_j$ for all $j$ having domain of dimension $\leq n$, then for those $j$ with domain having dimension $n+1$. (See \cite[Section 6]{oh-tanaka-liouville-bundles} for the details.)

\begin{notation}[$b_i$ and $\cL_{b_i}$]
For every $0 \leq i \leq n$, let $b_i$ be the image of the $i$th vertex of $|\Delta_e^n|$ under $j$. Recall we have chosen a countable collection of branes and a cofinal wrapping sequence of these (Choice~\ref{choice. cofinal wrappings}). In particular, $\cL_{b_i}$ denote the countable collection of branes associated to $b_i$.
\end{notation}

\begin{defn}[Objects]\label{defn. obj of O}
An object of $\cO_j$ is a triplet $(i, L, w)$ where
$$
i \in \{0, \ldots, n\}, \quad L \in \cL_{b_i}, \quad w \in \ZZ_{\geq 0}.
$$
\end{defn}

\begin{notation}[$L^{(w)}$]
One can informally think of the triplet $(i,L,w)$ as the brane $L^{(w)}$ inside $E_{b_i}$. For this reason, we will soon denote an object simply by $L^{(w)}$, omitting $i$. (See for example Definition~\ref{defn. cO_j}.)
\end{notation}

\begin{notation}[Parallel transport $\Pi$]\label{notation. parallel transport}
Fix a pair of objects $(L_0,i_0,w_0)$ and $(L_1,i_1,w_1)$. The integers $i_0$ and $i_1$ define a simplicial map $\beta: |\Delta^1| \to |\Delta^n| \subset |\Delta_e^n|$ sending the initial vertex of $|\Delta^1|$ to $i_0$ and the final vertex to $i_1$.

We let $h = j \circ \beta$. One also has an underlying ordered pair of branes $\vec L = (L_0, L_1)$.
By choosing Floer data $\Theta_{\vec L}$ (see \cite[Section 5.4]{oh-tanaka-liouville-bundles}), we have a parallel transport taking the initial fiber of $h^*E$ (i.e., the fiber above the initial vertex of $|\Delta^1|$) to the final fiber of $h^*E$.

We let $\Pi_{i_0, i_1}$ denote this parallel transport.
\end{notation}

\begin{defn}[Morphisms]\label{defn. hom groups}
For given two objects $(i_0, L_0, w_0)$ and $(i_1, L_1, w_1)$ of $\cO_j$,
we define the graded abelian group
	\eqnn
	\hom_{\cO_j}( (i_0, L_0, w_0) , (i_1, L_1, w_1))
	\eqnd
to be
	\eqnn
    \begin{cases}
    \bigoplus_{x \in \Pi_{i_0, i_1}(L_0^{(w_0)}) \cap L_1^{(w_1)}} \mathfrak{o}_{x}[-|x|]. & w_0 < w_1 \\
    R & (i_0, L_0, w_0) = (i_1, L_1, w_1) \\
    0 & \text{otherwise}.
    \end{cases}
    \eqnd
Here, $\Pi_{i_0,i_1}$ is the parallel transport map (Notation~\ref{notation. parallel transport}).
We also note that $\mathfrak{o}_{x}$ is the orientation $R$-module of rank one associated to the intersection point $x$, and $|x|$ is the Maslov index associated to the brane data.
\end{defn}

\begin{remark}
We have rendered $\cO_j$ to be {\em directed} in the $w$ index; this means that the morphism complex from $(i, L, w)$ to $(i', L', w')$ will be zero unless $w < w'$, or $(i, L, w) = (i',L',w')$ (in which case the morphism complex is just the ground ring $R$ in degree 0).
\end{remark}

\begin{remark}
The set $x \in \Pi_{i_0, i_1}(L_0) \cap L_1$ is also in bijection with the set of flat sections of $h^*E \to |\Delta^1|$ (with respect to $\Theta_{(L_0,L_1)}$) beginning at $L_0^{(w_0)}$ and ending at $L_1^{(w_1)}$. (See Notation~\ref{notation. parallel transport}.)
\end{remark}

Higher operations $\mu^d$ for $d \geq 1$ are then defined as follows.

\begin{defn}[$\mu^d$ for the non-wrapped categories]\label{defn. mu^d counting}
As usual, fix a smooth map $j: |\Delta_e^n| \to B$.
For $d \geq 1$, fix a collection
	\eqnn
	\vec{L} = \{(i_0, L_0, w_0) , \ldots, (i_d, L_d, w_d)\}.
	\eqnd
We may assume $w_0 < \ldots < w_d$ by Definition~\ref{defn. hom groups} (otherwise $\mu^d$ is forced to be 0) .

Note that the integers $i_0, \ldots, i_d$ induce a simplicial map $\beta: |\Delta^d| \to |\Delta^n| \subset |\Delta_e^n|$ by sending the $a$th vertex of $|\Delta^d|$ to the $i_a$th vertex of $|\Delta^n|$. (This assignment, of course, need not be order-preserving.)
For a given collection of intersection points
	\eqnn
    x_a \in \Pi_{i_{a-1}, i_{a}}\left(L_{a-1}^{(w_{a-1})}\right) \cap L_{a}^{(w_{a})}
    \qquad
    (a = 1, \ldots, d)
	\eqnd
and
	\eqnn
    x_0 \in \Pi_{i_0, i_d}\left(L_0^{(w_{0})}\right) \cap L_d^{(w_{d})},
    \eqnd
we define
	\eqnn
	\cM(x_d, \ldots, x_1; x_0)
	\eqnd
	to be the moduli space of holomorphic sections $u$
	\eqnn
		\xymatrix{
			& & & & E \ar[d] \\
		\cS_r \ar[r]_{\subset} \ar[urrrr]^{u} & \overline{\cS}^{\circ}_{d+1} \ar[r]_{\nu_\beta} & |\Delta^{d}| \ar[r]_-{\beta} & |\Delta^n| \subset |\Delta_e^n| \ar[r]_j & B
		}
    \eqnd
satisfying the obvious boundary conditions.

\begin{remark}\label{remark. nu beta}
The map $\nu_{\beta}: \overline{\cS}_{d+1}^\circ \to |\Delta^n|$ is defined in \cite[Subsection 5.3]{oh-tanaka-liouville-bundles}. These are degree one maps from the $(d+1)$st universal family of disks to the $d$-simplex and their existence is due originally to Savelyev~\cite{savelyev}.
\end{remark}

As usual, the brane structures on the $L^{(w)}$ allow us to orient these moduli spaces, and predict their dimension based on the degrees of the $x_a$. We define
	\eqnn
    \mu^d(x_d, \ldots, x_1)
    =
    \sum_{x_0} \#{\cM(x_d,\ldots,x_1;x_0)} x_0
    \eqnd
where the number $\# \cM$ is counted with sign. In case our branes are not $\ZZ$-graded, we as usual we declare the $x_0$ coefficient of $\mu^d$ to be zero when there is no zero-dimensional component of $\cM(x_d,\ldots,x_a;x_1)$.
\end{defn}

\begin{remark}\label{remark. depends only on j o beta}
Given an ordered $(d+1)$-tuple of objects in $\cO_j$ with underlying branes $\vec L$, consider the induced map $\beta: |\Delta^d| \to |\Delta^n| \subset |\Delta_e^n|$. The $A_\infty$-operations are defined by moduli spaces depending only on $h = j \circ \beta$.
\end{remark}

\begin{defn}\label{defn. cO_j}
Fix $j: |\Delta_e^n| \to B$. We let $\cO_j$ denote the $A_\infty$-category where
\begin{itemize}
	\item an object is as in Definition~\ref{defn. obj of O},
    \item $\hom_{\cO_j}$ is as in Definition~\ref{defn. hom groups}, and
    \item The operations $\mu^d$ are as in Definition~\ref{defn. mu^d counting}.
\end{itemize}
\end{defn}

\begin{remark}
When a $\mu^d$ operation involves an element of an endomorphism hom-complex $\hom_{\cO_j}(L,L) = R$, the operation is fully determined by demanding that the generator of the base ring $R$ (Choice~\ref{choice. base ring R}) be a strict unit.
\end{remark}

We then have the following result from \cite{oh-tanaka-liouville-bundles}:
\begin{theorem}[\cite{oh-tanaka-liouville-bundles}]
$\cO_j$ is an $A_\infty$-category.
\end{theorem}


\begin{construction}[Functors $\cO_{j} \to \cO_{j'}$]\label{construction. non wrapped functor}
Now suppose we have a map $\alpha: [n] \to [n']$, a smooth map $j': |\Delta_e^{n'}| \to B$, and consider the induced diagram
	\eqnn
    \xymatrix{
    |\Delta_e^n| \ar[rr]^\alpha \ar[dr]^j
    	&& |\Delta_e^{n'}| \ar[dl]^{j'} \\
        & B.
    }
    \eqnd
(Equivalently, consider a morphism in the category $\simp(B)$.) Then there are induced assignments as follows:
\enum
\item An object $(i, L, w)$ of $\cO_j$ is sent to the object $(\alpha(i), L, w)$ inside $\cO_{j'}$. Here, we are identifying the fiber of $(j')^*E$ above $\alpha(i)$ with the fiber of $j^*E$ above $i$ in the obvious way.
\item If $\alpha$ is an injection, then compatibility of the Floer data
over the above diagram (see $(\Theta2)$ of \cite[Section 5.4]{oh-tanaka-liouville-bundles}
 gives an isomorphism of graded abelian groups
	\eqnn
    \hom_{\cO_j}( (i_0,L_0,w_0) , (i_1, L_1, w_1)) \to \hom_{\cO_{j'}}( (\alpha(i_0), L_0, w_0) , (\alpha(i_1), L_1, w_1)).
    \eqnd
Otherwise, factoring $\alpha$ as a surjection followed by an injection, we have the same isomorphism by identifying the fiber above a vertex $i' \in |\Delta^{n'}|$ with a fiber above any point in the preimage $\alpha^{-1}(i')$.
\enumd
We call this assignment $\alpha_*$.
\end{construction}

\begin{prop}\label{prop. O functor}
The assignment $\alpha_*$ from Construction~\ref{construction. non wrapped functor} is an $A_\infty$-functor $\cO_{j'} \to \cO_{j}$.

In fact, the assignments $j \mapsto \cO_j$ and $\alpha \mapsto \alpha_*$ define a functor
	\eqnn
     \simp(B\Liouaut(M)) \to \Ainftycatt
    \eqnd
to the (strict) category of $R$-linear $A_\infty$-categories and $R$-linear functors between them, with the usual (strictly associative) composition  of functors.
\end{prop}

\begin{proof}
Remark~\ref{remark. depends only on j o beta} shows that the $\mu^d$ operations are respected on the nose, as one is counting sections over two bundles over $|\Delta^d|$ that admit an isomorphism respecting all boundary conditions and choice of $\cJ$ and $\cA$. This shows $\alpha_*$ is a functor of $A_\infty$-categories, simply by defining the functor to have higher homotopies equaling zero.

Now suppose we have a commutative diagram of smooth maps
	\eqnn
    \xymatrix{
    |\Delta_e^n| \ar[r]^{\alpha} \ar[dr]^{j} & |\Delta_e^{n'}| \ar[r]^{\alpha'} \ar[d]^{j'} & |\Delta_e^{n''}| \ar[dl]^{j''} \\
    & B & .
    }
    \eqnd
We must show that $(\alpha' \circ \alpha)_* = \alpha'_* \circ \alpha_*$. This is straightforward from Construction~\ref{construction. non wrapped functor}.
\end{proof}

\begin{remark}
It follows immediately from the isomorphism observed in Construction~\ref{construction. non wrapped functor} itself that when $\alpha: [n] \to [n']$ is a surjection, $\alpha_*$ is an equivalence of $A_\infty$-categories. (i.e., $\alpha_*$ is essentially surjective and induces a quasi-isomorphism of hom-complexes).
\end{remark}

\subsection{Continuation maps for different fibers}

We now consider a Liouville bundle $E \to |\Delta^1|$.
Then for any brane $X \subset E_0$ in the fiber above 0 and any brane $Y \subset E_1$ in the fiber above 1, there are two natural Floer complexes to associate:
\enum
\item The Floer complex whose $\mu^1$ term is defined by counting holomorphic sections of the Liouville bundle $E$. This is $\hom_{\cO}(X,Y)$.
\item The Floer complex obtained by first parallel transporting $X$ to $E_1$ along the edge $|\Delta^1|$, then computing the usual Floer complex $CF^*(\Pi_{0,1} X, Y)$ in $E_1$. (If $\Pi_{0,1} X$ is an object of $\cO$, this chain complex is equivalent to $\hom_{\cO}(\Pi_{0,1} X,Y)$.)
\enumd

There are obvious isomorphisms
	\eqn\label{eqn. parallel transport versus constant}
	\hom_{\cO}(X,Y) \cong CF^*(\Pi_{0,1} X, Y)
	\eqnd
and (for any $X, X' \in E_0$)
	\eqn\label{eqn. parallel transport of CF}
	\hom_{\cO}(X,X') \cong CF^*(\Pi_{0,1} X, \Pi_{0,1} X').
	\eqnd

By considering  a Liouville bundle $E \to |\Delta^2|$ over a two simplex, we have:

\begin{lemma}\label{lemma. parallel transport versus constant}
The maps \eqref{eqn. parallel transport versus constant} and~\eqref{eqn. parallel transport of CF} are quasi-isomorphisms.  Moreover, for any pair $X, X' \in \{X_i\}$, the diagram
	\eqnn
	\xymatrix{
	H^*\hom_{\cO}(X', X) \tensor H^*\hom_{\cO}(X,Y) \ar[r]^-{\mu^2} \ar[d]^{\eqref{eqn. parallel transport of CF} \tensor \eqref{eqn. parallel transport versus constant}}
		& H^*\hom_{\cO}(X', Y) \ar[d]^{\eqref{eqn. parallel transport versus constant}} \\
	HF^*(\Pi_{0,1} X', \Pi_{0,1} X) \tensor HF^*(\Pi_{0,1} X, Y) \ar[r]^-{\mu^2}
		& HF^*(\Pi_{0,1} X', Y)
	}
	\eqnd
commutes.
\end{lemma}
We leave the proof to the reader,
as it is standard given the techniques of~\cite{oh-tanaka-liouville-bundles}.

We previously knew how to speak of continuation elements of two branes in the same fiber (such as $\Pi_{0,1}X$ and $Y$; see  Section~\ref{section. continuation maps}). Given \eqref{eqn. parallel transport versus constant} and the compatibility of Lemma~\ref{lemma. parallel transport versus constant},  it makes sense to speak of continuation elements between objects of $\cO_j$ in possibly different fibers:

\begin{defn}\label{defn. continuation elements in O_j}
Given a Liouville fibration $E \to B$, fix a smooth map $j: |\Delta^n| \to B$ and consider the non-wrapped Fukaya category $\cO_j$ (Definition~\ref{defn. cO_j}). Fix two objects $(L_0,i_0,w_0)$ and $(L_1,i_1,w_1)$.

A {\em continuation element} from  $(L_0,i_0,w_0)$ to  $(L_1,i_1,w_1)$ is any element of
	\eqnn
	H^*\hom_{\cO_j} ( (L_0,i_0,w_0) , (L_1,i_1,w_1))
	\eqnd
arising from a continuation element of $HF^*(\Pi_{0,1}L_0^{(w_0)},L_1^{(w_1)})$ under the isomorphism \eqref{eqn. parallel transport versus constant}.
\end{defn}

\subsection{Equivalent choices of localizing morphisms}

Fix a smooth map $j: |\Delta_e^n| \to B$ and consider the $A_\infty$-category $\cO_j$ from Definition~\ref{defn. cO_j}.
There are two families of localizing morphisms one can consider:

\begin{defn}[$C$ and $C_\Pi$]\label{defn. C and CPi}
We let $C_\Pi$ denote the collection of morphisms in $\cO_j$ arising as continuation maps $(i, L, w) \to (i', L', w')$ for $w' > w$ (see Definition~\ref{defn. continuation elements in O_j}).

On the other hand, we let $C\subset C_\Pi$ denote the collection of continuation maps with $i=i'$ and $L=L'$.
\end{defn}

It turns out we can localize with respect to either $C$ or $C_\Pi$ from Definition~\ref{defn. C and CPi}, and end up with the same localization.

\begin{lemma}\label{lemma. localize however}
The natural map $\cO_j[C^{-1}] \to \cO_j[C_\Pi^{-1}]$ between localizations is an equivalence of $A_\infty$-categories.
\end{lemma}

\begin{proof}
Let $\cA$ be any $A_\infty$-category, and let $W$ be any class of morphisms. Consider the localization functor $\cA \to \cA[W^{-1}]$, and let $\overline{W} \subset \cA$ denote the subcategory of morphisms that are mapped to equivalences under the localization functor. It is immediate that $\cA[W^{-1}] \simeq \cA[\overline{W}^{-1}]$.

So it suffices to show that $\overline{C} \supset C_\Pi$. Suppose
	\eqnn
	[c_{01}] \in H^0\hom_{\cO}( (i_0, L_0, w_0), (i_1, L_1, w_1))
	\eqnd
is a non-negative continuation element. Then choose a large enough wrapping index $w_0'$ of $(i_0, L_0, w_0)$ so that one can find a non-negative continuation element $c_{10'}$ from $(i_1,L_1,w_1)$ to $(i_0, L_0, w_0')$. Then choose a positive enough wrapping index $w_1'$ so that one can find a non-negative continuation element $c_{0'1'}$ from $(i_0,L_0,w_0')$ to $(i_1, L_1, w_1')$. (Note that our assumption that the $(L, w)$ form a cofinal sequence allows for this.) Then one has a commutative diagram
	\eqnn
    \xymatrix{
     & (i_1, L_1, w_1') \\
    (i_0, L_0, w_0') \ar[ur]^{[c_{0'1'}]} \\
     & (i_1, L_1, w_1) \ar[ul]^{[c_{10'}]} \ar[uu]_{[c_{11'}]} \\
    (i_0, L_0, w_0) \ar[ur]^{[c_{01}]} \ar[uu]^{[c_{00'}]}
    }
    \eqnd
in the cohomology category of $\cO$, hence of $\cO[C^{-1}]$ and of $\cO[C_\Pi^{-1}]$. In $\cO[C^{-1}]$, the cohomology category admits a unique inverse to $[c_{00'}]$ and $[c_{11'}]$; then it follows that all the diagonal maps (and in particular, $[c_{01}]$) are equivalences in $\cO[C^{-1}]$. In particular, any $c_{01} \in C_\Pi$ is contained in $\overline{C}$.
\end{proof}

\begin{defn}[$\cW_j$]\label{defn. cW_j}
Fix $j: |\Delta_e^n| \to B$ a smooth map. We let $\cW_j$ denote the localization $\cO_j[C^{-1}]$, and we call it the partially wrapped Fukaya category associated to $j$.
\end{defn}

\begin{example}\label{example. GPS W over a point}
If $n=0$ and $j: |\Delta^0| \to B$ simply chooses a point $b \in B$, then $\cO_j[C^{-1}]$ is equivalent to the partially wrapped Fukaya category associated to the fiber $E_b$ (as in~\cite{gps}).
\end{example}

\subsection{Wrapped complexes as filtered colimits}

Now we compute hom-complexes $\hom_{\cW_j}(X,Y)$ for two branes $X$ and $Y$ in the same fiber.

\begin{remark}
We note that this also computes hom-complexes when the branes are in two different fibers of $E \to B$: Simply find $Y'$ in the same fiber as $X$ and consider a parallel-transport continuation $Y \to Y'$. In $\cW_j$, this map induces an equivalence $Y \simeq Y'$ by Lemma~\ref{lemma. localize however}, so we have $\hom_{\cW_j}(X,Y) \simeq \hom_{\cW_j}(X,Y')$.
\end{remark}

\begin{lemma}\label{lemma. hom is wrapped cohomology}
The hom-complex $\hom_{\cW_j}(X,Y)$ has cohomology given by the wrapped Floer cohomology (computed in the fiber containing $X$ and $Y$) as defined in \cite[Section 3.7]{abouzaid-seidel}.
\end{lemma}

\begin{proof}
This is essentially the same proof as in Lemma~3.37 of~\cite{gps} (which in turn is due to Abouzaid-Seidel's unpublished work). The main difference is that---because of our counterclockwise orientation of the boundary of disks---our continuation maps behave covariantly, as opposed to contravariantly.

Let $Y^{(0)} \to \ldots$ denote a cofinal sequence (Definition~\ref{defn. cofinal wrapping}). Without loss of generality, we may assume $Y = Y^{(0)}$. We first claim that if $c: L \to L'$ is a map in $C$ (in particular, $L$ and $L'$ are in the same fiber), then the map
	\eqn\label{eqn. wrap comparison}
	\hocolim_i \hom_{\cO_j}(L',Y^{(i)}) \to
	\hocolim_i \hom_{\cO_j}(L, Y^{(i)})
	\eqnd
is a quasi-isomorphism.
To verify this claim, note that $\mu^2$ is respected at the level of cohomology (Recollection~\ref{recollection. continuation maps}\eqref{item. continuation maps respect composition}). So we have induced maps of cohomology groups
	\begin{align}
	H^*( \hocolim_i \hom_{\cO_j}(L', Y^{(i)}))
	&	\cong \colim_i H^* \hom_{\cO_j}(L', Y^{(i)}) \nonumber \\
	& 	\to	\colim_i H^* \hom_{\cO_j}(L, Y^{(i)}) \label{eqn. make isom}\\
	& 	\cong H^*( \hocolim_i \hom_{\cO_j}(L, Y^{(i)})). \nonumber
	\end{align}
There are isomorphisms in the lines above---this is because a filtered colimit of cohomology is naturally isomorphic to the cohomology of a filtered homotopy colimit.
Now we note that the colimit of the cohomology groups is a definition for wrapped Floer cohomology (in the fiber containing $L$ and $Y$) employed in Abouzaid-Seidel's paper \cite[Section 3.7]{abouzaid-seidel}. Moreover, in this definition of the wrapped Fukaya category, $L$ and $L'$ are exhibited as equivalent objects by a continuation map (this is proven in~\cite{bko:wrapped}). So the arrow~\eqref{eqn. make isom} is an isomorphism of groups. This shows that~\eqref{eqn. wrap comparison} is a quasi-isomorphism.

Because \eqref{eqn. wrap comparison} is a quasi-isomorphism for any choice of $Y^{(i)}$ and any choice of $c: L \to L'$ in $C$, a general fact about localizations (Recollection~\ref{recollection. A oo facts}\eqref{item. filtered localization hom}) shows that the natural map
	\eqnn
	\hocolim_i \hom_{\cO_j}( X, Y^{(i)})
	\to
	\hocolim_i \hom_{\cO_j[C^{-1}]}( X, Y^{(i)})
	\eqnd
is an equivalence for any $X$. On the other hand, the latter homotopy colimit is indexed by a sequence of quasi-isomorphisms, because the maps $Y^{(i)} \to Y^{(i+1)}$ are already in $C$, being non-negative continuation maps. Thus we have
	\eqnn
	\colim_i H^*\hom_{\cO_j}( X, Y^{(i)})
	\cong \colim_i H^*\hom_{\cO_j[C^{-1}]}( X, Y^{(i)})
	\cong H^*\hom_{\cO_j[C^{-1}]}( X, Y^{(0)})
	\eqnd
while the left-hand side is the colimit definition of wrapped Floer cohomology.
This completes the proof.
\end{proof}

\subsection{Wrapped Functoriality}

We first note that $\alpha_*$  (Construction~\ref{construction. non wrapped functor}) respects $C$ (Definition~\ref{defn. C and CPi}). That is, consider a diagram
	\eqnn
    \xymatrix{
    |\Delta_e^n| \ar[rr]^\alpha \ar[dr]^j
    	&& |\Delta_e^{n'}| \ar[dl]_{j'} \\
        & B.
    }
    \eqnd
where $\alpha$ is induced by an order-preserving injection $[n] \to [n']$. By definition, the continuation maps $C$ of $\cO_j$ are sent to (some of the) continuation maps $C'$ of $\cO_{j'}$, so we have an induced functor on the localizations
	\eqn\label{eqn. map of localizations}
    \alpha_*: \cW_j \to \cW_{j'}.
    \eqnd
(Note that we have abused notation by using $\alpha_*$ again.)

Moreover, because $\cO: \simp(B) \to \Ainftycatt$ is a functor respecting each $C$, the naturality of localizations implies the following:

\begin{proposition}\label{prop. W functor}
The assignment $j \mapsto \cW_j$ and $\alpha \mapsto \alpha_*$ induces a functor of $\infty$-categories
	\eqnn
    \cW: N(\simp(B)) \to \Ainftycat.
    \eqnd
\end{proposition}

\begin{remark}
Note that the target of $\cW$ is the $\infty$-category $\Ainftycat$, rather than the (strict) category $\Ainftycatt$. Indeed, the former is where we can articulate the universal property of localizations; see Section~\ref{section. A oo localizations}.
\end{remark}

\begin{prop}[Local triviality]\label{prop. local triviality}
The map $\alpha_*: \cW_j \to \cW_{j'}$ in \eqref{eqn. map of localizations} is an equivalence of $A_\infty$-categories.
\end{prop}

\begin{proof}
We note that it suffices to prove the proposition when $j: |\Delta^0| \to B$ is a 0-simplex in $B$.

It follows that $\alpha_*$ is essentially surjective from the proof of Lemma~\ref{lemma. localize however}.
Now we must prove that the map is fully faithful (i.e., induces a quasi-isomorphism on morphism complexes). This follows from Lemma~\ref{lemma. hom is wrapped cohomology}, as $\alpha$ induces the obvious identity map on wrapped Floer cohomology groups.
\end{proof}

\subsection{Proof of Theorem~\ref{theorem. main theorem}}
Now we are ready to complete the proof of Theorem~\ref{theorem. main theorem}.

\begin{proof}[Proof of Theorem~\ref{theorem. main theorem}.]
By Proposition~\ref{prop. W functor}, we have a functor
	$$
\simp(B\Liouaut(M)) \to \Ainftycatt.
	$$
Moreover, by Proposition~\ref{prop. local triviality}, every edge in the domain is sent to an equivalence of $A_\infty$-categories; hence by the universal property of localization of $\infty$-categories, we have an induced diagram $\Delta^2 \to \Ainftycat$ as follows:
	\eqnn
	\xymatrix{
	N(\simp(B\Liouaut(M))) \ar[r]^-{\cW} \ar[d] & \Ainftycat \\
	\sing(B\Liouaut(M)) \ar@{-->}[ur] .
	}
	\eqnd
(Informally, the above is a homotopy-commutative diagram of functors between $\infty$-categories.)
Here, the left vertical map is the localization map along every edge of $N(\simp(B\Liouaut(M))$. (Recollection~\ref{recollection. smooth approximation}\eqref{item. localization of BLiou}.) The dashed arrow is the induced map on localizations, and the map we seek, because of the obvious equivalence
	\eqnn
	\BB \Liouaut(M) \simeq \sing(B\Liouaut(M)).
	\eqnd
This completes the proof.
\end{proof}

\clearpage
\section{Local systems and a bundle version of the Abouzaid map}

Fix a Liouville domain $M$ and a compact brane $Q \subset M$.
In~\cite{abouzaid-loops}, Abouzaid constructed a functor from the quadratically wrapped category of $M$ to the category of local systems on $Q$. In this section we construct a bundle version of this functor.

\begin{remark}
For concreteness, we construct this functor for the case of the Liouville bundle over $B\diff(Q)$ with fiber $M = T^*Q$; but one can do this for any Liouville bundle whose structure group has been reduced in such a way as to preserve the set $Q \subset M$.
\end{remark}

Let us outline this section's contents. Fix a smooth fiber bundle $E' \to B'$ with fiber given by a smooth, compact manifold $Q$ (possibly with boundary). To this data we assign a collection of local system categories---specifically, to each $j: |\Delta_e^n| \to B'$, we also associate a dg-category $\cP_j$ whose twisted complex category is quasi-equivalent to a category of local systems on $j^*E'$. In particular, the assignment $j \mapsto \Tw \cP_j$ is a local system of local system categories. Our main interest in this construction is the universal example when $B' = B\diff(Q)$.

Setting $M = T^*Q$, we again utilize the result that the localization of a subdivision recovers the original homotopy type (Recollection~\ref{recollection. smooth approximation}\eqref{item. localization of BLiou}). This implies the existence of a functor of $\infty$-categories
	\eqnn
	\Tw \cP: \BB\diff(Q) \to \Ainftycat.
	\eqnd

Then, the main result of the present section is to show that one can make the diagram
	\eqnn
	\xymatrix{
	\BB\diff(Q) \ar[rr] \ar[dr]^{\Tw \cP} && \BB\Liouautgrb(T^*Q) \ar[dl]_{\cW} \\
	& \Ainftycat
	}
	\eqnd
commute up to a natural transformation from $\cW$ to $\Tw \cP$ (Corollary~\ref{cor. natural transformation}). That is, we have a $\diff(Q)$-equivariant functor from $\cW(T^*Q)$ to $\Tw\cP(Q)$.

To accomplish this goal, we first construct a natural transformation from $\cO$ to $\Tw \cP$ (Proposition~\ref{prop. O to cP natural}). Then, the main aim is to prove that the non-negative continuation maps in $\cO$ are sent to equivalences in the category $\Tw C_* \cP$ of local systems (Theorem~\ref{thm. continuation maps are equivalences of twisted complexes})---this is the most geometrically involved component of our arguments. By the universal property of the localization $\cW$, we conclude that the bundle version of the Abouzaid functor descends to the wrapped  categories (Corollary~\ref{cor. natural transformation}).

We will prove in the next section that this natural transformation is a natural equivalence---i.e., that this is a $\diff(Q)$-equivariant equivalence (Theorem~\ref{theorem. cW to cP}).

\subsection{\texorpdfstring{$C_*\cP$}{C.cP} (families of local system categories)}\label{section. chains on CP}

\begin{remark}
As in Section~\ref{section. smooth approximation}, we can endow the topological space $B\diff(Q)$ with a diffeological space structure. Because $Q$ is compact, $B\diff(Q)$ satisfies smooth approximation. So, for example, if $\sing(B\diff(Q))$ denotes the usual singular complex of continuous simplices $|\Delta^n| \to B\diff(Q)$, and if $\sing^{C^\infty}(B\diff(Q))$ denote the simplicial set whose $n$-simplices are smooth maps $j: |\Delta^n_e| \to B\diff(Q)$ from extended simplices, then the inclusion of $\sing^{C^\infty}(B\diff(Q))$ to $\sing(B\diff(Q))$ is a homotopy equivalence of simplicial sets (Recollection~\ref{recollection. smooth approximation}\eqref{item. localization of BLiou}).
\end{remark}

\begin{notation}[$E_Q$]
Note that $B\diff(Q)$ carries a principal $\diff(Q)$ bundle---the universal one---and hence an associated fiber bundle with fibers $Q$.  We denote this fiber bundle by $E_Q \to B\diff(Q)$.
\end{notation}

\begin{notation}[$Q_a$]
Let $j: |\Delta_e^n| \to B\diff(Q)$ be a smooth map. For any $a \in [n]$, we let $Q_a$ denote the fiber of $j^*E_Q$ above the $a$th vertex of $|\Delta^n| \subset |\Delta_e^n|$.
\end{notation}

Now we give some notation to the Moore path space category modeling the $\infty$-groupoid $\sing(j^*E_Q)$.

\begin{construction}[$\cP_j$]\label{construction. cP_j}
Let $j: |\Delta_e^n| \to  B\diff(Q)$ be a smooth map. We let $\cP_j$ denote the topologically enriched category defined as follows:

We declare the object set to be the disjoint union of fibers
	\eqnn
	\cP_j := \coprod_{a \in [n]} Q_a.
	\eqnd
Given $q_a \in Q_a, q_b \in Q_b$, we declare
	$
	\hom_{\cP_j}(q_a,q_b)
	$
to be the topological space of continuous maps $\gamma: [0,\infty] \to j^*E_Q$ such that $\gamma$ is compactly supported (i.e., constant beyond some finite time $t \in [0,\infty]$) and such that $\gamma(0) = q_a$ and $\gamma(\infty) = q_b$. Composition is defined in the obvious way: If $t_\gamma$ is the smallest time for which $\gamma$ is constant, $\gamma' \circ \gamma$ is defined by setting
	\eqnn
	(\gamma' \circ \gamma)(t)
	=
	\begin{cases}
	\gamma(t) & t \leq t_0 \\
	\gamma'(t - t_0) & t \geq t_0.
	\end{cases}
	\eqnd
\end{construction}

\begin{notation}
Construction~\ref{construction. cP_j} defines a functor
	\eqnn
	\cP: \simp(B\diff(Q)) \to \Cat^{\Top},
	\qquad
	j \mapsto \cP_j
	\eqnd
which we denote (as indicated) by $\cP$. (It has the obvious effect on morphisms.)
Here, $\Cat^{\Top}$ is the category of categories enriched in topological spaces. For the notation $\simp(B\diff(Q))$, see Notation~\ref{notation. simp C oo}.
\end{notation}

\begin{remark}
Given any map $|\Delta_e^n| \to |\Delta_e^{n'}| \to B\diff(Q)$, the induced map $\cP_j \to \cP_{j'}$ is an equivalence\footnote{By an equivalence of topologically enriched categories, we mean an essentially surjective functors whose maps on morphism spaces are weak homotopy equivalences.} of topologically enriched categories; this follows by noting that the inclusion $j^*E_Q \to (j')^*E_Q$ is a homotopy equivalence.
\end{remark}

\begin{notation}[$\Tw C_* \cP$]\label{notation. Tw cP}
Recall that the functor $C_*$ sending a topological space $P$ to its singular chain complex $C_*P$ is lax monoidal. As a result, applying $C_*$ to the morphism spaces to a topologically enriched category $\cD$, we obtain a dg-category $C_* \cD$.

We let $C_* \cP_j$ denote the dg-category associated to $\cP_j$. We denote the composite functor
	\eqnn
	C_* \cP: \simp(B\diff(Q)) \xra{\cP} \Cat^{\Top} \xra{C_*} \dgcatt.
	\eqnd
We also denote by $\Tw C_* \cP$ the composite of $C_* \cP$ with the $\Tw$ functor (Recollection~\ref{recollection. A oo facts}\eqref{item. Tw}).
\end{notation}

\subsection{From Fukaya categories to local systems (the bundled Abouzaid functor)}
\label{subsec:familyAbouzaid}

Given any diffeomorphism $\phi: Q \to Q$, one has an induced exact symplectomorphism $\DD \phi: T^*Q \to T^*Q$ by pushing forward the effect of $\phi$ on cotangent vectors. Because this clearly respects the diffeological smooth structures, and there exist natural lifts of $\DD$ respecting the choices of $gr$ and $b=w_2(Q)$, we have an induced functor
	\eqnn
	\DD: \simp(B\diff(Q)) \to \simp(B\Liouautgrb(T^*Q)).
	\eqnd
Our present goal is to construct a natural transformation from $\cO \circ \DD$ to $\Tw C_* \cP$ (Proposition~\ref{prop. O to cP natural}).\footnote{The functor $\cO: \simp(B\Liouautgrb(T^*Q)) \to \Ainftycatt$ is from Proposition~\ref{prop. O functor}. } We do so by utilizing  a non-wrapped, family-friendly version of a construction we learned from Abouzaid's paper~\cite{abouzaid-loops}.

\begin{remark}
In~\cite{abouzaid-loops}, the author constructs a functor whose domain is a quadratically wrapped Fukaya category---i.e., one whose morphisms are defined by using quadratic Hamiltonians. Here, we instead use our non-wrapped Fukaya categories $\cO_j$ as the domain. Aside from this detail, the main ideas remain unchanged---in particular, the analytic input for the functor utilized in our present work is somewhat simpler.

And, as we show here, the construction carries through successfully when $M$ is a Liouville sector (not
necessarily a Liouville manifold). Let us remark that our bundle-version of the construction carries over to a setting in which the structure group of $\Liouaut(M)$ is reduced in such a way that every fiber admits a distinguished brane $Q$; such is the case we are in, as $\diff(Q)$ preserves the zero section of $T^*Q$.
\end{remark}

\begin{construction}[The Abouzaid functor on objects and morphisms]\label{construction. abouzaid functor on objects}
Fix a smooth map $j: |\Delta_e^n| \to B\diff(Q)$ and consider the associated fibration
$P_j= j^*B\diff(Q) \to |\Delta_e^n|$.
For an integer $a \in \{0,\ldots, n\}$, let $M_a$ be the fiber above the $a$th vertex of $|\Delta^e_n|$, and let $L_a \subset M_a$ be a brane. We assume $L_a$ intersects the zero section $Q_a \subset M_a$ transversally.
To $L_a$ we associate the following object of $\Tw C_*\cP(Q_a)$:
	\eqn\label{eqn. abouzaid functor object}
	\left(
	\bigoplus_{x_a \in L_a \cap Q_a} x_a[-|x_a|], D
	\right).
	\eqnd
That is, as a local system, the object is generated by free $R$-modules in degrees $|x_a|$, subject to a differential $D$. The differential is given as follows: Given two intersection points $x_a, x_a' \in L_a \cap Q_a$, we let
	\eqnn
	\overline{\cH}(x_a, x_a')
	\eqnd
be the compactified moduli space of (possibly broken) holomorphic sections
	\eqnn
	u: \RR \times [0,1] \to \RR \times [0,1] \times M_a
	\eqnd
with boundary on $L_a$ and $Q_a$, converging to $x_a$ and $x_a'$. Given a map $u$, consider the restriction of $u$ to the boundary line $\RR \times \{0\} \subset \RR \times [0,1]$ mapping to $Q_a$. The resulting map $\RR \to M_a$ admits an arc-length parametrization (we use the Riemannian metric induced by the choice of almost-complex structure on $M_a$), and in particular, any $u$ determines an arc-length parametrized path in $Q_a$. This induces a map from $\overline{\cH}(x_a, x_a')$ to the space of paths in $Q_a$ from $x_a$ to $x_a'$, and in particular a map of chain complexes
	\eqnn
	\cF^1: C_*\overline{\cH}(x_a, x_a') \to \hom_{\cP_j}(x_a,x_a').
	\eqnd
Then, one chooses fundamental chain classes for $\overline{\cH}$, and pushes them forward. These elements of the hom-complexes of $C_* \cP_j$ determine the differential $D$. We refer the reader to \cite[(2.27)]{abouzaid-loops} for details.
\end{construction}

\begin{construction}[The rest of the Abouzaid functor]\label{construction. abouzaid functor higher terms}
We now explain how the above assignment on objects extends to an $A_\infty$-functor $\cO_j \to \Tw C_*\cP_j$.

Fix an integer $d \geq 1$. For each $i \in \{0 , 1, \ldots, d\}$, choose also an integer $a_i \in \{0,\ldots,n\}$ and a brane $L_{a_i}$ above the $a_i$th vertex. The choices of $a_i$ determine a unique simplicial map from the $d$-simplex to the $n$-simplex (by sending the $i$th vertex to the $a_i$th vertex). Let us take a cone on this map---specifically, the map
	\eqnn
	|\Delta^{d+1}| \to |\Delta^n|,
	\qquad
	\begin{cases}
	i \mapsto a_i & i \in \{0,\ldots,d\}, \\
	i \mapsto a_{d} & i = d+1.
	\end{cases}
	\eqnd
Given generators $y_{i, i+1} \in \hom_{\cO_j}(L_{a_i},L_{a_{i+1}})$, we let
	\eqnn
	\cH(y_{0,1},\ldots,y_{(d-1),d}; Q_{a_d})
	\eqnd
denote the moduli space of holomorphic sections $u$
	\eqnn
	\xymatrix{
	j^*E|_{\cS_r} \ar[rr] & & j^* E \ar[d] \\
	\cS_r \subset \overline{\cS}_{d+1+1}^\circ \ar[r]^-{\nu} \ar@{-->}[u]^{u} & |\Delta^{d+1}| \to |\Delta^d| \ar[r] & |\Delta^n|
	}
	\eqnd
satisfying the following boundary conditions: On the strip like ends from the $i$th arc to the $(i+1)$st arc, $u$ converges to the parallel transport arc given by $y_{i, i+1}$, while the $(d+1)$st boundary arc---from the $(d+1)$st puncture to the 0th puncture---is constrained by the parallel transport boundary condition from $Q_{a_d}$ to $Q_{a_0}$. We note that $\nu$ is the same map as in Definition~\ref{defn. mu^d counting}.

Given such a $u$, one can measure the arclength of the restriction of $u$ to the $(d+1)$st boundary arc. (For example, by taking its {\em vertical} velocity; i.e., by using the fiberwise Riemannian metrics.) By the assumption of general position, an intersection $x_0 \in L_{a_0} \cap Q_{a_0}$ and $x_d \in L_{a_d} \cap Q_{a_d}$ are not related by parallel transport---we do not have triple intersections. So this guarantees that the vertically measured arclength is non-zero. (We point this out, as if the vertical velocity were zero the arc length parametrization would not be defined, and hence we would not have a map to $\hom_{\cP}$.)

As before we have a map from $\cH$ to the space of paths in $j^*E$ from $Q_{a_0}$ to $Q_{a_d}$, and this extends continuously to the compactification $\overline{\cH}$. Pushing forward fundamental chains, we obtain the desired $A_\infty$ functor maps.
\end{construction}

\begin{remark}
We note one subtlety in the construction; it is somewhat unnatural to utilize the cone map $|\Delta^{d+1}| \to |\Delta^n|$, as then verifying the $A_\infty$ functor relations forces us to use (for example) the fact that for $k \leq d$, the moduli space
	\eqnn
	\cH(y_{0,1},\ldots,y_{(k-1),k}; Q_{a_d})
	\eqnd
is cobordant to the moduli space
	\eqnn
	\cH(y_{0,1},\ldots,y_{(k-1),k}; Q_{a_k})
	\eqnd
so that the usual compactification of the moduli spaces yield the desired $A_\infty$ functor relations.
\end{remark}

\begin{notation}[The Abouzaid functor $\cF$.]\label{notation.Abouzaid cF}
We will denote by
	\eqnn
	\cF
	\eqnd
the $A_\infty$ functor defined in Constructions~\ref{construction. abouzaid functor on objects} and~\ref{construction. abouzaid functor higher terms}. We may write $\cF_j$ to denote the dependence on the simplex $j$, but will largely leave this dependence implicit.
\end{notation}

This construction clearly respects inclusions $|\Delta_e^n| \to |\Delta_e^{n'}|$. As such, we have:

\begin{prop}\label{prop. O to cP natural}
The non-wrapped Abouzaid construction $\cF$ induces a natural transformation
	\eqnn
        \xymatrix@C+2pc{
        \simp(B\diff(Q)) \rrtwocell^{\cO \circ \DD}_{\Tw C_*\cP}{\;\;\;} && \Ainftycatt.
        }
	\eqnd
\end{prop}

\subsection{The Abouzaid functor descends to the wrapped category}
Proposition~\ref{prop. O to cP natural} constructs a map from the non-wrapped, directed family of Fukaya categories to families of local system categories. To descend our construction to the {\em wrapped} setting, the main geometric result we must verify is the following.

\begin{theorem}\label{thm. continuation maps are equivalences of twisted complexes}
Let $M$ be a Liouville sector, and $c: L_0 \to L_1$ a continuation element associated to a non-negative isotopy. We also fix a compact test brane $X \subset M$. Then the map on twisted complexes
	\eqnn
	c_* : (X \cap L_0, D_0) \to (X \cap L_1, D_1)
	\eqnd
---induced by the Abouzaid functor from $\cO(M)$ to $\Tw C_*\cP(X)$---is an equivalence.
\end{theorem}

The theorem immediately implies:

\begin{cor} Let $M = T^*Q$ and $X = Q$. Then the map on twisted complexes
	$
	c_* : (Q \cap L_0, D_0) \to (Q \cap L_1, D_1)
	$
is an equivalence.
\end{cor}

By the universal property of localization and Lemma~\ref{lemma. localize however}, we also conclude:

\begin{cor}\label{cor. natural transformation}
The natural transformation from Proposition~\ref{prop. O to cP natural} induces a natural transformation from $\cW\circ \DD$ to $\Tw C_* \cP$:
	\eqnn
        \xymatrix@C+2pc{
        \simp(B\diff(Q)) \rrtwocell^{\cW \circ \DD}_{\Tw C_*\cP}{\;\;\;} && \Ainftycat
        }
	\eqnd
By the universal property of $\Tw$, this in turn induces a natural transformation $\Tw \cW  \circ \DD\to \Tw C_* \cP$.
\end{cor}

\begin{remark}
The functors  $\cW \circ \DD$ and $C_* \cP$ both send morphisms of $\simp(B\diff(Q))$ to equivalences in $\Ainftycat$. Thus they both induce a functor
	\eqnn
	\sing(B\diff(Q)) \simeq \BB \diff(Q) \to \Ainftycat
	\eqnd
by Recollection~\ref{recollection. smooth approximation}\eqref{item. localization of BLiou}. In particular, these exhibit $\diff(Q)$ actions on both $\cW(T^*Q)$ and $C_*\cP(Q)$. Corollary~\ref{cor. natural transformation} says that the map $\Tw\cW(T^*Q) \to \Tw C_* \cP(Q)$ is $\diff(Q)$-equivariant.
\end{remark}

\subsubsection{Toward a proof of Theorem~\ref{thm. continuation maps are equivalences of twisted complexes}}

Because Theorem~\ref{thm. continuation maps are equivalences of twisted complexes} is a statement contained in single fiber of a Liouville bundle (and in particular, is a statement about a single Liouville sector), we now work in a Liouville sector $M$. We will also fix a compact test brane $X \subset M$; this $X$ plays the role of the zero section $Q$ above.
For simplicity of exposition (and without loss of generality), we
will assume that $X$ is connected.

We fix a non-negative Hamiltonian isotopy
	\eqnn
	L(t): = \phi_F^t(L_0),
	\qquad L(0) = L_0,
	\qquad L(1) = L_1
	\eqnd
so that $L_0$ and $L_1$ are transverse to $X$.
By definition, the Hamiltonian vector field $X_F$ is, outside some compact region, equal to $\beta Z$ for some constant $\beta \geq 0$ (Definition~\ref{defn. non negative wrapping}).

\begin{notation}
[$n_\cL^\chi(y)$ and $c^\chi_{\cL}$]
\label{notation. c(cL)}
Let $L(t), \, t \in [0,1]$ be a Hamiltonian
isotopy generated by a Hamiltonian $F$ such that $F(r,y) = \theta r$
on the cylindrical regions of $L_0$ and $L_1$ for a positive constant $\theta > 0$.

In what follows, we denote by $c^\chi_{\cL} \in CF(L_0,L_1)$ the continuation element induced by the isotopy $\cL$ and a choice of elongation function $\chi$. $c^\chi_{\cL}$ is obtained by counting holomorphic disks with one boundary puncture, and with moving boundary condition dictated by $\chi$ and $\cL$. See Figure~\ref{figure. continuation disk} and Section~\ref{section. continuation maps}.
We let $n_\cL^\chi(y)$ denote the number of such disks with output $y$.
\end{notation}

The idea of the proof of Theorem~\ref{thm. continuation maps are equivalences of twisted complexes} is to replace the given isotopy $\cL = \{L(t)\}_{t \in [0,1]}$ by another isotopy ${\cL'} = \{L'(t)\}_{t \in [0,1]}$ satisfying the following properties:
\enum[(i)]
\item \label{item. isotopy doesnt change twisted complex} The isotopy preserves the twisted complexes $(X \cap L_0, D_0)$ and $(X \cap L_1, D_1)$---that is, $(X \cap L(t),D') = (X \cap L'(t),D')$ for $t=0,1$.
\item \label{item. isotopy doesnt change cF^1} $\cF^1(c^{\chi}_{{\cL'}})$  is homotopic to $\cF^1(c^{\chi}_{\cL})$ (these maps have the same domain and codomain in light of \eqref{item. isotopy doesnt change twisted complex}), and
\item \label{item. isotopy is compactly supported} $L'(1)$ can be isotoped back to $L(0)$ via a {\em compactly supported} Hamiltonian isotopy.
\enumd
Property \eqref{item. isotopy is compactly supported} allows us to prove that the map $\cF^1(c^{\chi}_{{\cL'}})$ admits a homotopy inverse, essentially constructed by counting continuation strips induced by the (compactly supported!) reverse isotopy of ${\cL'}$. Theorem~\ref{thm. continuation maps are equivalences of twisted complexes} will then follow from \eqref{item. isotopy doesnt change cF^1} .

By tracing through the definition of the Abouzaid functor in Construction~\ref{construction. abouzaid functor on objects}, we see that the map
$
\cF^1: CF^*(L_0,L_1) \to \Tw C_*\cP(X)
$
is given by
	$$
\cF^1([y]) = \bigoplus_{x, x'} (-1)^{|y| + (|x|+1)(|y| + |x'|)}
\ev_*([\overline{\cH}(x,y,x')])
	$$
for $y \in \cX(L_0, L_1)$ and $x\in X \cap L_0$, $x' \in X \cap L_1$. (See \cite[(2.27)]{abouzaid-loops} for more details on this formula; here, $[\overline{\cH}(x,y,x')]$ is a choice of fundamental class for $\overline{\cH}(x,y,x')$.)
The map $c_*: (X \cap L_0, D_0) \to (X \cap L_1 ,D_1)$ is defined by applying $\cF^1$ to the continuation element $c^\chi_{\cL}$, so to prove Theorem~\ref{thm. continuation maps are equivalences of twisted complexes}, we must prove that $\cF^1 (c^\chi_{\cL})$ is an equivalence in
$\Tw C_*\cP(X)$.

By the definition of $c^\chi_{\cL}$ (Notation~\ref{notation. c(cL)}), we have
\begin{eqnarray}
\cF^1 (c^\chi_{\cL}) & = &\cF^1\left(\sum_{y \in L_0 \cap L_1; |y| = 0} n_\cL^\chi(y) \langle y \rangle\right)\nonumber \\
& = &\sum_{y \in L_0 \cap L_1; |y| = 0} \bigoplus_{x, x'}(-1)^{|y| + (|x|+1)(|y| + |x'|)}n_\cL^\chi(y)
 \ev_*([\overline{\cH}(x,y,x')])\nonumber\\
& = &  \bigoplus_{x, x'}(-1)^{|y| + (|x|+1)(|y| + |x'|)} \ev_*\left(\sum_{y \in L_0 \cap L_1; |y| = 0}n_\cL^\chi(y)
 [\overline{\cH}(x,y,x')]\right). \label{eqn. F1 definition worked out}
 \end{eqnarray}

\begin{figure}[ht]
    \eqnn
			\xy
			\xyimport(8,8)(0,0){\includegraphics[width=2in]{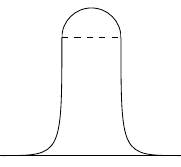}}
					,(4,8)*{\cL}
					,(4,5.8)*{y}
					,(5.8,3.4)*{L_0}
					,(2.3,3.4)*{L_1}
					,(7,0.8)*{x}
					,(1.1,0.8)*{x'}
					,(4,0.2)*{X}
			\endxy
    \eqnd
    \begin{figurelabel}\label{figure.disk-strip-glued}
    The $\cF^1$ term of the Abouzaid functor applied to a continuation element.
    \end{figurelabel}
\end{figure}

We first introduce the following notation.
\begin{notation}
Let $\cM(D^2 \setminus \{z_0\} ; \cL^\chi, y)$ be the moduli space counting once-punctured disks with out put $y$ and boundary dictated by $\cL$ and $\chi$, as in Notation~\ref{notation. c(cL)}. Then let
    \eqn\label{eqn. fiber product hashtag M Hbar}
    \cM(D^2 \setminus \{z_0\} ; \cL^\chi, y) \# \overline{\cH}(x,y,x'):=
    \cM(D^2 \setminus \{z_0\} ; \cL^\chi, y) {}_{\mathop{ev}_{z_0}} \times_{\mathrm{ev}_{z_1}} \overline{\cH}(x,y,x').
    \eqnd
\end{notation}

Then by definition, we have the following
(see Figure~\ref{figure.disk-strip-glued}.):

\begin{prop}
	\eqn\label{eqn. F^1 of c computation 1}
\ev_*\left(\sum_{y \in L_0 \cap L_1; |y| = 0}n_\cL^\chi(y)
[\overline{\cH}(x,y,x')] \right) =
\ev_*\left(
	\sum_{y \in L_0 \cap L_1; |y| = 0}
	[\cM(D^2 \setminus \{z_0\} ; \cL^\chi, y) \# \overline{\cH}(x,y,x')]
	\right).
	\eqnd
\end{prop}

Our next goal is to prove that the righthand side of \eqref{eqn. F^1 of c computation 1} is equal to a chain obtained by pushing forward the fundamental class of a different moduli space. Let us define the moduli space.
\begin{notation}
We denote:
\begin{itemize}
\item The lower semi-disc by
	$$
D^2_-= \{z \in D^2 \mid \text{Im } z \leq 0 \}
	$$
\item by
	\eqnn
	\del_+ D^2_- \cong [-1,1]
	\eqnd
	the part of the boundary of $\del D^2_-$ with $\text{Im }z = 0$, and
\item we write
	\eqnn
	\del_- D^2_- = \del D^2_- \setminus\text{Int } \del_+ D^2_- = \{|z|=1 \, \& \, \text{Im}z \leq 0 \}
	\eqnd
	where
$\text{Int } \del_+ D^2_- \cong (-1,1)$ under the above identification.
\end{itemize}
\end{notation}
\begin{notation}[$\cM(D^2_- ; \cL^\chi, x,x')$]\label{notation. M D minus}
We consider the equation
	\eqn\label{eq:onD2-}
\begin{cases}
u: D^2_- \to M \\
\delbar_J u = 0, \\
\int_{D^2_-} |du|^2 < \infty,\\
u(z) \in X & \text{for } \, z \in \del_- D^2_-\\
u(z) \in L_{\chi(z)}  & \text{for } \, z \in \del_+ D^2_-
\end{cases}
	\eqnd
where $\chi: {\del_+D^2_-} \to [0,1]$ is a monotone function such that $\chi(z) = 1$ and $\chi(z) =0$ near
$-1, \, 1 \in  {\del_+ D^2_-}$ respectively. We let
	\eqnn
	\cM(D^2_- ; \cL^\chi, x,x')
	\eqnd
denote the moduli space of solutions to \eqref{eq:onD2-} for which $u$ limits to $x \in X \cap L_0$ and $x' \in X \cap L_1$ in the obvious way.
\end{notation}
\begin{lemma}\label{lemma. concatenated moduli is continuation moduli}
For all $x, x'$, we have that
	\eqnn
	\ev_*\left(
	\sum_{y \in L_0 \cap L_1; |y| = 0}
	[\cM(D^2 \setminus \{z_0\} ; \cL^\chi, y) \# \overline{\cH}(x,y,x')]
	\right)
	=
	\ev_*\left([
	\cM(D^2_- ; \cL^\chi, x,x')]
	\right].
	\eqnd
\end{lemma}

\begin{proof}
We will smoothen corners of the elements in the fiber product
$\cM(D^2 \setminus \{z_0\} ; \cL^\chi, y) \# \overline{\cH}(x,y,x')$ and
construct a compact one dimensional moduli space
such that the fiber product becomes a part of its boundary.
Explanation of this construction is in order.

We start with explicit description of the domain curves of the two
moduli spaces appearing in the fiber product \eqref{eqn. fiber product hashtag M Hbar}.
\begin{notation}[$\Theta_-$ and $Z$]\label{notation. theta minus}
We denote by $\Theta_-$ the domain (equipped with the strip-like coordinates)
of the relevant moduli spaces, and denote by $\Theta_- \# Z$ the nodal curve obtained by the obvious grafting. (See Figure \ref{figure.disk-strip-glued} for the image of the grafted domain.)
We mention that we have conformal equivalences $\Theta_- \cong D^2 \setminus \{z_0\}$ and
$Z \cong D^2 \setminus \{z_0,z_1, z_2\}$. We take the following explicit model for $\Theta_-$:
Consider the domain
$$
\{z \in \mathbb{C} \mid |z| \leq 1, \, \mathrm{Im } z \geq 0\} \cup \{z \in \mathbb{C} \mid |\mathrm{Re } z| \leq 1,
\, \mathrm{Im } z \leq 0\} \label{eq:Theta-}
$$
and take its smoothing around $\mathrm{Im } z = 0$ that keeps the reflection symmetry about the $y$-axis
of the domain. Then we take
\eqn
Z = \{z \in \mathbb{C} \mid 0 \leq \mathrm{Im } z \leq 1\} \setminus \{(0,1)\}. \label{eq:Z}
\eqnd
Again we equip $Z$ with a strip-like coordinate at $z = (0,1)$ that keeps the reflection symmetry.
\end{notation}

\begin{notation}[$\#$ and $\#_r$]
We denote by $\Theta_- \# Z$ the nodal curve obtained from $\Theta_-$ and $Z$
by the obvious grafting. Note that we have conformal equivalences
$\Theta_- \cong D^2 \setminus \{z_0\}$ and $Z \cong D^2 \setminus \{z_0,z_1,z_2\}$.
Using a given strip-like coordinate around $\{(0,1)\}$ in $Z \setminus \{(0,1)\}$,
we construct one-parameter family of the glued domains which we denote by
	$$
\Theta_- {\#}_r Z
	$$
for all  sufficiently large $r > 0$. We perform this grafting while maintaining the above-mentioned reflection symmetry.
\end{notation}

By gluing the defining equations of the two moduli spaces on the glued domain
	$
\Theta_- {\#}_r Z
	$
for all sufficiently large $r > 0$ and then deforming the domain to the lower semi-disc
$D^2_-$ (adjusting the positively moving boundary condition accordingly), we arrive at \eqref{eq:onD2-}.
Thus, by a standard gluing-deformation and compactness argument, the result follows.
\end{proof}

\begin{remark}\label{remark. D2minus moduli is trip moduli}
Note that the moduli space of solutions to~\eqref{eq:onD2-} is equivalent to the moduli space defining a continuation map using holomorphic {\em strips} (not disks) with moving boundary conditions. This is
because under the conformal equivalence $D^2_- \setminus \{1,-1\} \cong D^2 \setminus \{1, -1\} \cong \RR \times [0,1]$,
the equation~\eqref{eq:onD2-} is equivalent to

   \eqn\label{eq:hcL}
        \begin{cases}
        {\frac {\partial u}{\partial\tau}} + J_{(\rho(\tau),t)}
         {\frac {\partial u}{\partial t}}=0 \\
        u(\tau ,0)\in X,\;\; u(\tau ,1)\in L_{\rho(1-\tau)}.
        \end{cases}
    \eqnd
after a suitable choice of nondecreasing elongation function $\rho:\RR \to [0,1]$ satisfying
\eqn\label{eq:rho}
\rho(\tau) =  \begin{cases} 1 \quad & \text{for } \, \tau \geq a \\
0 \quad & \text{for } \, \tau \leq 0
\end{cases}
\eqnd
 for some $a$.
\end{remark}

\begin{remark}
By assumption on the isotopy, the isotopy is non-negative outside a compact set $K$.
 By compactness of $X$, we may enlarge $K$ so that $X \subset \Int K$
without loss of generality.
This allows us to apply the strong maximum principle
to derive
	\eqn\label{eqn. C0 estimate inside K}
\Image u \subset K.
	\eqnd
(See \cite[section 13]{bko:wrapped} for similar application of strong maximum principle.)
\end{remark}

Next we need to establish a uniform energy bound.
Let $h_X: X \to \RR$ and $h_s: L_s \to \RR$ be Liouville primitives of $X$ and $L_s$
respectively. (So for example, $dh_X = \theta|_X$. For a given exact Lagrangian isotopy generated by a Hamiltonian $H$,
a smooth family of Liouville primitives of $L_s$ is given by
	\eqn\label{eq:s primitive}
h_s = h_0 + \int_0^s (\langle \theta, X_{H}\rangle - H)\circ \phi^t_H \  dt.
	\eqnd
(See \cite[Proposition 3.4.8]{oh:book1}, for example.)
Then we have the following energy identity.
\begin{lemma}\label{lem:action-diff} Let $X \subset \Int K$, let $\cL$ be a nonnegative (near infinity) isotopy,  and let
$H$ be a Hamiltonian, linear outside $K$, generating the isotopy $\cL$.  Then
for any finite energy solution $u$ of~\eqref{eq:hcL}  with $u(\pm\infty) \equiv x_\pm$, we have
\begin{eqnarray}
E_J(u)  = -(h_X(x_+) - h_X(x_-)) + (h_0(x_+) - h_1(x_-)) - \int_{-\infty}^\infty \rho'(\tau)H(u(1- \tau,1))\, d\tau.
\label{eq:EJu-continuation2}
\end{eqnarray}
where  $x_- \in X\cap L_1$ and $x_+ \in X\cap L_0$. In particular,
\begin{equation}\label{eq:energybound}
E_J(u) \leq C(X,\cL;K)
\end{equation}
for some constant $C(X,\cL;K)$ depending only on $X,\cL$ and $K$ but independent of $u$.
\end{lemma}
\begin{proof}
While what follows mostly parallels the proof of \cite[Lemma 7.3]{AOO}, there are two differences:
\begin{itemize}
\item The current case treats the general case of Lagrangian branes in general Liouville sectors,
while \cite[Lemma 7.3]{AOO} treats the case of the zero section $X = Q$ and the fiber $L=0 = T_q^*Q$
for which we can take $h_X = 0 = h_0$. Because of this, the statement of the current lemma is more general
including \cite[Lemma 7.3]{AOO} as a special case.
\item There are differences in the details of the proof due to differences in conventions.
\end{itemize}
Because of these, we include a complete proof here for the reader's convenience.

Consider the following family of the action functional
	\eqn
\mathcal{A}_{X,L_s}(\gamma)= - \int_{[0,1]} \gamma^*\theta + h_X(\gamma(1))- h_s(\gamma(0)),
	\eqnd
where $\gamma \in \mathcal{P}(X,L_{s})$.
Obviously, we have
	$$
\cA_{X,L_0} (x_+) - \cA_{X,L_1} (x_-)
= \int_{-\infty}^\infty \frac{d}{d\tau}\left(\cA_{X,L_{\rho(1-\tau)}}(u(\tau))\right) \, d\tau.
	$$
We derive
\begin{eqnarray}\label{eq:ddtauAL}
\frac{d}{d\tau} (\cA_{X,L_{\rho(1-\tau)}}(u(\tau)))
& = & \frac{d}{d\tau}\left(\int_{[0,1]} - (u(\tau))^*\theta - h_X(u(\tau,0))\right)
- \frac{d}{d\tau}\left(h_{\rho(1-\tau)}(u(\tau,1))\right) \nonumber\\
& = & \int_{[0,1]} \Big|\dudtau \Big|^2_{J_t}\,dt
- \left\langle \theta, \frac{\del u}{\del \tau}(\tau,1) \right\rangle - \frac{d}{d\tau}\left(h_{\rho(1-\tau)}(u(\tau,1))\right)
\end{eqnarray}
using the first variation of the action functional with free boundary
condition (see \cite[Equation (12.1.1)]{oh:book2}),
together with the fact that $u$ is a solution of \eqref{eq:onD2-}.
Therefore by combining the above calculations and integrating over $-\infty < \tau < \infty$, we have obtained
\begin{eqnarray*}
\cA_{X,L_0} (x_+) - \cA_{X,L_1} (x_-) = E_J(u)
 - \int_{-\infty}^\infty \left\langle \theta, \frac{\del u}{\del \tau}(\tau,1) \right\rangle\, d\tau
- (h_1(x_-) - h_0(x_+)).
\end{eqnarray*}

Noticing that the boundary condition $u(\tau,1)\in L_{\rho(1-\tau)}$ implies that
$u(\tau,1)=\phi^{\rho(1-\tau)}_H(v(\tau))$ for some curve $v(\tau) \in L_0$, we can write
	$$
h_{\rho(1-\tau)}(u(\tau,1))= \widetilde h_{\rho(1-\tau)}(v(\tau))
	$$
for $\widetilde h_{\rho(1-\tau)} = h_{\rho(1-\tau)} \circ \phi^{\rho(1-\tau)}_H: L_0 \to \RR$.
Then we compute
\begin{align}
\frac{d}{d\tau} (h_{\rho(1-\tau)}&(u(\tau,1)))  =  d \widetilde h_{\rho(1-\tau)}\left(\frac{d v}{d \tau}\right) -\rho'(1-\tau)\left(\frac{d \widetilde h_s}{d s}\right)\bigg|_{s=\rho(1-\tau)}(v(\tau)) \\
& = d h_{\rho(1-\tau)}\circ d \phi^{\rho(1-\tau)}_H\left(\frac{d v}{d \tau}\right) - \rho'(1-\tau)\left( \langle \theta, X_H \rangle(\phi^{\rho(1-\tau)}_H(v(\tau))) - H(\phi^{\rho(1-\tau)}_H(v(\tau)))\right).\nonumber
\end{align}
It follows from the definitions that
	$$\frac{\del u}{\del \tau}(\tau,1)= d \phi_H^{\rho(1-\tau)}\left(\frac{d v}{d \tau}\right) - \rho'(1-\tau) X_H(u(\tau,1)).$$
Plugging this in the previous equation and using the definition of Liouville primitive we obtain
	\eqnn
\frac{d}{d\tau} (h_{\rho(1-\tau)}(u(\tau,1))) =
\left\langle \theta, \frac{\del u}{\del \tau}(\tau,1) \right\rangle + \rho'(1-\tau)H(u(\tau,1)).
	\eqnd
Substituting this into \eqref{eq:ddtauAL}, we obtain
	$$
\frac{d}{d\tau} (\cA_{H;\rho(1-\tau)}(u(\tau)))= \int_{[0,1]} \Big|\dudt \Big|^2_{J_t}\,dt - \rho'(1-\tau)H(u(\tau,1))
	$$
which is equivalent to
	$$
\int_{[0,1]} \Big|\dudt \Big|^2_{J_t}\,dt = \frac{d}{d\tau} (\cA_{H;\rho(1-\tau)}(u(\tau))) +  \rho'(1-\tau)H(u(\tau,1)).
	$$
By integrating this over $\tau \in \RR$, we obtain
\beastar
E_J(u) & = &
\cA_{X,L_1} (x_-) - \cA_{X,L_0} (x_+)
+ \int_{-\infty}^\infty \rho'(1-\tau)H(u(\tau,1))\, d\tau\\
& = & \cA_{X,L_1} (x_-) - \cA_{X,L_0} (x_+)
- \int_{-\infty}^\infty \rho'(\tau)H(u(1- \tau,1))\, d\tau.
\eeastar
Then we evaluate
\beastar
\cA_{X,L_1} (x_-) - \cA_{X,L_0} (x_+) & = & -(h_1(x_+) - h_X(x_+)) + (h_0(x_-) - h_X(x_-))\\
& = & (h_X(x_+) - h_X(x_-)) - (h_0(x_+) - h_1(x_-)).
\eeastar
Combining the two, we have proved \eqref{eq:EJu-continuation2}.

Finally we prove the uniform energy bound \eqref{eq:energybound}.
Recall the support bound $\Image u \subset K$~\eqref{eqn. C0 estimate inside K}.
Therefore
	$$
|h_1(x_-) - h_0(x_-)| \leq \max_K|\langle \theta, X_{H}\rangle|+\|H\|_K
	$$
where $\|H\|_K : = \sup_{x \in K} |H(x)|$.
Next we get the bound
	$$
\left|\int_{\infty}^\infty \rho'(\tau) H(u(1-\tau,1))\, d\tau\right| \leq
\int_{-\infty}^\infty \rho'(\tau) |H(u(1- \tau,1))|\, d\tau \leq \|H\|_K.
	$$
Combining these with the energy identity \eqref{eq:EJu-continuation2}, we derive
	\eqn\label{eq:EJu<max}
E_J(u) \leq \max_K|\langle \theta, X_{H}\rangle|+ 2\|H\|_K + (\max h_X - \min h_X) = : C(X,\cL, H; K).
\eqnd
Finally we take
\eqn\label{eq:CXcLK}
C(X,\cL; K): = \inf_H \{ C(X,\cL, H; K) \mid L(t) = \phi_H^t(L_0)\, \forall t \in [0,1]\}
\eqnd
where the $\inf$ is taken over all Hamiltonians $H$ that are linear outside $K$.
This finishes the proof.
\end{proof}

\begin{remark}\label{rem:CXcLK} We examine
the nature of the upper bound $C(X,\cL, H; K)$ given in \eqref{eq:EJu<max}.
Clearly the term $(\max h_X - \min h_X)$ does not depend on $\cL$ but only on $X$.
Under the hypotheses of Lemma \ref{lem:action-diff}, we have the support property $\Image u \subset K$
(see \eqref{eqn. C0 estimate inside K}). This implies that we have only to examine the constant
\eqn\label{eq:maxKXHH}
\max_K|\langle \theta, X_{H}\rangle|+ 2\|H\|_K
\eqnd
when we take the infimum of $C(X,\cL, H; K)$ as we vary over $\cL$ \emph{that is linear outside $K$.}
We also remark that, for fixed $K$, the map $\cL \mapsto C(X,\cL; K)$ is continuous with respect to the fine $C^1$ topology on the space of isotopies $\cL$.
\end{remark}

\subsubsection{Proof of Theorem~\ref{thm. continuation maps are equivalences of twisted complexes}}

This section will be occupied by the proof of
Theorem~\ref{thm. continuation maps are equivalences of twisted complexes}. We choose $R > 0$ sufficiently large
so that $K \subset r^{-1}((-\infty,R/2))$.
It is easy to see that we can deform the isotopy $\cL$ to ${\cL'} = \{L'(t)\}$
via $\cL^{\text{\rm para}} = \{\cL_s\}_{s \in [0,1]}$ with $\cL_s= \{L_s(t)\}_{t \in [0,1]}$
and $\cL_0 = \cL$, $\cL_1 = {\cL'}$
so that the following hold:
\begin{itemize}
\item the hypotheses of Lemma \ref{lem:action-diff} holds for all $s \in [0,1]$,  and
\item  we have
	$$
L'(t) \cap r^{-1}((-\infty,R]) = L(t) \cap r^{-1}((-\infty,R]),
	$$
\item
	$$
L'(t) \cap \left(M \setminus r^{-1}((-\infty, 2R])\right)
= L_0 \cap \left(M \setminus r^{-1}((-\infty, 2R])\right),
	$$
\end{itemize}
and we suitably interpolate on the region $r^{-1}([R,2R])$. Therefore it follows from Remark \ref{rem:CXcLK}
that the constant $C(X,\cL, H; K)$ for $\cL_s$ is uniformly bounded over $s \in [0,1]$ by considering
$$
C(X, \cL^{\text{\rm para}};K): = \sup_{s \in [0,1]} C(X,\cL_s;K).
$$
In particular $L'(t) \cap X = L(t) \cap X$ for all $t \in [0,1]$ and
	$$
L'(1) \cap \left(M \setminus r^{-1}((-\infty, 2R])\right)
= L_0 \cap \left( M \setminus r^{-1}((-\infty, 2R])\right).
	$$
Furthermore $L'(t) \equiv L(t)$ on the region $r^{-1}([R/2,R])$ and so
the strong maximum principle can be applied which prevents
any trajectory associated to ${\cL'}$ continued from $\cM(X, \cL)$
from penetrating into $r^{-1}([R/2,R])$. Furthermore we have
	$$
X \cap L(t) = X \cap L'(t)
	$$
and we may assume this intersection to be contained in the compact region $K$.

Define twisted complexes
\begin{eqnarray*}
T &= & (X \cap L_0, D_0), \hskip0.1in {\quad D = \{\ev_*([\cH(X,L_0;x_0,x_1)])\}_{x_0,x_1 \in X \cap L_0},} \\
T'& = & (X \cap L_1, D_1), {\quad D' = \{\ev_*([\cH(X,L_1;x_0',x_1')])\}_{x_0',x_1' \in X \cap L_1}}.
\end{eqnarray*}
and two morphisms between them by
	$$
S =\ev_*\left(\sum_{x \in X \cap L_0, x' \in X \cap L_1} [\cM(X, \cL^\rho;x,x')]\right),
\qquad S' =\ev_*\left(\sum_{x \in X \cap L_0, x' \in X \cap L_1} [\cM(X, {\cL'}^{\rho};x,x')]\right)
	$$
Here we use the moduli space of solutions of \eqref{eq:hcL} with elongated
(moving) Lagrangian boundary and asymptotic boundary conditions.

\begin{proposition}\label{prop:evcL=evcL'}
	$$
[S'] = [S]
	$$
in $\mu^1_{\mathrm{Tw}(C_*\cP)}$-cohomology.
\end{proposition}

Assuming this proposition for the moment, we proceed with:

\begin{proof}[Proof of Theorem~\ref{thm. continuation maps are equivalences of twisted complexes}]
We deform $L'(1)$ to $L_0$ via a compactly supported isotopy $\overline{\cL'}$.
Since a compactly supported isotopy can be composed with its inverse so that their
composition (through an isotopy of compactly supported isotopies) is isotopic to the constant isotopy $\widehat L_0$, we obtain
	$$
\bigoplus_{x,x'}\ev_*([\cM(X,{\cL'}\# \overline{\cL'};x,x')]) \sim \bigoplus_{x,x' \in X \cap L_0}
\ev_*([\cM(X,\widehat L_0;x,x']) = id_{ (X \cap L_0, D)}.
	$$
Thus we have
\begin{eqnarray*}
id_{ (X \cap L_0, D)} & \sim & \bigoplus_{x,x'} \ev_*([\cM(X,{\cL'}\# \overline{\cL'};x,x')]) \\
& \sim &
\bigoplus_{x,z} \bigoplus_y\ev_*([\cM(X,{\cL'};x,y)])\cdot\ev_*([\cM(X,\overline{\cL'};y,z)]))\\
& = & \left(\bigoplus_{x,x'} \ev_*([\cM(X,{\cL'};x,x')])\right)\cdot
\left(\bigoplus_{z,z'} \ev_*([\cM(X, \overline{\cL'};z,z')])\right).
\end{eqnarray*}
So the morphism $\left(\bigoplus_{x,x'} \ev_*([\cM(X,{\cL'};x,x')])\right)$ admits a right homotopy inverse. Tracing through the same work for $\overline{\cL'} \# {\cL'}$ shows that the morphism admits a left homotopy inverse as well; that is, the morphism is an equivalence in $\Tw C_*\cP(X)$. On the other hand, we have
\begin{align}
\left(\bigoplus_{x,x'} \ev_*([\cM(X,{\cL'};x,x')])\right)
&= \left(\bigoplus_{x,x'} \ev_*([\cM(X,{\cL'};x,x')])\right) \nonumber \\
&=	\ev_*\left([
	\cM(D^2_- ; \cL^\chi, x,x')]
	\right] \nonumber \\
& =
	\ev_*\left(
	\sum_{y \in L_0 \cap L_1; |y| = 0}
	[\cM(D^2 \setminus \{z_0\} ; \cL^\chi, y) \# \overline{\cH}(x,y,x')]
	\right)  \nonumber\\
&=
	\ev_*\left(\sum_{y \in L_0 \cap L_1; |y| = 0}n_\cL^\chi(y)
	[\overline{\cH}(x,y,x')] \right) \nonumber \\
&=
	\cF^1 (c^\chi_{\cL}) .\nonumber
\end{align}
The first equality follows from the uniform energy bound (depending only on the behavior of $\cL$ and ${\cL'}$ inside $K$), guaranteeing that the count of continuation maps yields equivalent maps. The next equalities are given by
Remark~\ref{remark. D2minus moduli is trip moduli},
by Lemma~\ref{lemma. concatenated moduli is continuation moduli},
by \eqref{eqn. F^1 of c computation 1},
then by \eqref{eqn. F1 definition worked out}.

Combining the above, we have proved
	$
\cF^1(c^\chi_{\cL})
	$
is an equivalence. This finishes the proof of
Theorem~\ref{thm. continuation maps are equivalences of twisted complexes}.
\end{proof}

\begin{proof}[Proof of Proposition \ref{prop:evcL=evcL'}]
We consider one-parameter family $\cL_s$ (i.e. a homotopy of isotopies) with $0 \leq s \leq 1$
which are fixed on $K$ but deforms outside $r^{-1}((-\infty,R/2])$ and consider the parameterized
moduli space
	$$
\cM^{\mathrm{para}}_{(0)}(X, \{\cL_s\};x^-,x^+)
= \coprod_{s \in [0,1]'|x^-|=|x^+|} \{s\} \times \cM(X, \cL_s;x^-,x^+)
	$$
for the pairs $(x^-,x^+)$ with $|x^-| = |x^+|$ with $\cL_0 = \cL$ and $\cL_1 = {\cL'}$.
Here we introduce the moduli spaces
$$
\cM^{\mathrm{para}}_{(k)}(X, \{\cL_s\};x^-,x^+)
$$
in general where the integer $k$ appearing in the subindex $(k)$ of the moduli space stands for
the degree of the relevant operators which is the same as $|x^-| - |x^+|$. This is
one smaller than the dimension of the relevant parameterized moduli space.

We know that:
\begin{itemize}
	\item the relevant Hamiltonians defining $\cL_s$ are only non-linear in a compact region contained in $r^{-1}((-\infty,2R])$,
	\item there exists a uniform energy bound, and
	\item there is no bubbling,
\end{itemize}
so the boundary of
$\cM^{\mathrm{para}}_{(0)}(X, \{\cL_s\};x^-,x^+)$ consists of the types
\begin{eqnarray*}
&{}& \cM^{\mathrm{para}}_{(-1)} (X, \{\cL_s\};x^-,y) \# \cH(X,L_0; y,x^+), \\
&{}& \cH(X,L_1; x^-,z) \# \cM^{\mathrm{para}}_{(-1)}(X, \{\cL_s\};z,x^+), \\
&{}& \cM^{\mathrm{para}}(X, \{\cL_s\};x^-,x^+)|_{s=0}, \\
&{}& \cM^{\mathrm{para}}_{(-1)}(X, \{\cL_s\};x^-,x^+)|_{s=1}.
\end{eqnarray*}
Therefore we have
\begin{eqnarray}\label{eq:homotopy-relation}
&{}&\ev_*([\cM(X, {\cL'};x^-,x^+)]) -\ev_*([\cM(X, \cL;x^-,x^+)]) \nonumber\\
& = & {\bigoplus_{y \in \in X \cap L_0;|y| = |x^+|+1}}\ev_*\left(\cM^{\mathrm{para}}_{(0)}(X, \{\cL_s\};x^-,y) \# \cH(X,L_0;y,x^+)\right)\nonumber\\
&{}& + {\bigoplus_{z \in \in X \cap L_1;|z| =|x^-|-1}}\ev_*\left(\cH(X,L_1;x^-,z) \# \cM^{\mathrm{para}}_{(0)}(X, \{\cL_s\};z,x^+)\right).
\end{eqnarray}

Now we introduce a collection of moduli spaces
	$$
\cM^{\mathrm{para}}_{(-1)}(X, \{\cL_x\})
 =  \{\cM^{\mathrm{para}}(X, \{\cL_s\};x,x')\}_{s\in [0,1], x,x';|x'|=|x|-1}
	$$
and define a 2-cochain
	$$
\mathfrak{H}: = \left\{\ev_*\left([\cM^{\mathrm{para}}_{(-1)}(X, \{\cL_s\};x,x')]\right)\right\}_{s\in [0,1], x,x';|x'|=|x|-1}.
	$$
	
Then {\eqref{eq:homotopy-relation}} gives rise to
\begin{eqnarray} \label{eq:homotopy}
&{}&\ev_*([\cM_{(0)}(X, {\cL'})])- \ev_*([\cM_{(0)}(X, \cL)])\nonumber\\
& = &\ev_*\left([\cM^{\mathrm{para}}_{(-1)}(X, \{\cL_s\})]\right)\cdot\ev_*([\cH(X, L_0)])\nonumber\\
&{}& -\ev_*([\cH(X,L_1)]) \cdot\ev_*\left([\cM^{\mathrm{para}}_{(-1)}(X,\{\cL_s\})]\right)
\end{eqnarray}
which can be rewritten into
	$$
[S'] - [S] = \mu^1_{\mathrm{Tw}(\cP)}([\mathfrak{H}]):
	$$
By the definition of the $\mu^1$ on the twisted complex, we have
	$$
\mu^1_{\mathrm{Tw}(\cP)}(\mathfrak{H}) = \mu^1_{\cP}([\mathfrak{H}]) + \mu^2_{\cP}([\mathfrak{H}], D_0)
+ \mu^2_{\cP}(D_1, [\mathfrak{H}]).
	$$
In the current case, we have $\mu^1_{\cP}([\mathfrak{H}]) = \partial ([\mathfrak{H}]) = 0$ since
$\mathfrak{H}$ is a finite sum of  zero-dimensional chains. This finishes the proof.
\end{proof}

\clearpage
\section{The equivalence  \texorpdfstring{$\Tw\cW \simeq C_* \cP$}{TwW - C.P}}

The main result of this section is

\begin{theorem}\label{theorem. cW to cP}
The natural transformation $\Tw \cW \circ \DD \to \Tw C_* \cP$ from Corollary~\ref{cor. natural transformation} is a natural equivalence. That is, for every smooth $j: |\Delta^n| \to B\diff(Q)$, the map $\Tw \cW_{\DD \circ j} \to \Tw C_*\cP_j$ is an equivalence of $A_\infty$-categories.
\end{theorem}

The proof requires some preliminary results that verify:
\enum
	\item An isomorphism between two definitions of wrapped Floer cohomology---one using a colimit indexed over a non-negative sequence of wrappings (as in~\cite{abouzaid-seidel}) and the other using a Hamiltonian quadratic near infinity (as in~\cite{abouzaid-loops}). This is Proposition~\ref{prop:quad=cofinal}.
	\item That quadratically wrapped Floer cohomology does not change under continuation maps of linear-near-infinity non-negative Hamiltonians. This is Lemma~\ref{lemma.linear wrapping leaves quadratic wrapping unchanged}.
	\item A compatibility between the non-wrapped Abouzaid map (Proposition~\ref{prop. O to cP natural}) and the quadratically wrapped Abouzaid map from~\cite{abouzaid-loops} when mediated by the isomorphism from the just-mentioned Proposition~\ref{prop:quad=cofinal}. This compatibility is expressed in Corollary~\ref{cor. abouzaid constructions compatible cotangent}.
\enumd

These ingredients will be mixed in Section~\ref{section. proof of cW cP theorem} to give a proof of Theorem~\ref{theorem. cW to cP}.

As it turns out, the above ingredients can be proven in large generality, so that is what we will do; moreover, we only need to verify these ingredients in a single fiber of a Liouville bundle, so in what follows, we will fix some Liouville sector $M$.

\begin{remark}\label{remark. quadratic to localizing}
Let us comment on the proof of Theorem~\ref{theorem. cW to cP}, which shows that our family of wrapped Fukaya categories is equivalent to a family of local system categories. A non-trivial aspect of proving this equivalence is that Abouzaid's construction in~\cite{abouzaid-loops} utilized {\em quadratic} Hamiltonians to define wrappings; this is in contrast to the definition of $\cW$ in the present work (following~\cite{gps}), which is a result of {\em localizing} with respect to non-negative, {\em linear} Hamiltonian continuation maps. In particular, morphisms of $\cW$ are not so tractable using pure geometry.

Put another way, we must confront the fact that there are multiple definitions of the wrapped Fukaya category in the literature.

While possible, it is non-trivial to write down the analysis (and in particular, compactness arguments) to see that one has an $A_\infty$ {\em algebra} map between the different versions of wrapped endomorphisms, and we do not do this. Instead, we formally conclude that such an algebra map exists by the universal property of localizations---i.e., by making use of category theory. The analytical legwork, via this strategy, is reduced to checking that the underlying  map of endomorphism {\em complexes} is a quasi-isomorphism, which one can do by straightforward arguments invoking the action filtration and relating a cofinal sequence of linear wrappings to a single quadratic wrapping.

We refer the reader also to Section~2.2 of~\cite{sylvan-orlov-and-viterbo} for a separate approach.
\end{remark}

When we do apply the general results for our purposes, we will state this application as a corollary, and we will apply our general results in the following setting:

\begin{choice}[Choice for proving Theorem~\ref{theorem. cW to cP}.]\label{choice. for proving cW to cP}
For any simplex $j: |\Delta^n_e| \to B\diff(Q)$ and for any $0 \leq a \leq n$, we set $M= T^*Q_a$ to be the fiber above the $a$ the vertex of $|\Delta^n_e|$.

We also choose a point $q_a \in Q_a$ in the zero section above the $a$th vertex of $|\Delta^n_e|$. Choose also a cofinal sequence for the cotangent fiber $L = L^{(0)} = T^*_{q_a} Q_a$ (Definition~\ref{defn. cofinal wrapping}).

We can arrange so that a cotangent fiber $T^*_{q_a} Q_a$ and all its cofinal wrappings are transverse to $Q_a$ and have only a single intersection point with $Q_a$, so that the natural transformation induced by the Abouzaid map sends $T^*_{q_a} Q_a$ to the object $q_a \in \cP_j$. We will assume so.
\end{choice}

Below, $N_{A_\infty}$ refers to the $A_\infty$-nerve of an $A_\infty$ category. (See Recollection~\ref{recollection. A oo facts}\eqref{item. Aoo nerve}.)
\begin{lemma}\label{lemma. filtered Aoo}
Fix a Liouville sector $M$ and a cofinal sequence of wrappings for a brane $L = L^{(0)}$ (Definition~\ref{defn. cofinal wrapping}). Then the continuation maps induce a functor of $\infty$-categories $\ZZ_{\geq 0} \to N_{A_\infty}(\cO(M))$ as follows:
	\eqnn
	L^{(0)} \to L^{(1)} \to \ldots.
	\eqnd
\end{lemma}

\begin{proof}[Proof of Lemma~\ref{lemma. filtered Aoo}]
The continuation maps determine, for every $i \in \ZZ_{\geq 0}$, a morphism in $\cO(M)$ from the brane $L^{(i)}$ to the brane $L^{(i+1)}$; in particular, for each $i$ we have an edge in $N_{A_\infty}(\cO_j)$.  By the weak Kan property of $\infty$-categories, and because $\ZZ_{\geq 0}$ is a poset, this sequence of edges lifts to a unique (up to contractible choice) functor from the $\infty$-category $\ZZ_{\geq 0}$ to $N_{A_\infty}(\cO(M))$.
\end{proof}

\subsection{Comparing quadratic wrappings with cofinal wrappings}
\label{subse:cofinal-quadratic}

We have already shown that the colimit of a cofinal sequence of Floer cohomologies computes the hom-complex of the wrapped category $\cW$ (Lemma~\ref{lemma. hom is wrapped cohomology}). In this section, we show that this colimit also computes the quadratically wrapped Floer cohomology (Proposition~\ref{prop:quad=cofinal}).

We first set some notation.

\begin{notation}[$CF^*(L,L';H),HF^*(L,L';H)$]\label{notation. HF; H}
Let $M$ be a Liouville sector and let $L,L' \subset M$ be branes.
Fix a smooth function $H: M \to \RR$.
We denote the set of Hamiltonian chords of $H$ from $L$ to $L'$
	$$
\Chord(L,L';H) = \{ \gamma \in \cP(L,L') \mid \gamma(0) \in L, \, \gamma(1) \in L'\}
	$$
and the associated the Floer cochain complex  by
	\eqnn
	CF^*(L,L';H)
	\eqnd
whose differential is defined by considering the perturbed equation
$\delbar_{J,H}(u) = 0$, i.e.,
	\eqn\label{eq:CR-JXH}
\begin{cases}
\frac{\del u}{\del \tau} + J\left(\frac{\del u}{\del t} - X_H(t,u)\right) = 0\\
u(\tau,0) \in L, \, u(\tau,1) \in L'.
\end{cases}
	\eqnd
The cohomology of this complex (i.e., the Floer cohomology) will be denoted
	\eqnn
	HF^*(L,L';H).
	\eqnd
\end{notation}

\begin{remark}\label{remark. isomorphism between moving and unmoving complexes}
We recall that there exists a natural chain isomorphism
	$$
CF(L,L';H) \to CF(L, \phi_H^1(L');0)
	$$
induced by the correspondence
	$$
\gamma \in \cP(L,L') \mapsto \widetilde \gamma \in \cP(L, \phi_H^1(L')), \quad
\widetilde \gamma(t): = \phi_H^t(\gamma(t)),
	$$
and
$J_t \mapsto \widetilde J_t$ defined by
	$$
\widetilde J_t = (\phi_H^t)^* J_t
	$$
which transforms the equation $\delbar_{J,H}(u) = 0$ to
	$$
\delbar_{\widetilde J,0}(v)=0, \quad v(\tau,t) = \phi_H^t\left(u(\tau,t))\right).
	$$
\end{remark}
\subsubsection{Hamiltonians}

\begin{defn}[Quadratic near infinity]\label{defn. quadratic near infinity}
Let $M$ be a Liouville domain. A smooth function $H: M \to \RR$ is called {\em quadratic near infinity} if
	\eqnn
	H = \frac12 r^2
	\eqnd
outside a compact subset of $M$. (Here, we are using the coordinate $r$ from Notation~\ref{notation. liouville notation}.)
\end{defn}

\begin{remark}\label{remark. model cofinality linearly}
Our wish is to compare $CF(L,L;H)$---with $H$ quadratic near infinity---to a colimit of Floer complexes constructed from a cofinal\footnote{Definition~\ref{defn. cofinal wrapping}.} sequence $L^{(0)} \to L^{(1)} \to \ldots$.

Because each $L^{(w)}$ in this sequence is conical near infinity, we may assume that each is the result of a Hamiltonian isotopy by a linear-near-infinity Hamiltonian $F^{(w)}: M \to \RR$. In what follows, we will choose a particular model for such a sequence of Hamiltonians---namely, we will set $F^{(w)}$ to be obtained by increasing the slopes of a standard linear-near-infinity Hamiltonian (see~\eqref{eqn. rigid choice of Fwrapper}). The reader may make the necessary adjustments to the following proofs for a more general cofinal sequence of linear-near-infinity Hamiltonians. (For example, by altering~\eqref{eq:H(wrapper)} to interpolate $H$ with a given $F^{(w)}$, rather than with the sequence of $F$s we choose in~\eqref{eqn. rigid choice of Fwrapper}.)
\end{remark}

\begin{notation}[$H$ and $F$]
In this section, we will use the symbol
$H$ to denote an autonomous Hamiltonian that is quadratic near infinity (Definition~\ref{defn. quadratic near infinity}).

We will use the symbol $F$ to denote a Hamiltonian that is autonomous and
is linear near infinity, i.e.,
$F = a r + b$ outside a compact subset for some constant $a, \, b$ with $a> 0$. Note that
	\eqn\label{eqn. rigid choice of Fwrapper}
	\phi_{{v}F}^1(L) = \phi_F^{v}(L)
	\eqnd
is still linear near infinity, and in particular outside $\{r \leq R_K\}\supset K$ for any $R_K$ large enough.
\end{notation}

\begin{notation}[$HF^*_{\text{\rm quad}}$]\label{notation. HF quad}
Given a quadratic-near-infinity Hamiltonian $H$, we define the notation
    $$
    HF^*_{\text{\rm quad}}(L,L') = HF^*(L,\phi_{H}^1(L');0).
    $$
Note that the dependence on $H$ is suppressed on the left-hand side. (See also Notation~\ref{notation. HF; H} for the right-hand side.)
\end{notation}

\begin{lemma}\label{lem:lowerbound}
Fix a Liouville embedding $\iota:[0,\infty) \times \del_\infty M \to M$ and set $r = e^s$ as in Notation~\ref{notation. liouville notation}.
Let $H: M \to \RR$ be quadratic near infinity (Definition~\ref{defn. quadratic near infinity}).
Then there exists some constant $C = C(\iota,H)> 0$ such that
	\eqn\label{eq:H-lambdaXH}
H - \theta(X_{H}) \geq -C.
	\eqnd
\end{lemma}
\begin{proof} Outside a large compact subset of $M$, we have
	$$
\theta(X_{H}) = \theta(r X_r) = rdr(-J X_r) = rdr\left(\frac{\del}{\del r}\right)= r
	$$
so $H - \theta(X_{H}) = \frac{r^2}{2} - r > 0$ as long as $r > 2$.
The lemma follows because the (closure of) the complement of the image of $\iota$ is compact.
\end{proof}

\begin{notation}[The chain maps $\phi$]\label{notation. phi continuation maps}
Consider the interpolating homotopy
	\eqnn
	(1-s) {v}F + H,
	\qquad
	s
	\in [0,1]
	\eqnd
through Hamiltonian that induce non-negative wrappings near infinity. By the $C^0$ and energy estimates (Recollection~\ref{recollection. continuation maps}\eqref{item. gromov compactness for bundles}), we have induced chain maps
	$$
\phi_{{v}{v}'}: CF(L,L';{v}F) \to CF(L,L';{v}'F)
	$$
for all ${v} < {v}'$ and
	$$
\phi_{v} : CF(L,L';{v}F) \to CF(L,L';H)
	$$
for any ${v} \in \ZZ_+$. (See Notation~\ref{notation. HF; H}.)
\end{notation}

Passing to cohomology, the maps $\phi$ from Notation~\ref{notation. phi continuation maps} induce a commutative diagram
    \eqn\label{eqn. acceleration commutative diagram}
    \xymatrix{
    HF(L,L';F)  \ar[r] \ar[d] & HF(L,L';2F) \ar[r]\ar[d] & \cdots \ar[r] &
    HF(L,L';{v}F)  \ar[r]\ar[d] & \cdots \\
    HF(L,L';H) \ar[r] & HF(L,L';H) \ar[r] & \cdots \ar[r] &
    HF(L,L';H) & \cdots }
    \eqnd
and hence a homomorphism
	$$
\varphi_\infty: \colim_{{v} \to \infty} HF^*(L,L';{v}F) \to HF^*(L,L';H) \cong HF^*_{\mathrm{quad}}(L,L').
	$$
(The last isomorphism is Remark~\ref{remark. isomorphism between moving and unmoving complexes}.)

The main result we seek to prove now is:

\begin{prop} \label{prop:quad=cofinal}
$\varphi_\infty$ is an isomorphism.
\end{prop}

Before proving Proposition~\ref{prop:quad=cofinal}, let us state the following consequence,
which one obtains by setting $L=L' \subset T^*Q$ to be a cotangent fiber to a point $a \in Q$.

\begin{corollary}\label{corollary. filtered to quadratic}
In the setting of Choice~\ref{choice. for proving cW to cP}, the induced map
	\eqnn
	\colim_w H^* \hom_{\cO_j}(T^*_{q_a}Q_a, T^*_{q_a}Q_a^{(w)})
	\to
	HF^*_{\mathrm{quad}}(T^*_{q_a} Q_a, T^*_{q_a} Q_a)
	\eqnd
is an isomorphism.
\end{corollary}

\begin{proof}
Observe that $\hom_{\cO_j}(L,(L')^{(w)})$ is naturally isomorphic to $CF^*(L,L';wF)$ when one chooses the cofinal sequence to be given by $wF$. (See Remark~\ref{remark. isomorphism between moving and unmoving complexes} and Remark~\ref{remark. model cofinality linearly}.) One then observes that given two cofinal sequences of wrappings $(L')^{(w)}$ and $(L')^{(v)}$, there is a natural category whose object set is given by the union $\{L^{(w)}\}_{w} \cup \{L^{(v)}\}_{v}$, and a morphism is given by an isotopy class of a non-negative Hamiltonian isotopy between two objects. Each of the original cofinal sequences is cofinal in this category, so the colimits of the Floer cohomology groups agree.
\end{proof}

\subsubsection{Action filtration}

Before proving Proposition~\ref{prop:quad=cofinal}, we set some notation.

\begin{notation}[Action and energy]
The holomorphic strip equation~\eqref{eq:CR-JXH} may be interpreted as a gradient flow equation of the action functional
	\eqn\label{eq:action}
\cA(\gamma) = \cA_{L,L'}(\gamma) = -\int \gamma^*\theta + \int H(\gamma(t))\, dt
+ h_{L'}(\gamma(1)) - h_{L}(\gamma(0))
	\eqnd
where $h_L: L \to \RR$ is a function satisfying $\theta|_L = dh_L$.
We have the basic energy identity
	$$
E_{(J,H)}(u) = \cA(u(\infty)) - \cA(u(-\infty)).
	$$
\end{notation}

\begin{notation}\label{notation. action ell alpha}
Each chain of $CF(L',L;H)$ is a linear combination
	\eqn\label{eq:alpha}
\alpha = \sum a_z\, \langle z\rangle  \in CF(L,L';H), \quad z \in \Chord(L,L';H)
	\eqnd
where each $z$ is a Hamiltonian chord from $L_1$ to $L_0$.
We denote
	$$
\supp \alpha = \{z \in \Chord(L,L';H) \mid a_z \neq 0 \text{ in  \eqref{eq:alpha}}\}
	$$
and define  the action $\ell(\alpha)$ of the cycle  by
	$$
\ell(\alpha) = \max\{\cA(z) \mid z \in \supp \alpha\}.
	$$
\end{notation}

\begin{notation}[Action filtration]
In terms of this action, we have an increasing filtration
	$$
CF^{\leq c}(L,L';H): = \{ \alpha  \subset CF(L,L';H) \mid \ell(\alpha) \leq c\}
	$$
with $CF^{\leq c}(L,L';H) \subset CF^{\leq c'}(L,L';H)$ for $c < c'$.
Each $CF^{\leq c}(L,L';H)$ is a subcomplex of $CF(L,L';H)$.
\footnote{This is because we put the output of the differential at
$\tau = -\infty$, not at $\tau = \infty$.}
By definition, we have
	\eqn\label{eq:filtered}
CF(L,L';H) = \colim_{c} CF^{\leq c}(L,L';H).
	\eqnd
\end{notation}

\subsubsection{Proof of Proposition~\ref{prop:quad=cofinal}}
We start with surjectivity. Let $a \in HF(L,L';H)$
and choose a cycle $\alpha$ representing it, i.e., $\mu_1(\alpha) = 0$.
Denote $c = \ell(\alpha)$ (Notation~\ref{notation. action ell alpha}). Now we consider the following linear adjustment of $H$:
	\eqnn\label{eq:H(wrapper)}
H_{(v)}(x) =
\begin{cases} H(x) \quad & \text{ if } r(x) \leq {v} \\
({v} +1)r(x) -\frac{1}{2}({v}+1)^2 \quad &\text{ if }r(x) \geq {v}+1
\end{cases}
	\eqnd
with suitable smooth interpolation in between. We fix a sequence of integers $\{{v}_k\}_{k\in\NN}$ diverging to $\infty$ and consider the sequence $H_{({v}_k)}$ of linear near infinity Hamiltonians.
We note that $\{H_{({v}_k)}\}_{k \in \NN}$ is a monotone sequence of Hamiltonians
converging to $H$ uniformly on $M^{\text{\rm cpt}} \cup \{r \leq R\}$ for any large $R>0$.

The following lemma will conclude surjectivity of $\varphi_\infty$.

\begin{lemma} There exists $b \in HF(L,L';H_{({v}_k)})$ such that
	$$
a = \varphi_\infty(b).
	$$
\end{lemma}
\begin{proof} Let $\alpha = \sum_{i=1}^k a_i \langle z_i\rangle$ be a cycle representing
$a$ with $z_i \in \Chord(L,L';H)$.

Clearly $z_i \in \Chord(L,L';H_{({v}_k)})$ as long as
${v}_k \ell(\alpha)$ so large that
	$$
\supp \alpha \subset H^{-1}((-\infty,{v}_k])
	$$
where we abuse the notation $\supp \alpha$ by also denoting it
as $\cup_{z \in \supp \alpha} \Image z\subset M$.
Therefore $\alpha$ can be regarded as a chain in $CF(L,L';H_{({v}_k)})$.
Denote the resulting chain of $H_{({v}_k)}$ by $\beta_k$.
We now prove $\beta_k$ is a cycle of $H_{({v}_k)}$, if we choose $k$ even larger
if necessary.

To avoid confusion, we denote by $\mu_1^{H}$ and $\mu_1^{H_{({v}_k)}}$ be the
$\mu_1$-map for $H$ and $H_{({v}_k)}$ respectively. Then standing hypothesis
is $\mu_1^{H}(\alpha) = 0$ and we want to prove
	$$
\mu_1^{H_{({v}_k)}}(\beta_k) = 0
	$$
by choosing a larger $k$ if necessary. By definition, we have
	$$
\mu_1^{H}(\alpha) = \sum_{i=1}^k a_i \mu_1^{H}(\langle z_i\rangle )
	$$
where
	$$
\mu_1^{H}(\langle z_i\rangle )= \sum_{y \in \Chord(L,L';H)} n_{(J,H)}(z_i,y) \langle y \rangle
	$$
with $n_{(J,H)}(z_i,y) = \# \cM(z_i,y;J,H)$ where $\cM(z_i,y;J,H)$ is the moduli space of
solutions $u$ of \eqref{eq:CR-JXH} satisfying $u(-\infty) = z_i, \, u(\infty) = y$.
We rearrange the sum into
	$$
\mu_1^{H}(\alpha) = \sum_{y \in \Chord(L,L';H)} \left(\sum_{i=1}^k a_i n_{(J,H)}(z_i,y)\right) \langle y \rangle.
	$$
Therefore $\mu_1^{H}(\alpha) = 0$ is equivalent to
	$$
\sum_{i=1}^k a_i n_{(J,H)}(z_i,y) = 0
	$$
for all $y \in \Chord(L,L';H)$.

The same formula with $H$ replaced by $H_{({v}_k)}$ holds and so
	$$
\mu_1^{H_{({v}_k)}}(\beta_k) = \sum_{y \in \Chord(L,L';H_{({v}_k)})}
\left(\sum_{i=1}^k a_i n_{(J,H_{({v}_k)})}(z_i,y)\right) \langle y \rangle.
	$$
Therefore it remains to prove
	\eqn\label{eq:vanishing-Hwk}
\sum_{i=1}^k a_i n_{(J,H_{({v}_k)})}(z_i,y) = 0
	\eqnd
for all $y \in \Chord(L,L';H_{({v}_k)})$ by choosing $k$ sufficiently large.
This will follow if we establish
	\eqn\label{eq:MH=MHwk}
\cM(z_i,y;J,H_{({v}_k)}) = \begin{cases} \cM(z_i,y;J,H) \quad & \text{if }  \cM(z_i,y;J,H) \neq \emptyset\\
\emptyset \quad & \text{if }  \cM(z_i,y;J,H) = \emptyset
\end{cases}
	\eqnd
for all $i$ and $y$.

\begin{sublem} There are finitely many $y \in CF(L,L';H)$ such that
$\cM(z_i,y;J,H) \neq \emptyset$ for some $i = 1, \ldots, k$.
\end{sublem}
\begin{proof} By the energy identity, we have
	$$
\cA(y) \leq \ell(\alpha).
	$$
On the other hand, for any Hamiltonian chord $y$ of $H$, we derive
\beastar
\cA(y) & = &-\int y^*\theta + \int_0^1 H(y(t))\, dt\\
& = & \int_0^1 H(y(t)) - \theta(X_H(y(t)))\,dt
 >  \int_0^1 (-C)\, dt = -C
\eeastar
 by Lemma \ref{lem:lowerbound}. Therefore under the given hypothesis, we have
 $$
 - C < \cA(y) \leq \ell(\alpha).
 $$
 By the non-degeneracy assumption on $H$, this finishes the proof.
\end{proof}

Set
	$$
R_0 = \max_y \{r(y) \mid y \text{ is as in the above sublemma}\}.
	$$
Then it follows from the main $C^0$ estimate in \cite{oh-tanaka-liouville-bundles}
that there exists
a sufficiently large $k$ such that
	$$
\max r\circ u \leq R_0 + C'
	$$
for $u \in \cM(z_i,y;J,H)$ where $C'$ depends only on $\inf H> -\infty$.

The same discussion still applies to $H_{({v}_k)}$ since we still have
	$$
H_{({v}_k)} - \theta\left(X_{H_{({v}_k)}}\right) \geq -C
	$$
and
	$$
\max r\circ u \leq R_0 + C'
	$$
for $u \in \cM(z_i,y;J,H_{({v}_k)}$  for the same constant $C, \, C'$ above respectively.
Combining the two, we have proved \eqref{eq:MH=MHwk} and so $\mu_1^{H_{({v}_k)}}(\beta_k) = 0$.

Next we would like to prove
	$$
[\phi_{{v}_k}(\beta_k)] = [\alpha] = a.
	$$
For this we have only to know that
	$$
(1-s) H_{({v}_k)} + sH \equiv H
	$$
on $r^{-1}(-\infty, R_0 +C'])$ and the same $C^0$-estimate as \cite{oh-tanaka-liouville-bundles}
applies for the continuation equation for $\cH = \{H^s = (1-s) H_{({v}_k)} + sH\}$.
This implies that any solution $u$ of continuation equation satisfies
\eqref{eq:CR-JXH}
provided we choose ${v}_k$ sufficiently large so that $H^s \equiv H$ for all $s \in [0,1]$
on $M\setminus \iota([R_0 + C', \infty)$. This in fact implies
	$$
\phi_{{v}_k}(\beta_k) = \alpha
	$$
in chain level and hence proves $[\phi_{{v}_k}(\beta_k)] = [\alpha]$.
\end{proof}
This finishes the proof
of surjectivity.

For the proof of injectivity, let $\beta_k$ be a sequence of $H_{({v}_k)}$-cycle such that
$[\beta_{k+1}] = [\phi_{k(k+1)}(\beta_k)]$ and $[\phi_{{v}_k}(\beta_k)] = 0$ in $HF(L,L';H)$ for all
$k \geq k_1$ with $k_1$ sufficiently large.
 Then
 $$
 \phi_{{v}_{k}}(\beta_{k}) = \mu_1^{H}(\alpha'_{k})
 $$
for some $H$-chain $\alpha'_k$ or each $k \geq k_1$.
Denote $\lambda_2 = \ell(\alpha'_{k_1})$. Under this hypothesis,
by the similar argument given in the surjectivity proof, we can find a sufficiently large
$\ell = \ell(k_1,\lambda_2)> k_1$ such that
	\eqn\label{eq:phibetak1=mu1alpha'}
\phi_{{v}_{k_1}{v}_\ell}(\beta_{k_1}) = \mu_1^{H_{({v}_\ell)}}(\alpha'_{\ell}).
	\eqnd
Now we consider a conformally symplectic dilation $f: M \to M$ defined by
the Liouville flow for time $\log(\rho)$ with $\rho = \frac{{v}_\ell}{{v}_{k_1}}$ which
becomes
	$$
f(x) = (\rho r ,y)
	$$
for $ x = (r,y) \in M^{\text{end};\iota}$. The isotopy $t \mapsto f \circ \phi_{H_{({v}_\ell)}}^t \circ f^{-1}$ is still
a Hamiltonian isotopy generated by the Hamiltonian
	$$
\frac{{v}_{k_1}^2}{{v}_\ell^2} H_{({v}_\ell)}\circ f = : G_{k_1\ell}
	$$
We note that $\frac{{v}_{k_1}}{{v}_\ell} H_{({v}_\ell)}\circ f(r,y) = {v}_{k_1} r$
for any $x = (r,y)$ such that $H_{({v}_\ell)}(r,y) = {v}_\ell r$. Therefore we can find a \emph{chain isomorphism}
	$$
\eta_*: CF(L,L';G_{k_1\ell}) \to CF(L,L'; H_{({v}_{k_1})})
	$$
associated to the isotopy $\eta: s \mapsto \phi_{G_{k_1\ell}}^{1-s} \circ \phi_{H_{({v}_{k_1})}}^s$
which is compactly supported.

We then define a map
	$$
\psi_{k_1\ell}: CF(L,L';H_{({v}_\ell)}) \to CF(L,L';H_{({v}_{k_1})})
	$$
as the composition $\psi_{k_1\ell} = (\eta)_* \circ f_*$ which is a quasi-isomorphism.
Furthermore we also have
\begin{sublem} The map
	$$
\psi_{k_1\ell}\circ \phi_{{v}_{k_1}{v}_\ell} = (\eta)_* \circ f_* \circ \phi_{{v}_{k_1}{v}_\ell}
	$$
is chain homotopic to $\id$ on $CF(L,L';H_{({v}_{k_1})})$.
\end{sublem}
\begin{proof} We have only to notice that the isotopy
	$$
g_s: = \eta^{(1-s)}\circ f^{1-s}\circ  (\phi_{H_{({v}_{k_1})}}^{1-s}\phi_{H_{({v}_\ell)}}^s): M \to M
	$$
and $g_0 = \eta\circ f \circ \phi_{H_{({v}_{k_1})}}^1$ and $g_1 = \id$. In particular, we obtain
	$$
\psi_{k_1\ell}\circ \phi_{{v}_{k_1}{v}_\ell} -\id: CF(L,L';H_{({v}_{k_1})}) \to CF(L,L';H_{({v}_{k_1})})
	$$
is chain homotopic to 0. This finishes the proof.
\end{proof}

Now we apply the map $\psi_{k_1\ell}$
to \eqref{eq:phibetak1=mu1alpha'} and get
	$$
\psi_{k_1\ell}\circ\phi_{{v}_{k_1}{v}_\ell}(\beta_{k_1}) = \psi_{k_1\ell}\circ\mu_1^{H_{({v}_{\ell})}}(\alpha'_{\ell}).
	$$
The left-hand side can be written as
	$$
\psi_{k_1\ell}\circ\phi_{{v}_{k_1}{v}_\ell}(\beta_{k_1}) = \beta_{k_1} + \mu_1^{H_{({v}_{k_1})}}(\gamma)
	$$
for some chain $\gamma$ of $H_{({v}_{k_1})}$ and the right-hand side coincides with
	$$
\mu_1^{H_{({v}_{k_1})}} \circ \psi_{k_1\ell}(\alpha'_{\ell})
	$$
by the chain property of $\psi_{k_1\ell}$. Combining the two, we have derived
	$$
\beta_{k_1}  = \mu_1^{H_{({v}_{k_1})}} \circ \psi_{k_1\ell}(\alpha'_{\ell}) - \mu_1^{H_{({v}_{k_1})}}(\gamma)
=  \mu_1^{H_{({v}_{k_1})}}\left(\psi_{k_1\ell}(\alpha'_{\ell}) - \gamma\right).
	$$
This proves $[\beta_{k_1}] = 0$. By the compatibility of the sequence $[\beta_k]$,
this proves $\lim_k [\beta_k] = 0$ and hence the injectivity of the map $\varphi_\infty$.
This finishes the proof of the proposition.

\subsection{Positive wrappings leave\texorpdfstring{$HF_{\mathrm{quad}}$}{HFquad}  unchanged}
\label{subsec:quad-cofinal}

\begin{notation}[$H \# F$]
Given two (possibly time-dependent) Hamiltonians $H$ and $F$, we define a time-dependent Hamiltonian $H\# F: \RR \times M \to \RR$ by
	\eqnn
	H \# F (t,x) = H(t,x) + w F(\phi_H^t(x)).
	\eqnd
\end{notation}

\begin{lemma}\label{lemma.linear wrapping leaves quadratic wrapping unchanged}
Suppose that $(L')^{(w)}$ is obtained from $L'$ by a non-negative Hamiltonian isotopy, and let $H$ be a quadratic-near-infinity Hamiltonian. Then the continuation map
	\eqnn
	HF(L,L':H) \to HF(L,(L')^{(w)};H)
	\eqnd
is an isomorphism.
\end{lemma}

\begin{proof}
Let $F$ be a linear-near-infinity Hamiltonian inducing a cofinal sequence  of
nonnegative isotopies $(L')^{(v)}$.
Choose a linear-near-infinity Hamiltonian $G$ whose Hamiltonian flow realizes the isotopy
from $L'$ to $(L')^{(w)}$. We have a commutative diagram of Floer cohomology groups
	\eqn\label{eqn. spliced wrapping}
	\xymatrix{
	HF(L,L'; {v} F) \ar[r] \ar[d]
		& \ldots \to HF(L,L'; {v}' F)  \to \ldots  \ar[r] \ar[d]
		& HF(L,L'; H) \ar[d] \\
	HF(L,L';{v}F \# G) \ar[r]
		& \ldots \to HF(L,L';{v}'F \# G) \to \ldots  \ar[r]
		& HF(L,L'; H \# G)
	}
	\eqnd
where arrows are given by continuation maps, and ${v} < {v}'$.
Moreover, the sequences $\{ {v} F\}_{v}$ and $\{{v}F \# G\}_{v} $ are both cofinal in the spliced diagram of Hamiltonians
	\eqnn
	\{ {v}F \}_{v} \cup \{{v} F \# G\}_{v}.
	\eqnd
Thus the colimits
	\eqnn
	\colim_{ {v} \to \infty} HF(L,L';vF),
	\qquad
	\colim_{ {v} \to \infty} HF(L,L';vF\# G)
	\eqnd
are equivalent. On the other hand, the top row and the bottom row of~\eqref{eqn. spliced wrapping} are colimit diagrams by Proposition~\ref{prop:quad=cofinal}. Thus the rightmost vertical arrow of~\eqref{eqn. spliced wrapping} is an isomorphism.

On the other hand, we have the obvious isomorphism
	\eqnn
	HF(L,(L')^{(w)};H) \cong HF(L,L'; H \# G).
	\eqnd
This completes the proof.
\end{proof}

Now suppose that $L=L'$ is a cotangent fiber of $T^*Q$ at a point $a \in Q$, and choose a cofinal sequence for $L$:

\begin{corollary}\label{corollary: quad cofinal}
In the setting of Choice~\ref{choice. for proving cW to cP}, we have
	\eqnn
	\colim_w HF^*_{\mathrm{quad}} (T^*_{q_a}Q_a, T^*_{q_a}Q_a^{(w)})
	\cong HF^*_{\mathrm{quad}}(T^*_{q_a}Q_a, T^*_{q_a}Q_a).
	\eqnd
\end{corollary}

For later reference, we record explicitly the following fact, which follows from the commutativity of~\eqref{eqn. spliced wrapping}:

\begin{lemma}\label{lemma. acceleration and continuation commute}
In the setting of Choice~\ref{choice. for proving cW to cP}, the diagram of cohomology groups
	\eqnn
	\xymatrix{
	H^*\hom_{\cO_j}(T^*_{q_a}Q_a, T^*_{q_a}Q_a^{(w)}) \ar[d]
		\ar[r]
		&	HF^*_{\mathrm{quad}} (T^*_{q_a}Q_a, T^*_{q_a}Q_a^{(w)}) \ar[d] \\
	H^*\hom_{\cO_j}(T^*_{q_a}Q_a, T^*_{q_a}Q_a^{(w+1)})
		\ar[r]
		&	HF^*_{\mathrm{quad}}(T^*_{q_a}Q_a, T^*_{q_a}Q_a^{(w+1)})
	}
	\eqnd
is commutative. Here, all arrows are induced by continuation maps.
\end{lemma}

\subsection{Comparing the non-wrapped Abouzaid map to the quadratically wrapped Abouzaid map}

For this section, we let $|\Delta^n| = |\Delta^0|$, so that the Liouville bundle $E \to |\Delta^0|$ is simply a choice of Liouville sector $M$. In the following lemma, we thus drop the $j$ variable:

\begin{lemma}\label{lemma. abouzaid constructions compatible}
Let $M$ be a Liouville sector and $Q \subset M$ a compact exact brane. Fix objects $X, L^{(w)} \in \ob \cO(M)$.  The Abouzaid functor defines an object of $\Tw C_* \cP(Q)$ by~\eqref{eqn. abouzaid functor object} from any object $L$ of $\cO(M)$, and we abbreviate this object as
	\eqnn
	L \cap Q.
	\eqnd
Then the diagram
	\eqnn
	\xymatrix{
		& H^*\hom_{\Tw C_*\cP(Q)}(Q \cap X, Q \cap L^{(w)}) \\
	H^*\hom_{\cO}(X,L^{(w)}) \ar[r]^{Cor \ref{corollary. filtered to quadratic}} \ar[ur]^{Prop \ref{prop. O to cP natural}}
		& HF^*_{\mathrm{quad}}(X,L^{(w)}) \ar[u]_{\cite{abouzaid-loops}}
	}
	\eqnd
commutes. Here, the horizontal map is a continuation map, while the two other maps are the non-wrapped (Proposition~\ref{prop. O to cP natural}) and wrapped (Section 4 of~\cite{abouzaid-loops}) versions of the Abouzaid map.

Moreover, at the level of cohomology, these diagrams are compatible with the filtered diagram from Lemma~\ref{lemma. filtered Aoo}.
\end{lemma}

\begin{proof}[Proof of Lemma~\ref{lemma. abouzaid constructions compatible}]
Because we are working at the level of cohomology, we may model the isotopy from $L$ to $L^{(w)}$ by some linear-near-infinity Hamiltonian $F$, and choose the quadratic Hamiltonian $H$ in such a way that $H = F$ on the region where $L$ and $L^{(w)}$ intersect.

Choosing a non-negative interpolating isotopy from $F$ to $H$ identifies
	\eqnn
	H^*\hom_{\cO_j}(L,(L')^{(w)}) = HF^*(L,L';F) \to HF^*_{\mathrm{quad}}( L,(L')^{(w)})
	\eqnd
as a subcomplex of chords with action bounded by some $A$.

Both the non-wrapped and wrapped maps count holomorphic triangles one of whose vertices limit to Hamiltonian chords, and we find that the non-wrapped map counts precisely such triangles restricted to the subcomplex of chords with action less than or equal to $A$. This shows that the diagram commutes.

That the diagram is natural in $w$ follows straightforwardly by choosing the quadratic Hamiltonian to be equal to the Hamiltonian defining $L^{(w)}$ and $L^{(w')}$. (Note that while these choices certainly affect the chain-level maps, we may choose these Hamiltonians to have no effect at the level of cohomology.)
\end{proof}

Now we set $M = T^*Q = T^*Q_a$ and set $L = L'$ to be a cotangent fiber at some point $q_a \subset Q_a$.

\begin{corollary}\label{cor. abouzaid constructions compatible cotangent}
In the setting of Choice~\ref{choice. for proving cW to cP}, for every $w$, the diagram
	\eqnn
	\xymatrix{
		&& H^*\hom_{C_*\cP_j}(q_a, q_a^{(w)}) \\
	H^*\hom_{\cO_j}(T^*_{q_a}Q_a,T^*_{q_a}Q_a^{(w)}) \ar[rr]^-{\text{Cor \ref{corollary. filtered to quadratic}}}
\ar[urr]^{\text{Prop \ref{prop. O to cP natural}}}
		&& HF^*_{\mathrm{quad}}(T^*_{q_a}Q_a,T^*_{q_a}Q_a^{(w)}) \ar[u]^{\cite{abouzaid-loops}}
	}
	\eqnd
commutes. Here, the horizontal map is a continuation map, while the two other maps are the non-wrapped (Proposition~\ref{prop. O to cP natural}) and wrapped~\cite{abouzaid-loops} versions of the Abouzaid map. Finally, $q_a^{(w)}$ is the intersection point of $T^*_{q_a}Q_a^{(w)}$ with the zero section. (We have chosen our cofinal wrapping so that there is only one intersection point.)

Moreover, at the level of cohomology, these diagrams are compatible with the filtered diagram from Lemma~\ref{lemma. filtered Aoo}.
\end{corollary}

\subsection{Proof of Theorem~\ref{theorem. cW to cP}}\label{section. proof of cW cP theorem}

\begin{proof}[Proof of Theorem~\ref{theorem. cW to cP}.]
Because the cotangent fiber $T^*_{q_a} Q_a$ and the object $q_a \in Q_a$ generate\footnote{For $\Tw C_* \cP_j$ this is obvious, while for $\Tw \cW_j$, this follows from Abouzaid's Theorem and Proposition~\ref{prop. local triviality}.} the $A_\infty$-categories in which they reside, all that remains is to show that the map on endomorphism complexes of these objects---induced by Corollary~\ref{cor. natural transformation}---is a quasi-isomorphism.

We first claim the diagram below commutes:
	\eqn\label{eqn. quadratic vs localization square}
	\xymatrix{
	H^*\hom_{\cW_j}(T^*_{q_a} Q_a, T^*_{q_a}Q_a)
		\ar[rr]^-{\text{Cor}~\ref{cor. natural transformation}}
		&& H^*\hom_{C_* \cP_j}(q_a, q_a) \\
	\colim_w H^* \hom_{\cO_j}(T^*_{q_a}Q_a,T^*_{q_a}Q_a^{(w)})
		\ar[rr]^-{\text{Cor}~\ref{corollary. filtered to quadratic}}_-{\cong}	
		\ar[u]^-{\text{Lem}~\ref{lemma. hom is wrapped cohomology}}_-{\cong}
		\ar[urr]^-{\text{Prop}~\ref{prop. O to cP natural}}
		&& HF^*_{\mathrm{quad}}(T^*_{q_a} Q_a, T^*_{q_a} Q_a) \ar[u]^{\cong}_{\cite{abouzaid-loops}}
	}
	\eqnd
Let us explain the diagram.

We first note that, for a fixed $w$, there is a homotopy commutative diagram of chain complexes
	\eqnn
	\xymatrix{
		\hom_{\cW_j}(T^*_{q_a}Q_a,T^*_{q_a}Q_a^{(w)}) \ar@{-->}[r]
			&	\hom_{C_*\cP_j}(q_a, q_a^{(w)}) \\
		\hom_{\cO_j}(T^*_{q_a}Q_a,T^*_{q_a}Q_a^{(w)}) \ar[ur] \ar[u]
	}
	\eqnd
where the diagonal map is the Abouzaid construction for $\cO$ (Proposition~\ref{prop. O to cP natural}), and the upper, dashed, horizontal map is induced by the universal property of localization (Corollary~\ref{cor. natural transformation}). Here, $q_a^{(w)}$ is the intersection point of $T^*_{q_a}Q_a^{(w)}$ with the zero section. (We have chosen our cofinal wrapping so that there is only one intersection point.)

Moreover, there is a homotopy coherent functor $\ZZ_{\geq0} \to \cO_j$ by Lemma~\ref{lemma. filtered Aoo}, so the lower left corner of the triangle coheres into a homotopy-coherent sequential diagram indexed by $w$. Since we have functors $\cO_j \to \cW_j$ and $\cO_j \to C^*\cP_j$, we have an induced homotopy-coherent diagram of the colimits:
	\eqnn
	\xymatrix{
		\colim_w \hom_{\cW_j}(T^*_{q_a}Q_a,T^*_{q_a}Q_a^{(w)}) \ar[r]
			&	\colim_w \hom_{C_*\cP_j}(q_a, q_a^{(w)}) \\
		\colim_w \hom_{\cO_j}(T^*_{q_a}Q_a,T^*_{q_a}Q_a^{(w)}) \ar[ur] \ar[u]
	}
	\eqnd
We note that the two sequential colimits in the top horizontal line is a colimit of isomorphisms upon passage to cohomology---this is because continuation maps are sent to equivalences in $\cW_j$ (by definition of localization) and in $C_*\cP_j$ (by Theorem~\ref{thm. continuation maps are equivalences of twisted complexes}). Thus, for both items in the top horizontal line, the cohomology of the colimit is isomorphic to the cohomology of the $w=0$ term. This explains the upper-left triangle in~\eqref{eqn. quadratic vs localization square}.

The lower-right triangle of~\eqref{eqn. quadratic vs localization square} is obtained by applying the $w$-indexed colimit (at the level of cohomology) to the triangle in Corollary~\ref{cor. abouzaid constructions compatible cotangent}. We observe that the filtered colimits on the right vertical edge consists of maps that are all equivalences, so in this way we may identify
	\eqnn
	\colim_w H^*\hom_{C_*\cP_j}(q_a, q_a^{(w)})
	\cong
	H^* \hom_{C_*\cP_j}(q_a,q_a),
	\eqnd
and the lower-right corner of the triangle arises by using Corollary~\ref{corollary: quad cofinal}.
This completes the explanation of the commutative diagram~\eqref{eqn. quadratic vs localization square}.

Referring again to~\eqref{eqn. quadratic vs localization square}, note that the left-hand vertical arrow was verified to be an isomorphism at the level of cohomology in Lemma~\ref{lemma. hom is wrapped cohomology}, the bottom horizontal arrow was verified to be an isomorphism in Corollary~\ref{corollary. filtered to quadratic}, and the right-hand vertical arrow is an isomorphism by~\cite{abouzaid-cotangent}. Because these $\cong$-labeled arrows are isomorphisms, it follows that the top horizontal arrow is also an isomorphism, which is what we sought to prove.
\end{proof}

\clearpage
\section{The diffeomorphism action on Loc}

\begin{notation}[$\diff$]
Fix $Q$ an oriented, compact manifold. We let $\diff(Q)$ denote the topological group of orientation-preserving diffeomorphisms of $Q$.
\end{notation}

The goal of this section is to prove that the natural action of $\diff(Q)$ on $\loc(Q)$ is compatible with the action of $\Liouautgrb(T^*Q)$ on $\cW(T^*Q)$ from Theorem~\ref{theorem. main theorem informal}.

Given the results of our previous section, the only thing left to do is to verify that the diffeomorphism action on $\Tw C_*\cP$ is the standard action of the diffeomorphism group on the $A_\infty$-category of local systems. This is proven in Proposition~\ref{prop. P to Loc}.

\subsection{\texorpdfstring{$C_*\cP$}{C.P} is compatible with the diffeomorphism action}

Our eventual goal is to prove that the (orientation-preserving) diffeomorphism group action on the $\infty$-category of local systems is compatible with its action on the wrapped Fukaya category of a cotangent bundle. So first let us show that our construction $C_* \cP$ encodes the usual action on the $\infty$-category of local systems.

For this, recall that the $\infty$-category of local systems on
a space $B$ with values in an $\infty$-category $\cD$ is equivalent to the $\infty$-category
	\eqnn
	\fun(\sing(B), \cD)
	\eqnd
of functors from $\sing(B)$ to $\cD$. The evident action of $\haut(B)$ on $B$---and hence on $\sing(B)$---exhibits the action of $\haut(B)$ on the $\infty$-category of local systems.

On the other hand, our construction of $C_* \cP$ passes through a combinatorial trick that replaces $B\diff(Q)$ by a category of simplices in $B\diff(Q)$, which one can informally think of as the category encoding the barycentric subdivision of $\sing(B\diff(Q))$. We must show that this combinatorial trick allows us to recover the natural action of $\diff(Q)$ on $Q$; this is the content of Corollary~\ref{cor. p inverse classifies diff Q action} below.

\begin{construction}\label{construction. p inverse}
Let $p: E \to B$ be a Kan fibration of simplicial sets. We let $\subdivision(B)$ denote the subdivision simplicial set associated to $B$ (Recollection~\ref{recollection. smooth approximation}\eqref{item. localization of subdiv}).

We have an induced functor
	\eqnn
	p^{-1}:
	\subdivision(B) \to \Kan,
	\qquad
	(j: \Delta^k \to B)
	\mapsto
	p^{-1}(j)
	\eqnd
to the $\infty$-category of Kan complexes. Indeed, realizing $\subdivision(B)$ to be the nerve of a category, the above is induced by an actual functor to the category of simplicial sets, sending an object $j$ to the simplicial set $j^*E$.
\end{construction}

\begin{remark}\label{remark. p inverse induces functor}
Moreover, because $p$ is a Kan fibration, every edge in $\subdivision(B)$ is sent to an equivalence in $\Kan$; thus the above functor factors through the localization of $\subdivision(B)$. Moreover, we know the localization to be equivalent as an $\infty$-category to the Kan complex $B$ (Recollection~\ref{recollection. smooth approximation}\eqref{item. localization of subdiv}). We draw this factorization as follows:
	\eqnn
	\xymatrix{
	\subdivision(B) \ar[rr]^{p^{-1}} \ar[dr]&& \Kan \\
	& B \ar[ur]_F
	}
	\eqnd
That is, $p^{-1}$ induces some functor $F: B \to \Kan$ of $\infty$-categories.
\end{remark}

On the other hand, the Kan fibration $p: E \to B$ classifies a functor of $\infty$-categories from $B$ to $\Kan$ by the straightening/unstraightening correspondence. (See 3.2 of~\cite{htt}.) Our main goal is to prove:

\begin{lemma}\label{lemma. p inverse classifies p}
The functor classified by $p: E \to B$ admits a natural equivalence to the functor $F$ (induced by $p^{-1}$ in Construction~\ref{construction. p inverse}).
\end{lemma}

Given the lemma, we have

\begin{cor}\label{cor. p inverse classifies diff Q action}
Let $EQ$ be the tautological $Q$ bundle over $B\diff(Q)$. We then have a Kan fibration $p: \sing(EQ) \to \sing(B\diff(Q))$, and the induced functor
	\eqnn
	N(\simp(B\diff(Q))) \to \Kan,
	\qquad
	(j: |\Delta^a| \to B)
	\mapsto
	\sing(j^* EQ).
	\eqnd
Consider the induced functor $F$ from Remark~\ref{remark. p inverse induces functor}:
	\eqnn
	\xymatrix{
	N(\simp(B\diff(Q))) \ar[rr] \ar[dr] && \Kan \\
	& \sing(B\diff(Q)) \ar[ur]_{F} .
	}
	\eqnd
Then $F$ is naturally equivalent to the functor sending a distinguished vertex of $\sing(B\diff(Q))$ to $\sing(Q)$, and exhibiting the action of $\sing(\diff(Q))$ on $\sing(Q)$.
\end{cor}

\begin{proof}[Proof of Corollary~\ref{cor. p inverse classifies diff Q action}.]
The fibration $EQ \to B\diff(Q)$ classifies the functor $\sing(B\diff(Q)) \to \Kan$ exhibiting the $\diff(Q)$ action on $Q$. Now apply Lemma~\ref{lemma. p inverse classifies p}.
\end{proof}

We need to recall a few tools before proving the lemma.

\begin{recollection}[Relative nerve]\label{recollection. relative nerve}
Fix a functor $f:\cC \to \sset$ from a category $\cC$ to the category of simplicial sets. Then one can construct a coCartesian fibration $N_f(\cC) \to N(\cC)$ called the {\em relative nerve} of $f$. (See Section~3.2.5 of~\cite{htt}.)

$N_f(\cC)$ is a simplicial set defined as follows: For any finite, non-empty linear order $I$, an element of $N_f(\cC)(I)$ is given by the data of:
\begin{itemize}
	\item A simplex $\phi: \Delta^I \to \cC$ (where $\Delta^I \cong \Delta^n$ for $n = |I| -1 $), and
	\item For every subset $I' \subset I$, setting $i' = \max I'$, a simplex
		\eqnn
		\tau_{I'}: \Delta^{I'} \to f(\phi(i'))
		\eqnd.
\end{itemize}
These data must satisfying the following condition:
\begin{itemize}
	\item For any $I' \subset I''$, the diagram of simplicial sets
		\eqnn
		\xymatrix{
		\Delta^{I'} \ar[r]^{\tau_{I'}} \ar[d]
			& f(\phi(i')) \ar[d] \\
		\Delta^{I''} \ar[r]^{\tau_{I''}}
			& f(\phi(i''))
		}
		\eqnd
	must commute.
\end{itemize}
\end{recollection}

\begin{example}\label{example. relative nerve of p inverse}
Let us parse what the relative nerve $N_{p^{-1}}(\subdivision(B))$ is, where $p^{-1}$ is the functor from Construction~\ref{construction. p inverse}. An $n$-simplex in  $N_{p^{-1}}(\subdivision(B))$ is the data of
\begin{itemize}
	\item A collection of inclusions of simplices
		\eqnn
		\Delta^{a_0} \into \Delta^{a_1} \into \ldots \into \Delta^{a_n}
		\eqnd
	together with a map $j: \Delta^{a_n} \to B$, and
	\item A map $\tau: \Delta^n \to j^* E$, such that
	\item For every $i \in \{0, \ldots, n\}$, the $i$th vertex of $\tau_{[n]}$ must be a vertex in $j^*E|_{\Delta^{a_i}}$.
\end{itemize}
Because $p: E \to B$ is assumed to be a Kan fibration, it follows that the forgetful map $N_f(\subdivision(B)) \to \subdivision(B)$ is a coCartesian fibration (in fact, a left fibration)--see Proposition~3.2.5.21 of~\cite{htt}.
\end{example}

\begin{remark}\label{remark. relative nerve classifies functor}
Moreover, it is proven in~\cite{htt} that the fibration $N_f(\cC)$, when straightened, classifies a functor naturally equivalent to $f$. See again Proposition~3.2.5.21 of~\cite{htt}.
\end{remark}

\begin{notation}\label{notation. max and max^*E}
On the other hand, note that we have a natural map
	\eqnn
	\max: \subdivision(B) \to B
	\eqnd
for any simplicial set $B$. On vertices, it sends $j: \Delta^a \to B$ to the vertex $j(\max [a]) \in B_0$. This induces the obvious map on higher simplices.

Thus we have, for any Kan fibration $E \to B$, the pulled back Kan fibration $\max^*E \to \subdivision(B)$ as follows:
	\eqnn
	\xymatrix{
	\max^* E \ar[r] \ar[d] & E \ar[d]^p \\
	\subdivision{B} \ar[r]^{\max} & B
	}
	\eqnd
\end{notation}

\begin{proof}[Proof of Lemma~\ref{lemma. p inverse classifies p} .]
We have the map of coCartesian fibrations
	\eqnn
	\max^*E \to N_{p^{-1}}(\subdivision(B))
	\eqnd
which, on the fiber above $j: \Delta^a \to B$, includes the simplicial set of all maps whose $\tau$ lands in the fiber above $j(\max[a])$. This is obviously a weak homotopy equivalence along the fibers because $E \to B$ is a Kan fibration. Thus we have a diagram
	\eqnn
	\xymatrix{
	N_{p^{-1}}(\subdivision(B)) \ar[d]
	& \max^*E \ar[r] \ar[l] \ar[d]
	& E \ar[d]^p \\
	\subdivision(B) \ar@{=}[r]
	& \subdivision(B) \ar[r]^-{\max}
	& B
	}
	\eqnd
where each horizontal arrow is a weak homotopy equivalence (i.e., an equivalence in the model structure for Kan complexes). Thus every square in this diagram---upon passage to Kan complexes, i.e., their localizations---exhibits an equivalence of Kan fibrations. Because the straightening/unstraightening construction sends equivalences of fibrations to natural equivalences of functors, the result follows.
\end{proof}

Now we make use of the Quillen adjunction employing Lurie's dg nerve construction. See also~\cite{brav-dyckerhoff}.

\begin{notation}[Nerve and its adjoint]\label{notation. dg nerve and adjoint}
Fix a base ring $R$.
Consider the adjunction
	\eqnn
	R[-]: \sset \iff  \dgcatt : N_{dg}.
	\eqnd
Here, $N_{dg}$ is Lurie's dg nerve. (See Construction~1.3.1.6 of~\cite{higher-algebra}.) It is a functor sending any dg category to an $\infty$-category, and any dg-functor to a map of simplicial sets. We denote its left adjoint by $R[-]$.
\end{notation}

\begin{remark}
The adjunction of Notation~\ref{notation. dg nerve and adjoint} can be promoted to a Quillen adjunction.
The Quillen adjunction is with respect to the Joyal model structure for simplicial sets, and the Tabuada model structure for dg-categories.

Note that these two model structures have simplicial localizations equivalent to $\inftycat$ and $\Ainftycat$, respectively. For the fact that the model structure on dg-categories recovers the $\infty$-category of $A_\infty$-categories, see Section~1.2 of~\cite{tanaka-Aoo-units}.
\end{remark}

\begin{remark}
When $\cC$ is an $\infty$-groupoid---for example, $\sing(B)$ for some topological space $B$---then $R[\cC]$ is equivalent to the dg-category $C_* \cP$.
\end{remark}

\begin{notation}[$R_Q$]\label{notation. RQ}
Now let $R_Q$ denote the following composition:
	\eqnn
	N(\simp(B\diff(Q))
		\xrightarrow{p^{-1}}
		\Kan
		\into
		\sset
		\xrightarrow{R[-]}
		\Ainftycat.
	\eqnd
\end{notation}

\begin{prop}\label{prop. P to Loc}
There exists a natural equivalence
	\eqnn
        \xymatrix@C+2pc{
		N(\simp(B\diff(Q)) \rrtwocell^{C_* \cP}_{R_Q}{\;\;\;} && \Ainftycat.
        }	
	\eqnd
Here, $R_Q$ is the functor from Notation~\ref{notation. RQ} and $C_* \cP$ is the functor from Notation~\ref{notation. Tw cP}.
\end{prop}

\begin{proof}
For every $j: |\Delta^n| \to B\diff(Q)$, let $D(j) \subset p^{-1}(j)$ denote the full subcategory spanned by those 0-simplices of $j^*E$ that are contained in a fiber above one of the vertices of $|\Delta^n|$. Then the natural transformation induced by the inclusion $D(j) \to p^{-1}(j)$ is essentially surjective and obviously fully faithful. On the other hand, $D(j)$ is equivalent to the path category $\cP_j$. Thus the composite natural equivalences
	\eqnn
	p^{-1}(j) \leftarrow D \to \cP
	\eqnd
exhibits the natural equivalence we seek by choosing an inverse to either equivalence.
\end{proof}

\begin{remark}\label{remark. F is diff action}
The above proposition accomplishes the goal of seeing that $C_*\cP$ exhibits the action of $\diff(Q)$ on the $\infty$-category of local systems. To see this, consider the composite

\begin{align}
	N(\simp(B\diff(Q))
		& \xrightarrow{\max} \sing^{C^\infty} (\widehat{B\diff(Q)}) \nonumber \\
		& \xrightarrow{\sim} \sing  ( B\diff(Q) ) \nonumber \\
		& \xrightarrow{\sim} \BB \sing  ( \diff(Q) ) \nonumber \\
		& \xrightarrow{\iota} \Kan \nonumber \\
		& \xrightarrow{R[-]} \Ainftycat \label{eqn. p inverse composite}.
\end{align}
Here,
	\begin{itemize}
	\item $\max$ is the natural map from a subdivision to the underlying simplicial set; the next arrow is the natural map from smooth, extended simplices to continuous simplices.
 	\item The next arrow identifies the singular complex of the classifying space $B\diff(Q)$ with the $\infty$-category with one object, whose endomorphism space is given by $\sing(\diff(Q))$. The arrow $\iota$ is the natural inclusion of this subcategory into $\Kan$---the unique object of $\BB (\sing(\diff(Q)))$ is sent to the Kan complex $\sing(Q)$, and we have the obvious map on morphisms. Finally, $R[-]$ is the left adjoint to the dg-nerve from Notation~\ref{notation. dg nerve and adjoint}.
\end{itemize}

By the universal property of localization, $R_Q$ factors as in the below diagram:
	\eqnn
	\xymatrix{
	N(\simp(B\diff(Q)) \ar[dr] \ar[rr]^{R_Q}
		&& \Ainftycat \\
	& \sing(B\diff(Q)) \ar@{-->}[ur]
	}	
	\eqnd
We know from Lemma~\ref{lemma. p inverse classifies p} that the dashed arrow in the diagram is equivalent to our composite~\eqref{eqn. p inverse composite}.
Thus, by the natural equivalence of Proposition~\ref{prop. P to Loc}, we conclude that the functor $\sing(B\diff(Q)) \to \Ainftycat$ induced by $C_*\cP$ is also equivalent to the composite map of~\eqref{eqn. p inverse composite}. This was our goal.
\end{remark}

\subsection{Proof of Theorem~\ref{theorem. diff and ham compatible}}

\begin{proof}[Proof of Theorem~\ref{theorem. diff and ham compatible}.]
Among functors from $N \simp(B\diff(Q))$ to $\Ainftycat$, we have the following natural equivalences:
	\eqnn
	\Tw R_Q
	\xrightarrow{\text{Prop~\ref{prop. P to Loc} }} \Tw C_* \cP
	\xrightarrow{\text{Thm~\ref{theorem. cW to cP} }} \Tw \cW \circ \DD.
	\eqnd
Because all three of the above functors---$\Tw R_Q$, $\Tw C_*\cP$, and $\Tw \cW \circ \DD$---map morphisms in $N \simp(B\diff(Q))$ to equivalences in $\Ainftycat$, each induces a functor from the Kan completion of $N \simp(B\diff(Q))$. This Kan completion is an $\infty$-groupoid equivalent to $\sing(B\diff(Q))$, and hence to $\BB \diff(Q)$, by Recollection~\ref{recollection. smooth approximation}\eqref{item. localization of BLiou}.

By Remark~\ref{remark. F is diff action}, the functor induced by $\Tw R_Q$ classifies the $\diff(Q)$ action on $C_*\cP(Q)$. On the other hand, $\Tw \cW \circ \DD$ by construction classifies the action of $\diff(Q)$ on $\cW(M)$ induced by the action of $\Liouaut$ on $\cW(M)$. This completes the proof.
\end{proof}

\bibliographystyle{amsalpha}
\bibliography{biblio}

\end{document}